\documentclass{amsart}
\usepackage{amssymb}
\usepackage{mathtools}
\usepackage[only=llbracket,rrbracket,arrownot]{stmaryrd}
\usepackage[svgnames]{xcolor}
\usepackage[unicode,
  colorlinks=true,
  linktocpage=true,
  citecolor=OliveDrab,
  linkcolor=DarkMagenta,
  linkcolor=DarkMagenta,
  pdftitle={Higher Segal spaces and partial groups},
  pdfauthor={Philip Hackney and Justin Lynd},
  pdfkeywords={partial group, higher Segal space, root system, group action, closure system, symmetric set}
  ]{hyperref}
\usepackage[capitalise,noabbrev]{cleveref}
\usepackage{mathrsfs}
\usepackage{tikz-cd}
\usepackage{enumitem}
\usepackage{microtype}
\usepackage{dynkin-diagrams}
\usepackage[T1]{fontenc}

\usepackage{booktabs}

\DeclareMathOperator{\starset}{st}
\newcommand{\op}{\textup{op}}
\DeclareMathOperator{\id}{id}

\newcommand{\bcom}{B_{\textup{com}}}

\newcommand{\udec}{\textup{dec}_{\top}}
\newcommand{\ldec}{\textup{dec}_{\perp}}
\newcommand{\tw}{\operatorname{tw}}
\newcommand{\conj}{\operatorname{cj}}

\newcommand{\sset}{\mathsf{sSet}}
\newcommand{\sym}{\mathsf{Sym}}
\newcommand{\fun}{\mathsf{Fun}}
\newcommand{\set}{\mathsf{Set}}
\newcommand{\cat}{\mathsf{Cat}}
\newcommand{\gpd}{\mathsf{Gpd}}
\newcommand{\pgpd}{\mathsf{pGpd}}
\newcommand{\grp}{\mathsf{Grp}}
\DeclareMathOperator{\ob}{ob}

\newcommand{\ord}{\mathbf{\Delta}}
\newcommand{\fin}{\mathbf{\Upsilon}}
\newcommand{\ordalt}{\widetilde{\ord}}

\newcommand{\delint}{\ord_{\textup{int}}}

\newcommand{\finbot}{\fin_{\!\textup{z}}} 

\newcommand{\domraw}{\operatorname{D}}
\newcommand{\dom}[1]{\domraw(#1)}
\newcommand{\invset}{\operatorname{N}}

\newcommand{\edgemap}{\mathscr{E}}
\newcommand{\bous}{\mathscr{B}}
\DeclareMathOperator{\mat}{Mat}

\newcommand{\ps}{\mathcal{P}} 

\newcommand{\subsets}{\mathscr{S}}
\newcommand{\closeds}{\mathscr{C}}

\newcommand{\words}{\mathbf{W}}
\newcommand{\bswords}{\mathbf{S}}

\newcommand{\cube}[1]{\llbracket #1 \rrbracket}
\newcommand{\intcu}[1]{\cube{#1}}
\newcommand{\intcube}[2]{\cube{#1 \subset #2}}

\DeclareMathOperator{\core}{core}
\newcommand{\ul}{\underline}

\DeclareMathOperator{\cl}{cl}	

\newcommand{\rep}[1]{\Upsilon^{#1}}
\newcommand{\spine}[1]{\textup{Sp}^{#1}}
\newcommand{\NA}{\operatorname{NA}}
\DeclareMathOperator{\sk}{sk}
\DeclareMathOperator{\nondeg}{nd}
\newcommand{\name}[1]{\ulcorner \!#1\! \urcorner}

\DeclareMathOperator\image{im}

\newcommand{\trans}[2]{#1/\!\!/#2} 
\newcommand{\loc}[2]{L_{#1}(#2)} 
\newcommand{\pto}{\nrightarrow}

\newcommand{\rsys}{\Phi}
\newcommand{\pos}{\rsys^+}
\newcommand{\base}{\Pi} 
\newcommand{\cone}{\operatorname{cone}}
\newcommand{\conv}{\operatorname{conv}}
\newcommand{\reals}{\mathbb{R}}
\newcommand{\realspan}{\operatorname{span}_\reals}
\newcommand{\ints}{\mathbb{Z}}
\newcommand{\setmin}{\backslash}
\newcommand{\Sp}{\operatorname{Sp}}

\newcommand{\ko}{\textcolor{blue}}

\newcommand{\sepin}{\mathsf{G}}
\newcommand{\sepun}{\mathsf{F}}
\newcommand{\totin}{\mathsf{I}}

\newcommand{\tallmid}{\mathrel{\Big|}}

\newcommand{\barcl}{\overline{\cl}}
\newcommand{\bardom}{\overline{\domraw}}

\usepackage{manfnt}
\newcommand{\ncube}{\text{\mancube}}

\newtheorem{theorem}{Theorem}[section]
\newtheorem{proposition}[theorem]{Proposition}
\newtheorem{corollary}[theorem]{Corollary}
\newtheorem{lemma}[theorem]{Lemma}

\theoremstyle{definition}
\newtheorem{definition}[theorem]{Definition}
\newtheorem{convention}[theorem]{Convention}
\newtheorem{example}[theorem]{Example}
\newtheorem{notation}[theorem]{Notation}
\newtheorem{construction}[theorem]{Construction}

\theoremstyle{remark}
\newtheorem{remark}[theorem]{Remark}

\usetikzlibrary{decorations.markings}
\definecolor{lightpurple}{HTML}{c8c0fb}
\definecolor{mediumpurple}{HTML}{9080fb}
\definecolor{bblue}{HTML}{c5d3fb}
\definecolor{lblue}{HTML}{e3eafd}
\definecolor{lightpink}{HTML}{ff98f6}
\definecolor{verylightgray}{rgb}{0.95,0.95,0.95}

\begin{document}

\title{Higher Segal spaces and partial groups}

\author{Philip Hackney}
\address{Department of Mathematics, University of Louisiana at Lafayette}
\email{philip@phck.net} 
\urladdr{http://phck.net}

\author{Justin Lynd}
\address{Department of Mathematics, University of Louisiana at Lafayette}
\email{lynd@louisiana.edu}

\thanks{This work was supported by a grant from the Simons Foundation (\#850849, PH). PH was partially supported Louisiana Board of Regents through the Board of Regents Support fund LEQSF(2024-27)-RD-A-31. JL was partially supported by NSF Grant DMS-1902152.}

\date{September 4, 2025}

\thanks{Results from this paper were announced in \cite{Hackney:MFO,Lynd:MFO}.}



\begin{abstract}
The $d$-Segal conditions of Dyckerhoff and Kapranov are exactness properties for simplicial objects based on the geometry of cyclic polytopes in $d$-dimensional Euclidean space.
2-Segal spaces are also known as decomposition spaces, and most activity has focused on this case.
We study the interplay of these conditions with the partial groups of Chermak, a class of symmetric simplicial sets. 
The $d$-Segal conditions simplify for symmetric simplicial objects, and take a particularly explicit form for partial groups.
We show partial groups provide a rich class of $d$-Segal sets for $d > 2$, by undertaking a systematic study of the \emph{degree} of a partial group $X$, namely the smallest $k\geq 1$ such that $X$ is $2k$-Segal. 
We develop effective tools to explicitly compute the degree based on the discrete geometry of actions of partial groups, which we define and study.
Applying these tools involves solving Helly-type problems for abstract closure spaces.
We carry out degree computations in concrete settings, including for the punctured Weyl groups introduced here, where we find that the degree is closely related to the maximal dimension of an abelian subalgebra of the associated semisimple Lie algebra. 
\end{abstract}

\maketitle

\section{Introduction}

The classical Segal condition for a simplicial set, namely that the map 
\[
\mathscr{E}_n\colon X_n \to X_1 \times_{X_0} \cdots \times_{X_0} X_1
\]
sending an $n$-simplex to its spine is a bijection for all $n \geq 2$, characterizes which simplicial sets arise as nerves of categories.
It admits two recent generalizations of particular interest in the higher Segal spaces of Dyckerhoff and Kapranov \cite{DyckerhoffKapranov} and the partial groups of Chermak \cite{Chermak:FSL}.

The \emph{higher Segal conditions} are a family of conditions on simplicial objects based on triangulations of cyclic polytopes. 
They play a key role in the higher Dold--Kan correspondence \cite[4.27]{DyckerhoffJassoWalde} and have applications to higher algebraic K-theory \cite{poguntke}; see \cite{Dyckerhoff:CPOC} for a recent survey.
The conditions come in upper and lower variants, and are  progressively weaker as $d$ grows: a lower or upper $d$-Segal object is both lower and upper ($d{+}1$)-Segal. 
When $d = 1,2$ they may be interpreted as associativity conditions \cite{GKT1,Penney2017,Stern:2SOAS}, while for $d>2$ they represent measures of higher associativity \cite{GalGal:HSSLAA,Dyckerhoff:CPOC}.

The lower $1$-Segal condition is the ordinary Segal condition characterizing the essential image of the nerve. 
$2$-Segal simplicial objects have been studied intensely in recent years, often with the goal of categorifying and generalizing constructions of various types of associative algebras in representation theory and objective combinatorics, such as Hall and Hecke algebras, incidence algebras, coalgebras of rooted trees, and the like. 
They were introduced independently under the terminology \emph{decomposition spaces} by G\'alvez-Carrillo, Kock, and Tonks \cite{GKT1}. 
Roughly speaking, whereas the Segal condition implies that the span 
\[
X_{1} \times_{X_0} X_1 \xleftarrow{\mathscr{E}_2} X_2 \xrightarrow{d_1} X_1
\]
determines an associative, totally- and uniquely-defined composition law (the composition of the associated category), the $2$-Segal conditions enforce associativity for the same composition, which now may be partially defined and multiply valued. 
For recent introductions from the $2$-Segal space and decomposition space perspectives, respectively, 
see \cite{Stern:2SSP} and \cite{Hackney:DSP}. 

The second generalization relaxes the Segal condition on a simplicial set to require merely that Segal maps be \emph{injections} for $n\geq 2$. 
(The above composition will be then partially defined, but unique when it exists.)
We are interested in this condition chiefly for \emph{symmetric simplicial sets}, presheaves on the symmetric simplex category $\fin \supset \ord$ having the same objects $[n] = \{0,1,\dots,n\}$ as $\ord$ but with morphisms $[m] \to [n]$ all functions. 
A symmetric set is called \emph{spiny} if it satisfies this generalization of the Segal condition. 
Building on a theorem of Gonz\'alez \cite{gonzalez}, we showed in \cite{HackneyLynd:PGSSS} that the category of reduced spiny symmetric sets is equivalent to the category of \emph{partial groups} in the sense of Chermak \cite{Chermak:FSL}. 
Accordingly, a \emph{partial groupoid} is just a spiny symmetric set. 
In a partial groupoid an $n$-simplex can be written unambiguously in familiar bar notation $[f_1|\cdots|f_n]$ like in the nerve of a category or groupoid. 
But unlike in the nerve of a category, the total composition $[f_1|\cdots|f_n] \mapsto f_n \circ \cdots \circ f_1$ will only be defined for some tuples $(f_1,\dots,f_n)$ with edges agreeing successively target to source. 

An important class of examples of partial groups arises from partial actions of groups.
Given a partial action $G \times S \pto S$ of a group $G$ on the set $S$ in the sense of Exel \cite{Exel:PAGAIS}, there is a transporter groupoid $\trans{S}{G}$, having object set $S$ and morphisms $s \xrightarrow{g} g \cdot s$ whenever $g$ acts on $s$, as well as a functor $\trans{S}{G} \to G$. 
The associated partial group is the image $\loc{S}{G}$ the corresponding map $N(\trans{S}{G}) \twoheadrightarrow \loc{S}{G} \subseteq BG$ on nerves. 
The $n$-simplices of $\loc{S}{G}$ (the multipliable words of length $n$) are those tuples $[g_1|\cdots|g_n] \in BG_n$ of elements of $G$ that act successively on some element of $S$, that is, for which there is $s \in S$ such that $g_1 \cdot s$ is defined, $g_2\cdot (g_1 \cdot s)$ is defined, and so forth. 

The initial achievement of Chermak's theory of partial groups was to establish the existence and uniqueness of centric linking systems for saturated fusion systems on finite $p$-groups \cite{Chermak:FSL}. 
This was an important generalization of the Martino--Priddy conjecture (first proved by Oliver): two finite groups have homotopy equivalent Bousfield--Kan $p$-completed classifying spaces if and only if there is an isomorphism between their Sylow $p$-subgroups intertwining the conjugation maps between $p$-subgroups in the two groups \cite{Oliver:ECSPO,Oliver:ECSP2}. 
Linking systems are special types of partial groups called \emph{localities} in Chermak's framework \cite{Chermak:FL}, which are themselves special types of \emph{objective partial groups}. 
A standard example of a locality is the partial group $\loc{S\backslash1}{G}$ for the partial conjugation action on nonidentity elements $S\backslash1$ of a Sylow $p$-subgroup $S$ of $G$ \cite{HenkeLibmanLynd2023}.
Localities have been recently used by Chermak, Henke, and others \cite{ChermakHenke:FSLAD,Henke:CPNS} within a revision of the Classification of the Finite Simple Groups based on fusion systems along similar lines first outlined by Aschbacher. 

\subsection{Partial groups as higher Segal sets: initial results and motivation}
The main objective of this paper is to understand the higher associativity of partial groupoids by providing tools for deciding, for a fixed $d$, whether a partial groupoid is $d$-Segal. 

For $d = 1$ this is just the nerve theorem; a lower $1$-Segal partial groupoid is exactly (the nerve of a) groupoid. 
This project started when we wondered, like Segal's partial monoids considered in \cite{BOORS:2SSWC}, whether Chermak's partial groups give examples of $2$-Segal sets.
The answer is ``no'': 
\begin{quote}
\centering
If a partial group is 2-Segal, then it is a group.
\end{quote}
This is an immediate consequence of \cref{prop symmetric higher Segal}, which says that for symmetric sets, the lower $(2k{-}1)$-Segal, lower $2k$-Segal, upper $2k$-Segal, and upper $(2k{+}1)$-Segal conditions are all equivalent.

If $X$ is a spiny symmetric set, then these are further equivalent to the following (\cref{edgy segal characterization}): 
For each $n \geq 1$, each gapped sequence (meaning successive terms are at least two apart) 
\[
0 \leq i_0 \ll i_1 \ll \cdots \ll i_k \leq n
\]
of length $k+1$, and each potentially composable tuple 
\[
w \in X_1 \times_{X_0} \cdots \times_{X_0} X_1
\]
of length $n$, if the faces $d_{i_\ell}w$ are elements of $X_{n-1}$, then $w$ is an element of $X_n$.
This condition is the source of much of our additional intuition. 
It says that the composability of at least $k+1$ codimension $1$ faces not too close to one another implies the composability of the word itself. 
For instance, in a partial group, the first of the lower $3$-Segal conditions ($k = 2$, $n = 4$), says that if each of the three words $(f_2,f_3,f_4)$, $(f_1,d_1[f_2|f_3],f_4)$, and $(f_1,f_2,f_3)$ is multipliable, then so is $(f_1,f_2,f_3,f_4)$. 
Notice a tacit assumption in the condition that the relevant subwords of length two in $w$ are composable, so we can form word of length $n-1$ obtained by composing (applying the face map $d_1\colon X_2 \to X_1$) in the middle, like with $d_1[f_2|f_3]$. 

The collapsing of the higher Segal conditions to the lower odd ones leads to the following definition, which also makes sense for an arbitrary symmetric set.
\begin{definition}
\label{degree}
The \emph{degree} of a partial groupoid $X$ is the smallest $k \geq 1$ such that $X$ is lower $(2k{-}1)$-Segal. 
\end{definition}

The higher Segal conditions were originally defined by Dyckerhoff and Kapranov in terms of triangulations of cyclic polytopes.
The term ``degree'' comes instead from the results of Walde \cite{Walde:HSSHE}. 
In Walde's paper, lower $(2k{-}1)$-Segal objects 
are shown to be the polynomial functors of degree at most $k$ in a toy version of Goodwillie--Weiss manifold calculus for a class of ``coverings'' of the ``manifolds'' $[n]$. 

We are motivated to compute the degree of partial groupoids for at least two reasons.
First, the $d$-Segal conditions for $d > 2$ seem to be much less studied than the case $d = 2$, and there are fewer examples appearing in the existing literature. 
A primary purpose of this work is to provide a rich family of concrete examples of $d$-Segal sets for $d > 2$ of group theoretic interest. 
Second, Definition~\ref{degree} provides a new invariant of a partial group measuring its higher associativity.
Within finite group theory, this gives a new invariant of $p$-local structures of finite groups through Chermak's localities. 

\subsection{Degree as Helly number: main result}
Our main theorem is that the degree of a partial groupoid is a Helly number in the discrete geometry of (partial) actions, as we now explain.

One of the contributions of this paper is a flexible notion of action of a partial group.
An \emph{action} of a partial groupoid $L$ on a set $S$ is a map of partial groupoids $\rho \colon E \to L$ that is injective on stars as defined in \cref{actions} and has $E_0 = S$.
The prototypical example derives from a partial action of a group, where one can take $E$ to be the nerve of the corresponding transporter groupoid and $L$ to be the partial group $\loc{S}{G}$.

A partial action of a group gives a somewhat special type of action of $L_S(G)$, in as much as $E$ is the nerve of a groupoid (not just a partial groupoid) and $\rho$ is a surjective map of symmetric sets. 
If these two additional properties hold, we will call $\rho$ a \emph{characteristic action} of $L$. 
The defining internal conjugation action of an objective partial group $L$ on its object set is a characteristic action. 
The following theorem gives perspective as to the relative position of Chermak's objective partial groups within all partial groups. 

\begin{theorem}
\label{characteristic action}
Every partial groupoid $L$ admits a characteristic action.
If $\rho\colon E \to L$ is a characteristic action and $L$ embeds in $BG$ for some group $G$, then there is a partial action of $G$ on $E_0$ such that $\rho$ is isomorphic to $N(\trans{E_0}G) \to \loc{E_0}G = L$.
\end{theorem}

Thus, every partial group $L$ is ``objective with respect to some action'', but not necessarily an internal conjugation action on a set of subgroups of $L$. 
For an example of a partial group that does not embed in a group, see \cref{sec NA}. 

A partial action of $G$ on the set $S$ imbues $S$ with the structure of a closure space, in which the closed sets are the intersections of domains of the partial functions $S \overset{g}{\pto} S$. 
Likewise, for an action $\rho\colon E \to L$, the \emph{domain}  of an $n$-simplex $f \in L_n$ is the set of those $x \in E_0$ for which there is a lift of $f$ having source $x$.
This endows $E_0$ with a closure operator $A \mapsto \cl(A)$, where $\cl(A)$ is the intersection of those domains of simplices that contain $A$.

The classical \emph{Helly number} is the smallest $h$ such that whenever each $h$ members of a finite family of at least $h$ convex sets has nonempty intersection, the entire family has nonempty intersection. 
Helly's Theorem from 1913 says that the Helly number for convex subsets of $\mathbb{R}^d$ is $d+1$. 
In the current setting, the relevant Helly number $h=h(\rho)$ is the one for the abstract closure space $(E_0,\cl)$.

\begin{theorem}[Main Theorem]
\label{main}
Let $\rho \colon E \to L$ be a characteristic action of a partial groupoid $L$ such that $E_0$ satisfies the  descending chain condition on closed subsets.
If $L$ is not a groupoid, then $\deg(L) = h(\rho)$.
\end{theorem}

The chain condition is not necessary for the inequality $\deg(L) \leq h(\rho)$, but we do not know if it is necessary for the reverse inequality.
As one application of \cref{main}, we prove the following upper bound for the degree of a partial groupoid in terms of its dimension as a symmetric set.  
\begin{theorem}[\cref{deg dim}]
\label{degdim}
If $L$ is a nonempty partial groupoid, then $\deg(L) \leq \dim(L)+1$.
In particular, a finite partial groupoid (i.e., one with finitely many edges) has finite degree. 
\end{theorem}
The bound in \cref{degdim} does not hold for arbitrary symmetric sets (see \cref{ex sym sphere}).
We would be interested to know if there is an upper bound on the degree of a symmetric set in terms of its dimension. 

\subsection{Punctured Weyl groups and abelian sets of roots}

As an illustration of how \cref{main} can be applied to make concrete calculations, we introduce and study a collection of partial groups we call \emph{punctured Weyl groups}. 
These are the partial groups coming from the partial action of a finite Coxeter group $W$ on a set of positive roots $\pos$ in an associated root system $\rsys$. 
The underlying set of elements ($1$-simplices) of the partial group is $L_{\pos}(W)_1 = W\backslash\{w_0\}$ where $w_0$ is the longest element. 
We show the associated closure operator on $\pos$ is \emph{convex cone}, sending a subset $A$ of positive roots to $\cone_{\mathbb{R}}(A) = \mathbb{R}_{\geq 0}A \cap \pos$. 
\cref{main} then tasks us with solving a Helly type problem for $\pos$.
This is facilitated by the following observation, which appears to be new. 
\begin{proposition}[\cref{helly free Z-closure}]
\label{hZpos}
If $\rsys$ is crystallographic, then the Helly number of $\pos$ with respect to ordinary closure $A \mapsto \mathbb{Z}_{\geq 0}A \cap \pos$ is the maximal size of an abelian set of positive roots. 
\end{proposition}
A subset $A \subseteq \rsys$ is \emph{abelian} if the sum of two roots in $A$ is never a root. 
The maximal size of an abelian set of positive roots was computed by Malcev in 1945, since it agrees with the maximal dimension of an abelian subalgebra of the corresponding complex semisimple Lie algebra. 
A \emph{really abelian} set of positive roots is to convex closure what an abelian set of roots is to ordinary closure in crystallographic types. 
By determining the maximal size of a really abelian set of positive roots, we produce the degrees in \cref{degLposW table} 
(and the degree is additive in orthogonal unions of root systems). 
For example, the punctured Weyl group of $E_8$ is lower $71$-Segal, but not lower $69$-Segal. 
\begin{table}
{
\renewcommand{\arraystretch}{1.3}
\setlength{\tabcolsep}{12pt}
\centering
\caption{Degree of punctured Weyl groups}\label{degLposW table}
\begin{tabular}{@{}lcclc@{}}\toprule
$\rsys$ & $\deg(L_{\pos}(W))$ && $\rsys$ & $\deg(L_{\pos}(W))$\\
\cmidrule(r){1-2} \cmidrule(l){4-5}
$A_n$ & $\lfloor \frac{(n+1)^2}{4} \rfloor$ && $B_n/C_n$ & $\binom{n}{2}+1$  \\
$D_n$ & $\binom{n}{2}$  && $F_4$ & $6$  \\
$E_6$ & $16$ && $G_2$& $2$ \\
$E_7$ & $27$             && $I_2(m)$ & $2$ 
\\
$E_8$ & $36$ && $H_3$& $5$  \\
      &      && $H_4$& $8$ \\
\bottomrule
\end{tabular}
}
\end{table}

Punctured Weyl groups are combinatorial analogues of the $p$-local punctured groups of \cite{HenkeLibmanLynd2023}. 
The latter are special types of localities admitting a characteristic action by conjugation on the set of nonidentity subgroups of a Sylow subgroup of $L$. 
\Cref{main} applies to all (finite) localities, so can be used to compute the degree of localities via the consideration of Helly type problems for Sylow intersections.
For example, it can be shown that the degree of a $p$-local punctured group $L$ is at most the $p$-rank of a Sylow subgroup.
This turns out to be a good bound in many cases.
For symmetric groups of odd degree at least 7 at the prime 2, our student Omar Dennaoui has shown in work in progress that it is sharp. 
As a curiosity, it is also sharp for a $2$-local punctured group $L$ of an exotic Benson--Solomon fusion system (giving $\deg(L) = 4$). 
\Cref{main} applies in certain infinite settings, such as for the discrete localities of Chermak and Gonzalez \cite{ChermakGonzalez:DLI} associated with the $p$-local compact groups of Broto, Levi, and Oliver \cite{BrotoLeviOliver2007}, which model the $p$-local structures of compact Lie groups and $p$-compact groups. 
We plan to return to these themes in later papers.

\subsection{Outline and suggestions for reading}

The next section (with the exception of \cref{sec NA}, which may be skipped) consists of background material and establishes notation.
\Cref{sec higher Segal} is core material on the higher Segal conditions, including for symmetric simplicial objects.
Some of this material is given an alternate interpretation in \cref{sec d Segal pgpd}, both in terms of words and in terms of stars.
Readers primarily interested in partial groups may wish to skim \cref{sec higher Segal} on a first reading and proceed quickly to \cref{sec d Segal pgpd}, with special attention to \cref{edgy segal char redux}, \cref{d seg embed in BG}, and \cref{bs words}.

Actions of partial groups are defined and studied in \cref{actions}, along with a concrete description in \cref{app action}.
Of particular importance are the characteristic actions of \cref{def char map} and the actions coming from a partial action of a group in \cref{ex realizable locality}.
\Cref{sec action to closure} discusses how each action gives rise to a closure space.

We next turn to the classical theory of Helly independence and the Helly number in \cref{sec helly independence}. 
The key takeaways from this section are \cref{helly independence,def helly crit} along with \cref{helly rank v helly critical} comparing them.

Our main theorem comparing degree and Helly number is proved in \cref{sec:degree_as_helly_rank}.
We apply the main theorem to finite dimensional partial groupoids in \cref{sec:the_finite_dimensional_case}, establishing the dimension bound and addressing the stability of degree under reduction, and we compute the degree of punctured Weyl groups in \cref{punctured weyl}.

\subsection*{Acknowledgements}
We've had illuminating and helpful discussions about this project with many people over the past few years, and in addition to others we're probably forgetting we'd like to thank
Kaya Arro,
Alexander Berglund,
Julie Bergner,
Tim Campion,
Andy Chermak,
Tobias Dyckerhoff,
Matthew Dyer,
George Glauberman,
Alex Gonzalez,
Jonas Hartwig,
Joachim Kock,
Robin Koytcheff,
Rémi Molinier,
Marco Praderio,
Edoardo Salati,
Brandon Shapiro,
Jan Steinebrunner,
Rafael Stenzel,
Walker Stern,
and
Jan Šťovíček.

\section{Simplicial machinery}
\subsection{Simplicial and symmetric sets}\label{sec simp symm}
Let $\ord$ denote the simplicial indexing category, whose objects are the sets $[n] = \{0, 1, \dots, n\}$ for $n\geq 0$ and whose morphisms are the order preserving maps.
The category $\ord$ is contained in the category $\fin$ which has the same objects, but where the morphisms are arbitrary functions.
The category of simplicial sets is $\sset = \fun(\ord^\op, \set)$ whose objects are contravariant functors and morphisms are natural transformations, while the category of symmetric (simplicial) sets is $\sym = \fun(\fin^\op, \set)$.
If $X$ is a simplicial set (or symmetric set) we write $X_n$ for $X([n])$ and $\alpha^* \colon X_n \to X_m$ for the image of the map $\alpha \colon [m] \to [n]$ in $\ord$ (resp.\ in $\fin$).
There is a forgetful functor $\sym \to \sset$ obtained by restriction along the inclusion $\ord \to \fin$, and we do not notationally distinguish between a symmetric set $X$ and its underlying simplicial set.
We also write $d_i \colon X_n \to X_{n-1}$ and $s_i \colon X_n \to X_{n+1}$ for the usual face and degeneracy operators in a simplicial set $X$.
The symbols $d_\bot, d_\top \colon X_n \to X_{n-1}$ will denote $d_0$ and $d_n$, respectively.

\begin{example}[Nerve of a category or groupoid]\label{ex nerve}
Each object $[n] = \{ 0 < \dots < n \}$ can be considered as a category with a unique morphism $i \to j$ just when $ i \leq j$.
The nerve functor $N \colon \cat \to \sset$
sends a small category $C$ to the simplicial set $NC$ given by the formula $NC_n = \hom_{\cat}([n], C)$.
The $n$-simplices are composable strings of morphisms
\[ \begin{tikzcd}[sep=small]
x_0 \rar{f_1} & x_1 \rar{f_2} & x_3 \rar{f_3} & \cdots \rar{f_n} & x_n,
\end{tikzcd} \]
sometimes written concisely as $[f_1 | f_2 | f_3 | \cdots | f_n]$ 
(when $n\geq 1$; a length $0$ string is an object of $C$, not easily expressible in the bar notation.)
Regarding $[n]$ as a groupoid with a unique morphism between any two objects, there is likewise an inclusion $\fin \to \gpd$.
An analogous nerve construction assigns to every groupoid its nerve as a symmetric set.
We do not distinguish notationally between the nerve and this groupoidal nerve 
as the underlying simplicial set of the latter agrees with the former.
\end{example}

\begin{convention}
We regard a group $G$ as a category with a single object $\ast$, automorphism group $\hom(\ast, \ast) \coloneq G$, and $g \circ h \coloneq g  h$.
This provides an embedding $B \colon \grp \to \sset$ (or $B \colon \grp \to \sym$).
The $i$\textsuperscript{th} face map $d_i$ has the effect
\[
  [g_1 | \dots | g_n] \mapsto [g_1 | \dots | g_{i-1} | g_{i+1} g_i | g_{i+2} | \dots | g_n]
\]
since our convention is to apply maps from right to left. 
Also $d_0$ deletes $g_1$ and $d_n$ deletes $g_n$. 
\end{convention}

\begin{example}\label{ex representable}
For $n\geq 0$, we write $\rep{n} \coloneq \hom_{\fin}(-,[n]) \in \sym$ for the representable symmetric set.
This is the nerve of the chaotic groupoid on $n+1$ objects, i.e., the groupoid having a unique morphism between any two objects. 
A $k$-simplex of $\rep{n}$ is just a function $[k] \to [n]$. 
It may be written unambiguously as a list of of length $k+1$, like $32351 \in \rep{7}_4$ in place of $\begin{tikzcd}[cramped] 3 \to 2 \to 3 \to 5 \to 1 \end{tikzcd}$.
Its boundary $\partial \rep{n}$ consists of those functions $[k]\to [n]$ which are not surjective.
More generally, its $m$-skeleton $\sk_m \rep{n}$ consists of those functions $[k] \to [n]$ whose image has at most $m+1$ elements.
\end{example}

Recall that every simplicial set (resp.\ symmetric set) $X$ has an opposite $X^\op$ (see, for instance, \cite[\href{https://kerodon.net/tag/003L}{Tag 003L}]{kerodon}).
It is induced by precomposing with the identity on objects functor $\ord \to \ord$ (resp.\ $\fin \to \fin$) which sends $f \colon [n] \to [m]$ to $\tau_m f \tau_n$, where $\tau_n \colon [n] \to [n]$ is given by $\tau_n(i) = n-i$.
If $C$ is a category or groupoid, then there is a natural isomorphism $N(C^\op) \cong (NC)^\op$.

It is also convenient to have at hand the category $\ordalt$ whose objects are nonempty subsets of the nonnegative integers $\mathbb{N} = \{0, 1, 2, \dots \}$, and morphisms are order preserving functions.
The inclusion 
\[
\ord \to \ordalt
\]
is an equivalence of categories, and there is a unique functor $\ordalt \to \ord$ realizing this: send an object $\{ i_0 < i_1 < \dots < i_n \} \subseteq \mathbb{N}$ to $[n]$.
Restriction gives an equivalence of categories $\fun(\ordalt^\op, \set) \to \fun(\ord^\op, \set) = \sset$.
We tacitly regard every simplicial set $X$ as a presheaf over this larger category via the composite
\[
	\ordalt^\op \to \ord^\op \xrightarrow{X} \set.
\]

\subsection{Edgy simplicial sets and spiny symmetric sets}\label{spiny edgy}

It was observed by Gonz\'alez in \cite{gonzalez} (see also \cite{BrotoGonzalez:ETPG}) that a partial group as defined by Chermak is really a simplicial set of a certain type.
In \cite{HackneyLynd:PGSSS}, we explained how the extra data of an inversion in Gonz\'alez's characterization of partial groups was just a property that a simplicial set might or might not have, and this property was equivalent to the simplicial set having a unique lift to a symmetric set. 
Here we recall some parts of \cite{HackneyLynd:PGSSS}, but formulated slightly differently in terms of an outer face complex $\words X$ of words in $1$-simplices of $X$. 

\begin{definition}
A map $f \colon [m] \to [n]$ in $\ord$ is called \emph{inert} if $f(i+1) = f(i) + 1$ for all $i=0,\dots,m-1$. 
Let $\delint \subseteq \ord$ be the subcategory with the same objects as $\ord$ and with morphisms the inert maps.
An \emph{outer face complex} is a functor $\delint^\op \to \set$.
\end{definition}

To put it another way, an outer face complex is a sequence of sets $X_n$ together with operators $d_\bot, d_\top \colon X_n \to X_{n-1}$ such that $d_\top d_\bot = d_\bot d_\top$.
Given a simplicial set $X$, there is an evident outer face complex given by restriction along the inclusion of $\delint$ into $\ord$.
Here is different sort of example. 

\begin{definition}[Word complex]\label{def word complex}
Given a simplicial set $X$, let $\words X$ be the outer face complex with $\words (X)_0 = X_0$ and for $n\geq 1$,
\begin{align*}
\words (X)_n &= \{ (f_1, \dots, f_n) \mid f_i \in X_1 \text{ and } d_1(f_i) = d_0(f_{i-1}) \} \\
&= X_{\{0,1\}} \times_{X_{\{1\}}} X_{\{1,2\}} \times_{X_{\{2\}}} \cdots \times_{X_{\{n-1\}}} X_{\{n-1,n\}}\\
&= X_1 \times_{X_0} X_1 \times_{X_0} \cdots \times_{X_0} X_1.
\end{align*}
The maps $d_\bot, d_\top \colon \words (X)_1 = X_1 \to \words (X)_0 = X_0$ agree with the maps in $X$, while for $n > 1$, the maps $d_\bot$ and $d_\top$ delete $f_1$ and $f_n$, respectively. 
\end{definition}

Notice that $\words X$ depends only on the simplicial 1-skeleton of $X$.
We've defined it as a presheaf on the category of inert maps, but we could have instead defined it as a presheaf on the larger category of \emph{contractive maps}: those $f \colon [m] \to [n]$ such that $f(i+1) \leq f(i) + 1$ for all $i=0,1,\dots, m-1$.
This would amount to having outer face maps together with degeneracy operators $s_i \colon \words (X)_n \to \words (X)_{n+1}$.

If $X$ is a simplicial set and $n\geq 1$, then the \emph{Segal map} 
\[ 
\edgemap_n \colon X_n \to \words (X)_n \subseteq \prod_{i=1}^n X_1
\]
sends $x \in X_n$ to $(\epsilon_{01}^*x,\dots,\epsilon_{n-1,n}^*x)$. 
Here, $\epsilon_{ij} \colon [1] \to [n]$ is $ij$\textsuperscript{th} \emph{coedge map} sending $0$ to $i$ and $1$ to $j$. 
Along with the identity $\edgemap_0 \colon X_0 \to X_0$, these maps assemble into a map of outer face complexes $\edgemap \colon X \to \words X$.
The following is \cite[4.1]{Grothendieck:TCTGA3}.

\begin{theorem}
A simplicial set $X$ is isomorphic to the nerve of a category if and only if the map of outer face complexes $\edgemap \colon X \to \words X$ is an isomorphism, the \emph{Segal condition}.
Similarly, a symmetric set $X$ is isomorphic the nerve of a groupoid if and only if $\edgemap \colon X \to \words X$ is an isomorphism.
\end{theorem}
In an edgy simplicial set, the Segal condition is relaxed. 

\begin{definition}
A simplicial set is \emph{edgy} if $\edgemap \colon X \to \words X$ is injective.
\end{definition}

\begin{notation}
If $X$ is an edgy simplicial set, we will sometimes write elements of $X_n$ in the form $[f_1 | \dots | f_n]$ when their image under $\edgemap_n$ is $(f_1, \dots, f_n)$.
That is, the bar notation is reserved for actual elements of $X_n$, in contrast with the group theoretical literature, where $(f_1,\dots,f_n)$ typically plays both roles.
An edge in the image of the degeneracy $X_0 \to X_1$ is denoted $\id_x$, and if $[f_1 | f_2] \in X_2$ then $f_2 \circ f_1$ or $f_2 f_1$ is notation for $d_1 [f_1 | f_2] \in X_1$.
\end{notation}

We can make the same definition for a symmetric set.
\begin{definition}
A symmetric set is \emph{spiny} if its underlying simplicial set is edgy.
\end{definition}

The Segal map has to do with the standard spine of a simplex, namely that associated to the spanning tree $\{\{i-1,i\}\}$ of $[n]$. 
However, by \cite[Theorem~3.6]{HackneyLynd:PGSSS}, the property of being spiny can be checked on whichever spanning tree of $[n]$ one prefers.
(This is definitely not the case for the property of being edgy, which is why we use terminology that distinguishes between the two.)
Other than the standard spine, the most important of these for us is $\{\{0,i\}\}$, which gives rise to the Bousfield--Segal map
\[
  \bous_n \colon X_n \to \prod_{i=1}^n X_1
\]
sending $x$ to $(\epsilon_{01}^*x,\dots,\epsilon_{0n}^*x)$ for $n\geq 1$. 
For example, $\bous_3$ sends a simplex of the form $[f_1|f_2|f_3]$ to the word $(f_1, f_2f_1, f_3f_2f_1)$.
The Bousfield--Segal map motivates the following.

\begin{definition}[Starry word complex]
Let $\finbot \subset \fin$ be the subcategory consisting of those maps $\alpha \colon [n] \to [m]$ such that $\alpha(0) = 0$.
Given a symmetric set $X$, let $\bswords X$ be the presheaf $\finbot^\op \to \set$ with $\bswords(X)_0 = X_0$ and for $n\geq 1$
\begin{align*}
\bswords(X)_n &= \{ (f_1, \dots, f_n) \mid d_1(f_1) = d_1(f_2) = \cdots = d_1(f_n) \}
 \\ 
 &= X_{\{0,1\}} \times_{X_{\{0\}}} X_{\{0,2\}} \times_{X_{\{0\}}} \cdots \times_{X_{\{0\}}} X_{\{0,n\}}.
\end{align*}
Suppose $m,n \geq 1$.
The unique maps $[n] \to [0]$ and $[0] \to [m]$ induce the maps $ \bswords(X)_0 \to \bswords(X)_n$ with $x \mapsto (\id_x, \dots, \id_x)$ and $\bswords(X)_m \to \bswords(X)_0$ with $(f_1, \dots, f_m) \mapsto d_1(f_i)$.
Given $\alpha \colon [n] \to [m]$ in $\finbot$ with $m,n \geq 1$, define $\alpha^* \colon \bswords(X)_m \to \bswords(X)_n$ by 
\[
  \alpha^*(f_1, \dots, f_m) = (f_{\alpha(1)}, \dots, f_{\alpha(n)})
\]
where $f_0$ is the identity having the same source as the other $f_i$.
\end{definition}

The Bousfield--Segal map $\bous_n$ lands in $\bswords(X)_n$.
These maps (including $\bous_0 = \id \colon X_0 \to \bswords(X)_0$) assemble into a $\finbot$-presheaf map $\bous \colon X \to \bswords X$.
The following is a consequence of \cite[Theorem~3.6]{HackneyLynd:PGSSS} and \cite[Theorem~4]{HackneyMolinier:DPG}.

\begin{theorem}\label{bous spiny Segal}
A symmetric set $X$ is spiny if and only if if $\bous \colon X \to \bswords X$ is a monomorphism, and $X$ is isomorphic to the nerve of a groupoid if and only if $\bous \colon X\to \bswords X$ is an isomorphism.
\end{theorem}

The second part of this is a version of Grothendieck's nerve theorem for groupoids.
This perspective on spininess and Segality will be very important later in the paper, and we will return to it starting in \cref{bs words}.

A spiny symmetric set also has, for each $n$, a useful injection
\[
	X_n \to \mat_{n+1,n+1} (X_1)
\]
sending $f$ to the matrix whose $ij$\textsuperscript{th} entry is $f_{ij} \coloneq \epsilon_{ij}^*f$, and we sometimes identify $f$ with the matrix
\[
(f_{ij}) = 
\begin{bmatrix}
f_{00} & f_{01} & f_{02} & \cdots & f_{0n} \\
f_{10} & f_{11} & f_{12} & \cdots & f_{1n} \\
\vdots & & & & \vdots \\
f_{n0} & f_{n1} & f_{n2} & \cdots & f_{nn}
\end{bmatrix}.
\]
The superdiagonal of this matrix is $\edgemap_n(f)$ and the tail of its top row is $\bous_n(f)$, but it can be helpful to have the entire matrix at hand. 
For example, we have $f_{jk} \circ f_{ij} = f_{ik}$ and $f_{ij}^{-1} = f_{ji}$, and if $\sigma$ is a map in $\fin$ we have $\sigma^*(f_{ij}) = (f_{\sigma i, \sigma j})$.
One can also read off information about degeneracy from the matrix form.

\begin{definition}\label{def degenerate}
Let $X$ be a symmetric set and $x\in X_n$ an $n$-simplex.
Then $x$ is \emph{degenerate} if there exists a noninvertible surjection $\sigma \colon [n] \twoheadrightarrow [m]$ and an element $y \in X_m$ such that $x = \sigma^*y$.
If no such pair $(\sigma, y)$ exists, then $x$ is \emph{nondegenerate}.
\end{definition}

\begin{lemma}
\label{char nondeg}
In a spiny symmetric set, the following are equivalent for an $n$-simplex $f$ with matrix form $(f_{ij})$. 
\begin{enumerate}[label=\textup{(\arabic*)}, ref=\arabic*]
\item\label{nondeg} $f$ is nondegenerate.
\item\label{nonid offdiag} If $f_{ij}$ is an identity, then $i = j$. 
\item No row of $(f_{ij})$ contains a repeated element.
\item There exists a row of $(f_{ij})$ which does not contain a repeated element.
\item\label{bs distinct} $f_{01},\dots, f_{0n}$ are distinct, nonidentity elements. 
\end{enumerate}
\end{lemma}
\begin{proof}
The equivalence of \eqref{nondeg} and \eqref{nonid offdiag} is \cite[Lemma~7]{HackneyMolinier:DPG}.
For a fixed $0 \leq k \leq n$, we have
$f_{ij} = f_{kj} f_{ik} = f_{kj}f_{ki}^{-1}$,
so row $k$ contains a repeated element if and only if there is an $i\neq j$ with $f_{ij}$ an identity.
This gives the equivalence of \eqref{nonid offdiag} with the last three.
\end{proof}

The last item of \cref{char nondeg} says that $\bous_n$ fully detects degeneracy in a spiny symmetric set, by asking whether any elements of $\bous_n(x)$ are duplicates.
By skew-symmetry, we could have used columns instead of rows in the statement of \cref{char nondeg}.

The following is \cite[Corollary 4.7]{HackneyLynd:PGSSS}.

\begin{theorem}
The category of reduced spiny symmetric sets is equivalent to the category of partial groups.
\end{theorem}

\begin{convention}
In this paper we use the term \emph{partial groupoid} as a synonym for ``spiny symmetric set'' and write $\pgpd \subset \sym$ for the full subcategory of partial groupoids.
Likewise, \emph{partial group} will mean a reduced partial groupoid. 
We also frequently identify groups and groupoids with their nerves.
\end{convention}

\subsection{The platonically non-associative partial groupoid}\label{sec NA}
It is immediate from the definitions that any symmetric subset of the nerve of a groupoid is automatically a partial groupoid.
In particular, if $G$ is a group, any nonempty symmetric subset of $BG$ is a partial group.
But it has been known from the beginning that not every partial group embeds in a group, and in this section we construct a small example of this. 

First, detach the front two faces from the back and bottom faces of the boundary $\partial \rep{3} \subseteq \rep{3} = \hom(-,[n])$ of a symmetric $3$-simplex. 

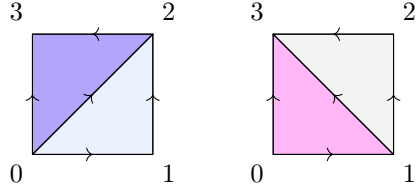
\begin{center}
\begin{tikzpicture}[
  scale = 0.8,
  line join = round, 
  line cap = round,
  decoration={
    markings,
    mark=at position 0.5 with {\arrow{>}}
  }
]
\coordinate [label={[font=\normalsize]below left:0}] (a0) at (0,0);
\coordinate [label={[font=\normalsize]below right:1}] (a1) at (2,0);
\coordinate [label={[font=\normalsize]above right:2}] (a2) at (2,2);
\coordinate [label={[font=\normalsize]above left:3}] (a3) at (0,2);

\draw[fill=lblue,fill opacity=.7] (a0)--(a1)--(a2)--cycle;
\draw[fill=mediumpurple,fill opacity=.7] (a0)--(a2)--(a3)--cycle;

\draw[postaction={decorate}] (a0)--(a1) node [midway, below=1pt] {};
\draw[postaction={decorate}] (a1)--(a2) node [midway, right=1pt] {};
\draw[postaction={decorate}] (a2)--(a3) node [midway, above=1pt] {};
\draw[postaction={decorate}] (a0)--(a3);
\draw[postaction={decorate}] (a0)--(a2);

\coordinate [label={[font=\normalsize]below left:0}] (b0) at (4,0);
\coordinate [label={[font=\normalsize]below right:1}] (b1) at (6,0);
\coordinate [label={[font=\normalsize]above right:2}] (b2) at (6,2);
\coordinate [label={[font=\normalsize]above left:3}] (b3) at (4,2);

\draw[fill=lightpink,fill opacity=.7] (b0)--(b1)--(b3)--cycle;
\draw[fill=verylightgray,fill opacity=.9] (b1)--(b2)--(b3)--cycle;

\draw[postaction={decorate}] (b0)--(b1) node [midway, below=1pt] {};
\draw[postaction={decorate}] (b1)--(b2) node [midway, right=1pt] {};
\draw[postaction={decorate}] (b2)--(b3) node [midway, above=1pt] {};
\draw[postaction={decorate}] (b0)--(b3);
\draw[postaction={decorate}] (b1)--(b3);
\end{tikzpicture}
\end{center}
Then, glue these back along the spine of the original $3$-simplex
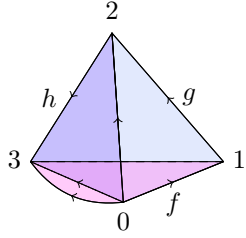
\begin{center}
\begin{tikzpicture}[scale = 1.4, line join = round, line cap = round]
\coordinate [label={[font=\normalsize]above:2}] (2) at (0,{sqrt(2)},0);
\coordinate [label={[font=\normalsize]right:1}] (1) at ({.5*sqrt(3)},0,-.5);
\coordinate [label={[font=\normalsize]below:0}] (0) at (0.3,0,0.5);
\coordinate [label={[font=\normalsize]left:3}] (3) at ({-.5*sqrt(3)-0.1},0,-.5);

\draw[fill=bblue,fill opacity=.5] (2)--(1)--(0)--cycle;
\draw[fill=mediumpurple,fill opacity=.5] (3)--(2)--(0)--cycle;
\draw[fill=lightpink,fill opacity=.5] (0)--(1)--(3) to[bend right=30] (0)--cycle;

\draw (0)--(2);
\draw[dashed] (1)--(3);
\draw[densely dotted] (0)--(3); 

\begin{scope}[decoration={markings,mark=at position 0.5 with {\arrow{to}}}]
\draw[postaction={decorate}] (0)--(1) node [midway, below=0.5pt] {$f$};
\draw[postaction={decorate}] (1)--(2) node [midway, right=2pt] {$g$};
\draw[postaction={decorate}] (2)--(3) node [midway, left=2pt] {$h$};
\draw[postaction={decorate}] (0)--(2);
\draw[postaction={decorate}] (0)--(3);
\draw[postaction={decorate}, bend left=30] (0) to (3); 
\end{scope}
\end{tikzpicture}
\end{center}
to obtain a symmetric set we will call $\NA$. 

In order to formalize this, let $F$ (front faces) be the symmetric subset of $\rep{3}$ whose $k$-simplicies are those $\alpha\colon [k] \to [3]$ with image in either $\{0,1,2\}$ or $\{0,2,3\}$. 
Similarly let $B$ (back faces) be the symmetric subset consisting of those $\alpha$ with image in either $\{0,1,3\}$ or $\{1,2,3\}$. 
Then $F$ and $B$ are partial groupoids containing the spine $\spine{3} \subset \rep{3}$, the symmetric subset of those $\alpha$ with image in one of $\{0,1\}$, $\{1,2\}$, or $\{2,3\}$ (see \cite[Remark 5.20]{HackneyLynd:PGSSS}). 
Finally let $Q = F \cap B$, the union of $\Sp^3$ with those simplices having image in $\{0,3\}$. 

Consider the pushout $\NA = F \sqcup_{\spine{3}} B$ in the category of symmetric sets: 
\[
\begin{tikzcd}
\spine{3} \dar[hook] \rar[hook] & F \dar[hook]\\
B \rar[hook] & \NA. \arrow[ul, phantom, "\ulcorner" very near start]
\end{tikzcd}
\]
It is a lot like $\partial \rep{3}$, except that the 03 edge of the 023 triangle has not been reglued to the 03 edge of the 013 triangle.  
Instead, $\NA$ has a double $03$-edge corresponding to the two ways of associating the $01$, $12$, and $23$ edges. 
In the pushout
\[
F \sqcup_Q B = \partial \rep{3}
\]
the two $03$ edges get reidentified. 
Thus, there is a quotient map $q \colon \NA \twoheadrightarrow \partial \rep{3}$ induced by the inclusion $\spine{3} \hookrightarrow Q$ and the identities on $F$ and $B$. 

\begin{lemma}
\label{NA spiny}
$\NA$ is spiny.
\end{lemma}
\begin{proof}
Let $k,k'\in \NA_n \cong \hom(\rep{n},\NA)$ be two $n$-simplices which have the same spine. 
If $k$ and $k'$ are both in $F_n \subseteq \NA_n$, then since $F$ is spiny we have $k=k'$.
The same holds if $k,k'$ are both in $B_n$, so we assume $k\in F_n$ and $k'\in B_n$.
Since $\partial \rep{3}$ is spiny, $q \circ k = q \circ k'$.
This means that both factor through $Q \cong F \times_{\partial \rep{3}} B$. 
\[
\begin{tikzcd} 
\rep{n} \drar[dashed] \ar[drr, bend left, "k"] \ar[ddr, bend right, "k'"'] &[-0.5cm] \\[-0.5cm]
& Q \dar[hook] \rar[hook] & F \dar[hook]\\
& B \rar[hook] & \partial \rep{3}
\end{tikzcd}
\]
But $Q$ is 1-dimensional, so $\rep{n} \to Q$ factors through $\rep{1}$, hence so do $k$ and $k'$.
This implies $k,k' \in \sk_1(\NA)$, and since 1-dimensional symmetric sets are always spiny, we conclude that $k=k'$ in $\sk_1(\NA)$, hence in $\NA$.
\end{proof}

We call $\NA$ the \emph{platonically non-associative partial groupoid}.
If $(f,g,h)$ denotes the image in $\words(\NA)_3$ of the spine $(01,12,23)$ of $\id_{[3]} \in \rep{3}_3$ (as pictured above), then the edges $f,g,h \in \NA_1$ have the property that
\[
[f|g], [g|h], [g \circ f | h], [f| h \circ g] \in \NA_2, 
\]
but 
\[
h \circ (g \circ f) \neq (h \circ g) \circ f \text{ in } \NA_1. 
\]
$\NA$ is the minimal example of such a partial groupoid in the sense that, by Yoneda's lemma, maps from $\NA$ to a partial groupoid $X$ which are injective on the pair of long edges $0\to 3$ are in bijection with words $(f,g,h) \in \words (X)_3$ with the above properties. 
Likewise, the reduction of $\NA$ could be called the platonically non-associative partial group, and shares the same universal property (in the category of partial groups, rather than the category of partial groupoids).
We will compute the degree of $\NA$ in \cref{ex degree NA}. 

\section{Higher Segal spaces}\label{sec higher Segal}

At the beginning of this section, we provide background material on higher Segal spaces. 
Our approach highlights that the types of arguments used in the decomposition space literature can also be used for higher Segal spaces, by replacing pullback squares with cartesian cubes of larger dimension.
This relies on work of Walde, who recast the higher Segal conditions in terms of cartesian cubes. 
We'll begin with preliminaries on cartesian cubes before turning to the higher Segal conditions for simplicial objects in \cref{ss HSC Walde}.
Arguments in this section are in the spirit of those in \cite{Hackney:DSP}.
In \cref{ss SSO Segal} we discuss the case of symmetric simplicial objects, where a number of subtleties vanish.

In this section, $\mathcal{C}$ will denote a fixed category (or $\infty$-category) with finite limits.
All subsequent sections of the paper will take $\mathcal{C}$ to be the category of sets, and the reader is welcome to make this replacement immediately.

\subsection{Cubical diagrams}
We recall basics about cube-shaped diagrams in a category or $\infty$-category; references include \cite[\S3.3]{Walde:HSSHE} and \cite[\S6.1.1]{LurieHA}.
The \emph{generic cube} of dimension $n$ is the $n$-fold product $[1]^n \in \cat$ of the generic arrow $ \{ 0 \to 1 \} = [1] \in \ord \subset \cat$.
An \emph{$n$-dimensional cubical diagram} in $\mathcal{C}$, or briefly an \emph{$n$-cube}, is a functor $[1]^n \to \mathcal{C}$.
The functor category $\fun([1]^n, \mathcal{C})$ is the associated category of cubes.
If $S$ is a set of cardinality $n$, then we may also think of a functor $\ps(S) \to \mathcal{C}$ from the powerset of $S$ as an $n$-cube in $\mathcal{C}$ by choosing an isomorphism $\ps(S) \cong [1]^n$ (and similarly for $\ps(S)^\op$).

A map between $n$-cubes may be regarded as an $(n{+}1)$-cube.
Namely, we have the following description of the arrow category of the category of cubes, for each choice of isomorphism $[1] \times [1]^n \cong [1]^{n+1}$:
\[
	\fun([1], \fun([1]^n, \mathcal{C})) \cong \fun([1] \times [1]^n, \mathcal{C}) \cong \fun([1]^{n+1}, \mathcal{C}).
\]

An $n$-cube $Q \colon [1]^n \cong \ps(S) \to \mathcal{C}$ is \emph{cartesian} if it is a limit diagram.
Another way to say this is that $Q$ is cartesian if and only if it is right Kan extended from its restriction to the punctured cube $[1]^n \setmin 0 \cong \ps(S) \setmin \{ \varnothing \}$ (i.e.\ $Q \simeq i_*i^*Q$ where $i$ is in the inclusion of the punctured cube, $i_*$ is right Kan extension, and $i^*$ restriction).
We now recount several basic lemmas about cartesian cubes that we will need below.

\begin{lemma}\label{lem retract}
Retracts of cartesian $n$-cubes are cartesian.
\end{lemma}
\begin{proof}
This is an instance of a general fact about closure of limit diagrams under retracts; see e.g.\ \cite[\href{https://kerodon.net/tag/05E6}{Tag 05E6}]{kerodon}.
\end{proof}

The following two well-known lemmas likely first appear in the Goodwillie calculus literature (with $\mathcal{C}$ the $\infty$-category of spaces). 
See Proposition~1.6 and Proposition~1.8 of \cite{Goodwillie:CalcII}.
In this generality, the next lemma is \cite[Lemma 3.3.8]{Walde:HSSHE}.

\begin{lemma}\label{cube lemma 2}
Let $P$ and $Q$ be $n$-cubes, and $R \colon P \to Q$ an $(n{+}1)$-cube. 
If $Q$ is cartesian, then $P$ is cartesian if and only if $R$ is cartesian.
\end{lemma}

\begin{lemma}\label{generalized pasting law}
Suppose $P$, $Q$, and $R$ are $(n{+}1)$-cubes, which satisfy $R = Q \circ P$ when regarded as maps of $n$-cubes.
If $Q$ is cartesian, then $P$ is cartesian if and only if $R$ is cartesian.
\end{lemma}
This lemma generalizes the usual pasting law for pullbacks when $n=1$.
For completeness, we provide a proof in \cref{cube appendix} in the generality of $\mathcal{C}$ an arbitrary $\infty$-category with finite limits.

\begin{remark}
When $\mathcal{C}$ is a complete and cocomplete $\infty$-category, these lemmas also follow from corresponding results for derivators (Theorem~8.7 and Proposition~8.11 of \cite{GrothStovicek:TTSHS}) applied to the homotopy derivator of $\mathcal{C}$. 
When $\mathcal{C}$ is a stable $\infty$-category, stronger statements hold -- see Corollary~A.16 and Corollary~A.18 of \cite{DyckerhoffJassoWalde}.
\end{remark}

\subsection{Higher Segal conditions after Walde}\label{ss HSC Walde}

In Walde's perspective on the higher Segal conditions \cite{Walde:HSSHE}, a simplicial object is lower $(2k{-}1)$-Segal if and only if it maps each strongly bicartesian $(k{+}1)$-dimensional cube in $\ord$ to a cartesian cube.
Strongly bicartesian means that each $2$-dimensional face is bicartesian.
On the other hand, not all strongly bicartesian cubes need be checked, only the ones with injective edges.
This collection of cubes corresponds precisely to intersection cubes associated with gapped subsets, as we now explain. 

Given an object $S \in \ordalt$ and proper subset $I \subset S$, the associated \emph{intersection cube} in $\ordalt$ is the functor 
\[
\intcu{I} = \intcube{I}{S} \colon \ps(I)^\op \to \ordalt
\]
that sends a subset $J \subseteq I$ to its complement $S\setmin J$ in $S$. 
We use the abbreviation $\intcu{I}$ when $S$ is understood. 
To illustrate the terminology, note that if we let $S_i = S\setmin i$ for $i\in I$,
then the vertices of the cube are the intersections $\intcu{I}_J = \bigcap_{j\in J} S_j$ of the $S_i$, and the edges are the inclusions.
The initial vertex of the cube is $\bigcap_{i\in I} S_i = S \setmin I$, and the terminal vertex is $S$.
For example, if $I = \{i_0,\dots,i_k\} \subset [n] = S$, then the intersection cubes for $k = 0$ and $1$ are 
$[n]$ and $[n] \leftarrow [n]\setmin \{i_0\}$, while the ones for $k = 2$ and $3$ look like
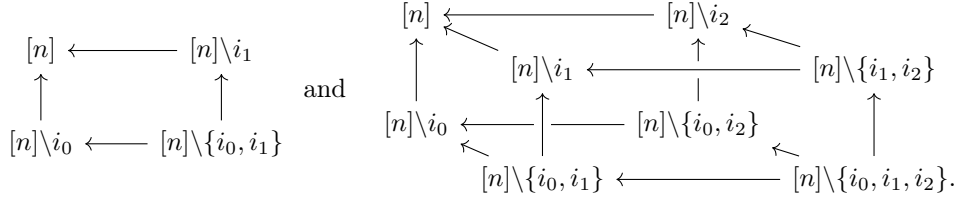
\[
\begin{tikzcd}
{[n]} & {[n]\setmin i_1} \lar\\
{[n]\setmin i_0} \uar & {[n]\setmin \{i_0,i_1\}} \uar \lar
\end{tikzcd}
\,\, \text{and} \quad
\begin{tikzcd}[column sep=0.8ex, row sep=0.8ex]
{[n]} \ar[from=rr] \ar[from=dd] \ar[from=dr] & & 
{[n]\setmin i_2} \ar[from=dd] \ar[from=dr]  \\
& {[n]\setmin i_1} \ar[from=rr, crossing over] & & {[n]\setmin\{i_1,i_2\}} \ar[from=dd] \\
{[n]\setmin i_0} \ar[from=rr] \ar[from=dr] & & {[n]\setmin\{i_0,i_2\}} \ar[from=dr] \\
& {[n]\setmin\{i_0,i_1\}} \ar[from=rr] \ar[uu, crossing over] & & {[n]\setmin\{i_0, i_1,i_2\}}.
\end{tikzcd}
\]

\smallskip
If $X \colon \ordalt^\op \to \mathcal{C}$ is a simplicial object, then composing $\intcu{I}$ with $X$ yields a cube
\[
X\intcu{I} = X\intcube{I}{S} \colon \ps(I) \to \mathcal{C}
\]
in $\mathcal{C}$ whose initial vertex is $X_S$ and whose terminal vertex is $X_{S \setmin I}$. 
In the special case where $S = [n]$ and $I$ has cardinality $k+1$, these amount to $X_S = X_n$ and $X_{S \setmin I} = X_{n-k-1}$. 

\smallskip
A subset $I \subseteq S$ is \emph{gapped} if for each pair of elements $i < i'$ in $I$, there is $j \in S$ such that $i < j < i'$. 
If $S = [n]$, this just means that each pair of distinct elements of $I$ are at distance at least two from one another. 
We use the notation $i \ll i'$ if there exists such a gap $j$ between $i$ and $i'$, so that a gapped subset can be written as a sequence
\[
i_0 \ll i_1 \ll i_2 \ll \cdots \ll i_k
\]
where $k+1$ is the cardinality of $I$.
As an example, if $I$ is the gapped subset $0 \ll i \ll n$ of $[n]$, then 
\[
  \begin{tikzcd}[row sep=small]
    X_n \ar[rr,"d_n"] \ar[dd,"d_0"'] \ar[dr,"d_i"'] & & 
    X_{n-1} \ar[dd, "d_0"' very near start] \ar[dr,"d_i"]  \\
    & X_{n-1} \ar[rr,"d_{n-1}" near end, crossing over] & & X_{n-2} \ar[dd,"d_0"] \\
    X_{n-1} \ar[rr,"d_{n-1}" near start] \ar[dr,"d_{i-1}"'] & & X_{n-2} \ar[dr,"d_{i-1}" near start] \\
    & X_{n-2} \ar[rr,"d_{n-2}"] \ar[from=uu, crossing over, "d_0" very near start] & & X_{n-3}.
  \end{tikzcd}
\]
is the corresponding 3-dimensional cube $X\cube{I}$.
\begin{definition}\label{def lower odd Segal}
Let $k$ be a positive integer.
A simplicial object $X$ is \emph{lower $(2k{-}1)$-Segal} if for every $n \in \mathbb{N}$ and every gapped set $I \subset [n]$ of cardinality $k+1$, the associated cube $X\cube{I}\colon \ps(I) \to \mathcal{C}$ is cartesian.
\end{definition}

These conditions generalize the usual Segal condition. 
Indeed, the lower 1-Segal condition coincides with the Segal condition (see e.g.\ \cite[\S1]{Walde:HSSHE} or \cite[Ex.~3.9]{Dyckerhoff:CPOC}), and if $X$ is lower $(2k{-}1)$-Segal, it is also lower $(2k{+}1)$-Segal (\cref{prop hierarchy}).

\begin{remark}\label{rmk walde lower segal}
\Cref{def lower odd Segal} is a distillation of the main theorem of \cite{Walde:HSSHE}. 
Specifically, it is a combination of Corollary 4.1.5, Theorem 6.1.1, and Theorem 7.2.2 of \cite{Walde:HSSHE}, along with an unraveling of a compatible (Definition 4.1.1) and primitive (Definition 4.3.1) claw whose constituent maps are injective.
\end{remark}

Each lower $(2k{-}1)$-Segal condition is a one-parameter family of conditions on $X$ involving gapped subsets of $[n]$ for each $n \geq 0$.
Observe however, that there are no gapped subsets $I$ of $[n]$ of cardinality $k+1$ if $n < 2k$;
the first nonvacuous condition involves the gapped sequence $0 \ll 2 \ll \cdots \ll 2k$ in $[2k]$. 

\begin{example}
If $X$ is the nerve of a category, then the first lower $1$-Segal condition ($k = 1$, $n = 2$) says that 
a $2$-simplex amounts to a pair $(f,g)$ of morphisms agreeing target to source, which is of course the case. 
The first of the lower $3$-Segal conditions ($k = 2$, $n = 4$) says a $4$-simplex amounts to a 
triple of $3$-simplices $[f|g|h]$, $[g|h|k]$, and $[f|h \circ g|k]$, again the case.
\end{example}

\begin{remark}\label{exercise sset}
Using a cofinality argument, the lower $(2k{-}1)$-Segal condition for a simplicial set $X$ can be reformulated as follows:
for every gapped set $I \subset [n]$ of cardinality $k+1$ and every list of $(n{-}1)$ simplices $(x_i) \in \prod_I X_{n-1}$ satisfying $d_i x_j = d_{j-1} x_i$ for $i < j$ in $I$, there exists a unique $x\in X_n$ such that $d_i x = x_i$ for all $i\in I$.
\end{remark}

The following is a variant on similar results for pullback squares, e.g. \cite[Lemma 3.10]{GKT1}. 
It reduces further the number of cubes one needs to check for lower $(2k{-}1)$-Segality. 

\begin{lemma}\label{segality top bottom}
Let $k$ be a positive integer and $X$ a simplicial object.
Assume that for each $n \in \mathbb{N}$ and each gapped subset $I \subset [n]$ of cardinality $k+1$ containing both $0$ and $n$, the cube $X\cube{I}$ is cartesian.
Then $X$ is lower $(2k{-}1)$-Segal.
\end{lemma}
\begin{proof}
By the cowidth of a subset $I \subset S \in \ordalt$ we mean the cardinality of the set $\{s \in S \mid s < \min(I) \text{ or } s > \max(I)\}$, 
so for $S = [n]$ we are in the hypotheses of the lemma just when the cowidth is $0$. 
Assume that for each $S' \in \ordalt$ and each gapped subset $I' \subset S'$ of cardinality $k+1$ and cowidth $0$, the associated cube $X\cube{I'}$ is cartesian. 
Fix $S \in \ordalt$ and a gapped subset $I \subset S$ of cardinality $k+1$. 
We wish to prove that $X\cube{I}$ is cartesian, and this is by induction first on $|S|$, and then on the cowidth of $I$ in $S$.
In the case $|S| = 2k+1$, the unique gapped subset has cowidth $0$ and so the result holds by hypothesis. 

Assume now $|S| > 2k+1$. 
We may assume that either $\min(S) \notin I$ or $\max(S) \notin I$, say the latter, so that $\max(I) < \max(S)$.
Write $m = \max(I)$ and $n = \max(S)$ for short. 
A subscript on $I$ or $S$ indicates that the corresponding elements have been removed. 
For example, $I_m = I \setmin m$ and $S_{m,n} = S \setmin \{m,n\}$. 
We also set $J = I \setmin m \cup n$, which is gapped in both $S_m$ and in $S$. 
Observe that since $m < n$ by assumption, $J$ has strictly smaller cowidth in $S$ than $I$ does.

Regard $\intcube{I}{S}$ as a map of cubes from $\intcube{I_m}{S_m}$ to $\intcube{I_m}{S}$. 
This is the top arrow in the commutative diagram
\[
\begin{tikzcd}[column sep=large]
\intcube{I_m}{S_m} \rar{\intcube{I}{S}} & \intcube{I_m}{S}\\
\intcube{I_m}{S_{m,n}} \uar{\intcube{J}{S_m}} \rar{\intcube{I}{S_n}} & \intcube{I_m}{S_n} \uar[swap]{\intcube{J}{S}}.
\end{tikzcd}
\]
By induction, $X\intcube{I}{S_n}$ and $X\intcube{J}{S}$ are cartesian, hence so is their composite by \cref{generalized pasting law}. 
By induction, the cube $X\intcube{J}{S_m}$ is cartesian as well. 
So by \cref{generalized pasting law} again, $X\intcube{I}{S}$ is cartesian.
\end{proof} 

We are now ready to introduce the other higher Segal conditions.

\begin{definition}[Other Higher Segal conditions]\label{def other higher Segal}
Let $k$ be a positive integer and $X$ a simplicial object.
Consider the collection of gapped subsets $I \subset [n]$ of cardinality $k+1$ and the associated collection of cubes $X\cube{I} \colon \ps(I) \to \mathcal{C}$.
We say that $X$ is
\begin{enumerate}
\item \emph{lower $2k$-Segal} if $X\cube{I}$ is cartesian whenever $0 \notin I$,
\item \emph{upper $2k$-Segal} if $X\cube{I}$ is cartesian whenever $n \notin I$, and
\item \emph{upper $(2k{+}1)$-Segal} if $X\cube{I}$ is cartesian whenever $0 \notin I$ and $n \notin I$.
\end{enumerate}
\end{definition}

If the definition seems a bit ad hoc, the reason is that all of our definitions are in terms of cartesian cubes, rather than the original geometric definitions (see \cite{poguntke} and \cite[p.\ xv]{DyckerhoffKapranov}) in terms of upper and lower triangulations of cyclic polytopes.
Walde proved in \cite{Walde:HSSHE} (see \cref{rmk walde lower segal}) that \cref{def lower odd Segal} is equivalent to the geometric definition of lower $(2k{-}1)$-Segal.
Independently, Poguntke proved the characterization in \cref{prop path space criterion} below for the original geometric definitions \cite[Proposition 2.7]{poguntke}.
As \cref{prop path space criterion} holds for the conditions defined in \cref{def other higher Segal}, this means that they coincide with the geometric ones.

\begin{remark}
One could also take $k=0$ in \cref{def lower odd Segal} and \cref{def other higher Segal} to arrive at notions of lower $(-1)$-Segal, lower and upper $0$-Segal, and upper $1$-Segal.
The latter three appear in \cite{poguntke}.
An adaptation of the proof of \cref{prop symmetric higher Segal} below shows that all four of these conditions coincide for a simplicial object $X \colon \ord^\op \to \mathcal{C}$, and just mean that $X$ is constant \cite[Ex.\ 3.9]{Dyckerhoff:CPOC}.
We will not consider this `degree zero' (\cref{def degree}) case any further in this paper, always taking $k > 0$ and not distinguishing between discrete and nondiscrete groupoids.
\end{remark}

The following is immediate from \cref{def lower odd Segal} and \cref{def other higher Segal}.

\begin{lemma}[Opposites]\label{lem opposites}
Let $X$ be a simplicial object and $d$ a positive integer.
If $d$ is odd, then $X$ is lower or upper $d$-Segal if and only if $X^\op$ is so.
If $d$ is even, then $X$ is lower $d$-Segal if and only if $X^\op$ is upper $d$-Segal. \qed
\end{lemma}

To state the next proposition, we need the \emph{d\'ecalage} functors of Illusie \cite[VI.1]{Illusie:CCD2}
\[
\ldec, \udec \colon \fun(\ord^\op, \mathcal{C}) \to \fun(\ord^\op, \mathcal{C}),
\]
which we now define (see also \cite[\S6]{Hackney:DSP}). 
There is a functor $\ord \to \ord$ which sends $[n]$ to the ordinal sum $[0] \star [n] = [n+1]$. 
Restriction along this functor induces the lower d\'ecalage functor $\ldec \colon \fun(\ord^\op, \mathcal{C}) \to \fun(\ord^\op, \mathcal{C})$.
If $X$ is a simplicial object, then $\ldec X$ is obtained from $X$ by deleting $X_0$, setting $\ldec X_n = X_{n+1}$, and deleting the bottom face and degeneracy maps (and renumbering the remaining ones by 1):
\[ \begin{tikzcd}
X: &  X_0 \rar["s_0" description] & X_1 \lar[shift left=2, "d_0"] \lar[shift right=2, "d_1"']  \rar[shift left=1.5] \rar[shift right=1.5] & X_2 \lar[shift left=3,"d_0"] \lar[shift right=3,"d_2"'] \lar
\rar[shift left=3] \rar[shift right=3] \rar &
X_3 
\lar[shift left=1.5] \lar[shift left=4.5] \lar[shift right=1.5] \lar[shift right=4.5] \cdots \\[+0.25cm]
\ldec X: &   & X_1  \rar[shift left=1.5] \rar[shift right=1.5, dotted] & X_2 \lar[shift left=3, dotted, "\color{gray} d_0"] \lar[shift right=3,"d_2"'] \lar
\rar[shift left=3] \rar[shift right=3, dotted] \rar &
X_3 
\lar[shift left=1.5] \lar[shift left=4.5, dotted] \lar[shift right=1.5] \lar[shift right=4.5] \cdots 
\end{tikzcd} \]
That is, $d_k \colon \ldec X_n \to \ldec X_{n-1}$ is equal to $d_{k+1} \colon X_{n+1} \to X_n$ (and similarly for degeneracies).
(In \cite{DyckerhoffKapranov}, $\ldec X$ is called the initial path space $P^\triangleleft X$.)
Likewise, there is a functor $\ord \to \ord$ sending $[n]$ to $[n] \star [0] = [n+1]$ and restriction along it induces the upper d\'ecalage functor $\udec \colon \fun(\ord^\op, \mathcal{C}) \to \fun(\ord^\op, \mathcal{C})$.
We again have $\udec X_n = X_{n+1}$ for $n\geq 0$, and this time we delete the top face and degeneracy maps (no renumbering of the remaining faces/degeneracies is necessary).

\begin{proposition}[Path space criterion \cite{poguntke}]\label{prop path space criterion}
Let $X$ be a simplicial object and $k$ a positive integer.
\begin{enumerate}
		\item $X$ is lower $2k$-Segal if and only if $\ldec X$ is lower $(2k{-}1)$-Segal.\label{PSC ldec}
		\item $X$ is upper $2k$-Segal if and only if $\udec X$ is lower $(2k{-}1)$-Segal.\label{PSC udec}
    \item $X$ is upper $(2k{+}1)$-Segal if and only if $\ldec \udec X = \udec \ldec X$ is lower $(2k{-}1)$-Segal.\label{PSC double}
\end{enumerate}
\end{proposition}
Combining the criteria, $X$ is upper $(2k{+}1)$-Segal if and only if $\ldec X$ is upper $2k$-Segal if and only if $\udec X$ is lower $2k$-Segal.
These separate conditions are how \eqref{PSC double} is presented in \cite{poguntke,Dyckerhoff:CPOC}.
\begin{proof}
We prove \eqref{PSC udec}.
The natural inclusion $\delta^{n+1} \colon [n] \to [n+1]$ gives a bijection between gapped sets $I\subset [n]$ of cardinality $k+1$ and gapped sets $I' \subset [n+1]$ of cardinality $k+1$ such that $n+1 \notin I'$.
Under this correspondence, the cube $(\udec X)\cube{I}$ is equal to the cube $X\cube{\delta^{n+1}I}$, as $(\udec X)_{[n] \setmin J} = X_{[n+1] \setmin J}$ for $J\subseteq I \subset [n]$.
This establishes \eqref{PSC udec}.
The other statements are proved similarly, replacing $\delta^{n+1}$ by $\delta^0 \colon [n] \to [n+1]$ and $\delta^0 \delta^{n+1} \colon [n] \to [n+2]$.
\end{proof}

In a sense, the path space criterion tells us that \cref{def lower odd Segal} is the most essential of the higher Segal conditions. 
In \cref{prop symmetric higher Segal} we will see that this is even more pronounced for symmetric sets.

These higher Segal conditions fit into a hierarchy, due to the following proposition which appears as \cite[Proposition 2.10]{poguntke}; the cases that are not immediate from the definitions are that upper $(2k{-}1)$-Segal implies $2k$-Segal, and that lower or upper $2k$-Segal implies lower $(2k{+}1)$-Segal.

\begin{proposition}[Poguntke]\label{prop hierarchy}
If $X$ is lower or upper $d$-Segal, then $X$ is both lower $(d{+}1)$-Segal and upper $(d{+}1)$-Segal.
\end{proposition}
\begin{proof}
Throughout $k$ is a positive integer.
We first show that if $X$ is lower $2k$-Segal, then $X$ is lower $(2k{+}1)$-Segal.
Let $I \subseteq S = [n]$ be a gapped subset of cardinality $k+2$  with $0,n \in I$.
By \cref{segality top bottom} it is enough to show that $X\cube{I}$ is cartesian.
The set $I_0 = I \setmin 0$ is gapped in both $S_0 = S\setmin 0$ and $S$ and contains the minimal element of neither.
By lower $2k$-Segality, the cubes $X\intcube{I_0}{S_0}$ and $X\intcube{I_0}{S}$ are cartesian. 
Since $X\cube{I}$ is the map $d_0$ between them, $X\cube{I}$ is cartesian by \cref{cube lemma 2}. 

If $X$ is upper $2k$-Segal, then it is lower $(2k{+}1)$-Segal by \cref{lem opposites} and the previous paragraph. 
If $X$ is upper $(2k{+}1)$-Segal, then $X$ is both lower and upper $(2k{+}2)$-Segal by the path space criterion and the previous paragraph. 
\end{proof}

\subsection{Symmetric simplicial objects}\label{ss SSO Segal}
We now turn to the case of symmetric simplicial objects in $\mathcal{C}$, i.e.\ objects in the category $\fun(\fin^\op, \mathcal{C})$. 
For us the most important case will be symmetric sets $\sym = \fun(\fin^\op, \set)$.
We say that a symmetric simplicial object $X$ is (upper or lower) $d$-Segal if and only if its underlying simplicial object is so.
That is, we use the restriction functor $\fun(\fin^\op, \mathcal{C}) \to \fun(\ord^\op, \mathcal{C})$ associated to the subcategory inclusion $\ord \to \fin$ to define the conditions.

\begin{proposition}\label{prop symmetric higher Segal}
Let $X \colon \fin^\op \to \mathcal{C}$ be a symmetric simplicial object in $\mathcal{C}$ and $k$ a positive integer.
The following are equivalent:
\begin{enumerate}
\item\label{lower2k-1} $X$ is lower $(2k{-}1)$-Segal.
\item\label{lower2k} $X$ is lower $2k$-Segal.
\item\label{upper2k} $X$ is upper $2k$-Segal.
\item\label{upper2k+1} $X$ is upper $(2k{+}1)$-Segal.
\end{enumerate}
\end{proposition}

Of course it is immediate from the definitions that $\eqref{lower2k-1} \Rightarrow \eqref{lower2k} \Rightarrow \eqref{upper2k+1}$ and $\eqref{lower2k-1} \Rightarrow \eqref{upper2k} \Rightarrow \eqref{upper2k+1}$, without the hypothesis that $X$ is symmetric.
We prove $\eqref{upper2k+1} \Rightarrow \eqref{lower2k-1}$ in the symmetric case using the following lemma.
In the special case when $m=n$, we have $X\cube{I} \simeq X\cube{I'}$, so $X\cube{I}$ is cartesian if and only if $X\cube{I'}$ is so.

\begin{lemma}\label{lem removal of gaps}
Let $X$ be a symmetric simplicial object and $I \subset [n]$ is a proper subset of cardinality $k+1$.
If $m \geq n$ and there is a subset $I' \subset [m]$ of cardinality $k+1$ such that $X\cube{I'}$ is cartesian, then $X\cube{I}$ is cartesian as well.
\end{lemma}
\begin{proof}
Let $\sigma \colon [m] \to [n]$ be a surjective function with $\sigma(I') = I$ and $\sigma^{-1}(I) = I'$.
Such a function exists since $I$ is a proper subset of $[n]$ (but it may not always be taken to be order preserving).
Let $\delta \colon [n] \to [m]$ be a section of $\sigma$.
Since $\delta$ is injective, it induces a map of cubes $\cube{I} \to \cube{I'}$.
But our choice of $\sigma$ also implies that it induces a map of cubes $\cube{I'} \to \cube{I}$.
Namely, the dashed function below exists for each $J \subseteq I'$.
\[ \begin{tikzcd}
{[m] \setmin J} \rar[dashed] \dar[hook] & {[n] \setmin \sigma(J)} \dar[hook] \\
{[m]} \rar{\sigma} & {[n]}
\end{tikzcd} \]
Since $\sigma \circ \delta = \id_{[n]}$, this exhibits $\cube{I}$ as a retract of $\cube{I'}$, and hence $X\cube{I}$ as a retract of $X\cube{I'}$. 
But $X\cube{I'}$ is cartesian by assumption, so $X\cube{I}$ is cartesian by \cref{lem retract}.
\end{proof}

\begin{proof}[Proof of \cref{prop symmetric higher Segal}]
Suppose $X$ is upper $(2k{+}1)$-Segal, and $I \subset [n]$ a gapped subset of cardinality $k+1$.
Then $I' = \{ i + 1 \mid i\in I \} \subset [n+2]$ is gapped and has $0,n+2 \notin I'$, hence $X\cube{I'}$ is cartesian.
By \cref{lem removal of gaps}, $X\cube{I}$ is cartesian.
\end{proof}

\begin{remark}\label{rmk non-gapped}
\Cref{lem removal of gaps} shows that if a symmetric simplicial object $X$ is lower $(2k{-}1)$-Segal, then $X\cube{I}$ is cartesian for every proper subset $I \subset [n]$, without any hypothesis about $I$ being gapped.
In particular, the first of these cubes has initial vertex $X_n$ where $n=k+1$, rather than $n=2k$.
\end{remark}

In light of \cref{prop symmetric higher Segal} and \cref{prop hierarchy}, the following is natural.

\begin{definition}\label{def degree}
The \emph{degree} of a symmetric simplicial object $X$ is the least positive integer $k$ such that $X$ is lower $(2k{-}1)$-Segal.
It is denoted by $\deg(X)$.
If no such integer exists, we say that $X$ has infinite degree and set $\deg(X) = \infty$.
\end{definition}

The terminology is motivated by \cite{Walde:HSSHE}, where lower $(2k{-}1)$-Segal objects are interpreted as polynomial functors of degree at most $k$.
If $X$ is a symmetric set, then $\deg(X) = 1$ if and only if $X$ is isomorphic to the nerve of a groupoid. 
We first look at a family of examples where the degree grows linearly in the dimension.

\begin{example}[Skeleta of the symmetric simplex]\label{ex boundary degree}
For $1 \leq m \leq n$, the $m-1$ skeleton of the representable object on $[n]$ has degree $m$ (see \cref{def skeleton}).
In particular, $\deg(\partial \rep{n}) = \deg(\sk_{n-1} \rep{n}) = n$.
Of course the statement is not true for $m > n$, as then $\sk_{m-1} \rep{n} = \rep{n}$ has degree 1.
The $p$-simplices of $X = \sk_{m-1} \rep{n}$,  may be identified with length $p+1$ ordered lists of elements in $[n]$ which include at most $m$ values.
If $m=1$, then $\sk_0 \rep{n}$ is the nerve of the discrete groupoid with object set $[n]$, hence has degree 1.
If $2 \leq m \leq n$, then $x = 102030\cdots 0m \in \rep{n}_{2m-2}$ is not an element of $X_{2m-2}$ since it includes the $m+1$ elements $\{0,1,\dots, m\}$, but its face $d_{2i} x$ is missing $i+1$, hence is in $X_{2m-3}$. 
Using \cref{exercise sset}, the elements $x_0 = d_0 x, x_2 = d_2 x, \dots, x_{2m-2} = d_{2(m-1)} x$ show that $X$ is not $(2(m{-}1){-}1)$-Segal, hence has degree at least $m$.
But $X$ is $(m{-}1)$-dimensional and spiny, hence has degree at most $m$ by \cref{deg dim} below.
\end{example}

\begin{example}[Symmetric sphere]\label{ex sym sphere}
Fix $n\geq 1$, and let $X = \rep{n} / \partial \rep{n}$, given by identifying all $m$-simplices in $\partial \rep{n}$ to a single point $\ast_m$.
So elements of $X_m$ are the surjective functions $[m] \twoheadrightarrow [n]$ (alternatively, length $m+1$ strings containing all elements of $[n]$), along with $\ast_m$.
We will show in \cref{app sphere} that $\deg(X) = 2n$.
\end{example}

Above, we defined d\'ecalage functors $\ldec, \udec \colon \fun(\ord^\op, \mathcal{C}) \to \fun(\ord^\op, \mathcal{C})$, and these may be extended to the symmetric case.
Indeed, the endofunctors $[0] \star (-)$ and $(-) \star [0]$ on $\ord$ extend to functors $\fin \to \fin$, and pulling back along them gives $\ldec, \udec \colon \fun(\fin^\op, \mathcal{C}) \to \fun(\fin^\op, \mathcal{C})$.

\begin{proposition}\label{degree dec}
If $X$ is a symmetric simplicial object, then \[ \deg(\ldec X) = \deg(X) = \deg(\udec X).\]
\end{proposition}
\begin{proof}
By \cref{prop path space criterion}, $\ldec X$ is lower $(2k{-}1)$-Segal if and only if $X$ is lower $2k$-Segal.
According to \cref{prop symmetric higher Segal}, this occurs if and only if $X$ is lower $(2k{-}1)$-Segal.
A similar argument establishes the second equality.
\end{proof}

\section{Higher Segal conditions for partial groupoids}
\label{sec d Segal pgpd}

We now shift our focus to edgy simplicial sets and spiny symmetric sets, where we can be more concrete about the higher Segal conditions.
These reduce to the following question described in the introduction: given a word $w = (f_1, \dots, f_n) \in \words (X)_n$ that has several ``faces'' in $X_{n-1} \subset \words (X)_{n-1}$, is it always the case that $w$ is in $X_n$?
The main result of this section is \cref{edgy segal characterization}, which provides this characterization.
For partial groupoids, it is often easier to work with starry words as in \cref{bs words}, and we give an analogous characterization via starry words in \cref{bs words characterization}.

\subsection{Edgy simplicial sets}
\label{sec edgy segal}

Let $X$ be an edgy simplicial set, and recall the outer face complex $\words X$ from \cref{def word complex} along with the map $\mathscr{E}\colon X \to \words X$ of outer face complexes. 
We discussed in \cref{spiny edgy} why it is reasonable to consider $\words X$ as having degeneracies as well as outer faces.
But it will generally have some inner faces as well:
\begin{definition}\label{extra inner face}
Let $w = (f_1, \dots, f_n) \in \words(X)_n$ and $1 \leq i \leq n-1$.
If $[f_i | f_{i+1}] \in X_2$, then we define
\[
	d_i(w) \coloneq (f_1, \dots, f_{i-1},  f_{i+1} \circ f_i, f_{i+2}, \dots, f_n).
\]
\end{definition}
Notice $d_i \edgemap_n(x)$ is defined for every $x\in X_n$, and is equal to $\edgemap_{n-1}d_i(x)$.
But \cref{extra inner face} is strictly more general, and $d_i(w)$ may lie outside of the image of $\edgemap_{n-1}$.

\begin{notation}
\label{notation WkX}
When $X$ is an edgy simplicial set and $I \subset [n]$ is a gapped subset, $\words_I(X)_n$ is the set
of words $w$ of length $n$ such that $d_i(w)$ is defined and in the image of $\edgemap_{n-1}$ for each $i \in I$. 
Let $\words_k(X)_n$ be the union of $\words_I(X)_n$ as $I$ ranges over the gapped subsets of $[n]$ of cardinality $k+1$. 
\end{notation}

The Segal map $\edgemap_n$ is normally regarded as having codomain $\words(X)_n$, but it has image in $\words_I(X)_n$ for each $I$ so can be viewed as a map to $\words_I(X)_n$ or to $\words_k(X)_n$ when convenient (the latter only when $n \geq 2k$). 
Our main goal in this section is to explain in \cref{edgy segal characterization} that the $(2k{-}1)$-Segality of an edgy simplicial set comes down to the surjectivity of the Segal maps $X_n \to \words_k(X)_n$. 
The most important ingredient for this is \cref{WkX as limit}, which requires a little setup.

Whenever a gapped subset $I \subset [n]$ is fixed, consider the following inclusions of full subcategories of $\ps(I)$ below left and the corresponding limits of restrictions of $X\cube{I}$ below right.
\[
\begin{tikzcd}[column sep=1em]
    & \ps(I)  & \\
\ps(I)_{12} \arrow[ur, "\iota_{12}"] \arrow[rr, swap,"\kappa"] && \ps(I)_{>0} \arrow[ul, swap, "\iota"],
\end{tikzcd}
\quad\quad
\begin{tikzcd}[column sep=1em]
    & X_n \arrow[dl, swap, "\iota_{12}^*"] \arrow[dr] & \\
\lim X\cube{I}_{12}  && \lim X\cube{I}_{>0} \arrow[ll],
\end{tikzcd}
\]
Here, $\ps(I)_{>0}$ has nonempty subsets of $I$ and $\ps(I)_{12}$ has subsets of cardinality $1$ or $2$. 
We are interested in the cartesianess of the cube $X\cube{I}$, which is that the right diagonal map above is a bijection.
Here, and in what follows, we identify $X_n \cong \lim X\cube{I}$.  

\begin{lemma}\label{lem suspects as limit}
Let $X$ be an edgy simplicial set and $I \subset [n]$ a gapped subset of size at least $2$. 
There is a bijection $\delta \colon \words_I(X)_n \to \lim X\cube{I}_{12}$ such that $\delta \circ \edgemap_n = \iota_{12}^*$.
\label{WkX as limit}
\end{lemma}
\begin{proof}
The following square of partial functions commutes for $i \ll j$ in $I$, crucially because $I$ is gapped:  
\[ 
\begin{tikzcd}
\words(X)_n \rar["\arrownot\quad"{marking}] \dar["\arrownot\quad"{marking}] & \words(X)_{[n] \setmin i} \dar["\arrownot\quad"{marking}] \\
\words(X)_{[n] \setmin j} \rar["\arrownot\quad"{marking}] & \words(X)_{[n] \setmin \{i,j\}}. 
\end{tikzcd} 
\]
For example, both ways around send a word $(f_1, \dots, f_n)$ to 
\[ 
(f_1, \dots, f_{i-1}, f_{i+1} \circ f_i, f_{i+2}, \dots, f_{j-1}, f_{j+1} \circ f_j, \dots, f_n)
\]
if neither $i$ nor $j$ are endpoints, provided $[f_i | f_{i+1}]$ and $[f_j | f_{j+1}]$ are in $X_2$.
(A modified argument applies if $i = 0$ or $j=n$.)
This shows that a word $w \in \words_I(X)_n$ determines an element $(x_\bullet(w)) \in \lim X\cube{I}_{12}$ defined uniquely by the conditions $\edgemap_{n-1}x_i(w) = d_i(w)$ and $\edgemap_{n-2}x_{ij}(w) = d_i d_j(w) \in X_{n-2}$ for $i \ll j$ in $I$. 

The map $\delta\colon w \mapsto (x_\bullet(w))$ satisfies $\delta \circ \edgemap_n = \iota_{12}^*$ by construction, and we claim that it is a bijection. 
Given an element $(x_\bullet)$ of the limit, we can produce a unique word $w = (f_1, \dots, f_n)$ from it as follows.
Choose, for each $m = 1, \dots, n$ an element $i \in I \setmin \{m-1,m\}$, and let $f_m$ be the $\{m-1,m\}$ edge of $x_i$, i.e., the restriction of $x_i \in X_{[n]\setmin i}$ along the inclusion of $\{m-1,m\}$ into $[n]\setmin i$. 
The existence of $i$ uses the assumption $|I| \geq 2$ when $\{m-1,m\} \cap I$ is not empty, but $f_m$ is otherwise independent of the choice of $i$. 
Indeed, if $j$ is another element of $I$ not in $\{m-1,m\}$, then the inclusion of $\{m-1,m\}$ factors through $[n]\setmin \{i,j\}$. 
So $f_m$ is the restriction of $x_{ij}$, and hence also of $x_j$. 

This implies the uniqueness of $w$. 
To complete the proof of existence, it remains to check that $w$ is in $\words_I(X)_n$ and has $\edgemap_{n-1}x_i = d_i(w)$ for each $i \in I$. 
The set $\{ i-1, i , i+1 \}$ is a subset of $[n] \setmin j$ for each $j \in I\setmin i$ by gappedness.
So using that $|I| \geq 2$ to get there is such a $j$, we have first that $[f_i | f_{i+1}] \in X_2$ for each $i \in I$ with $0 < i < n$, and then that the $\{i-1,i+1\}$ edge of $x_i$ agrees with the $\{i-1,i+1\}$ edge of $x_j$, which is $d_1[f_i|f_{i+1}]$.
Since the other principal edges of $x_i$ agree with $d_i(w)$ by construction of $w$, this shows $\edgemap_{n-1}x_i = d_i(w)$. 
\end{proof}

\begin{theorem}\label{edgy segal characterization}
Let $k \geq 0$ be an integer.
An edgy simplicial set $X$ is lower $(2k{-}1)$-Segal if and only if $\edgemap_n\colon X_n \to \words_k(X)_n$ is surjective for all $n \geq 2k$. 
\end{theorem}
\begin{proof}
From the definitions, $\edgemap_n \colon X_n \to \words_k(X)_n$ is a surjection (i.e., a bijection) if and only if it is a bijection onto $\words_I(X)_n$ for all gapped sets $I \subset [n]$ of cardinality $k+1$. 
Fix $n \geq 2k$ and such an $I$, and consider the diagram
\[
\begin{tikzcd}[column sep=3em]
    & X_n \arrow[dl, swap, "\edgemap_n"] \arrow[d, "\iota_{12}^*"] \arrow[dr] & \\
\words_I(X)_n \arrow[r,"\delta"] & \lim X\cube{I}_{12} & \arrow[l] \lim X\cube{I}_{>0} \arrow[l]. 
\end{tikzcd}
\]
For each nonempty subset $J$ of $I$, the overcategory $\kappa \downarrow J$ is the poset $\ps(J)_{12}$ of $1$ and $2$ element subsets of $J$. 
Since this is nonempty and connected, the bottom right map is a bijection by \cite[Lemma~8.3.4]{Riehl:CHT}. 
It follows that $X\cube{I}$ is cartesian if and only if $\iota_{12}^*$ is a bijection, which is the case just when $\edgemap_n$ is a bijection by \cref{WkX as limit}. 
\end{proof}

The following is a restatement of the theorem.

\begin{corollary}
\label{edgy segal char redux}
Let $X$ be an edgy simplicial set. 
Then $X$ is lower $(2k{-}1)$-Segal if and only if for each $n \geq 1$, each gapped sequence 
$0 \leq i_0 \ll i_1 \ll \cdots \ll i_k \leq n$
of length $k+1$, and each potentially composable tuple 
$w \in X_1 \times_{X_0} \cdots \times_{X_0} X_1$
of length $n$, if 
$d_{i_0}w, d_{i_1}w, \dots, d_{i_k}w$ are all defined and in $X_{n-1}$, then $w \in X_n$.
\end{corollary}

When a partial group embeds in a group, one may use this additional structure to avoid discussion of partially-defined inner face maps on $\words X$, as in the following.

\begin{corollary}
\label{d seg embed in BG}
Let $C$ be a category and suppose $X\subseteq NC$ is a simplicial subset.
The simplicial set $X$ is lower $(2k{-}1)$-Segal if and only if for each $n\geq 1$ and each gapped sequence $0 \leq i_0 \ll i_1 \ll \cdots \ll i_k \leq n$, and each $f= [f_1| \dots| f_n] \in NC_n$, if $d_{i_0}f, d_{i_1}f, \dots, d_{i_k}f \in X_{n-1}$ then $f\in X_n$. \qed
\end{corollary}

\begin{example}\label{ex bcom M}
If $M$ is a monoid, we write $\bcom M \subseteq BM$ for the simplicial subset of commuting tuples of elements, i.e.\ $(\bcom M)_n = \hom(\mathbb{N}^n, M) \subseteq M^{\times n}$, where $\mathbb{N}^n$ is the free commutative monoid on $n$ generators.
This simplicial set is always lower $3$-Segal.
Suppose we have $(m_1, \dots, m_n) \in M^{\times n}$ and $1 < i < n-1$ such that 
\[
  [m_2 | \dots | m_n], \quad [m_1 | \dots | m_{i+1}m_i | \dots | m_n], \quad [m_1 | \dots | m_{n-1}] \in (\bcom M)_{n-1}.
\]
The first element tells us that $m_i m_j = m_j m_i$ for $i,j \geq 2$, the third element tells us this same equality for $i,j \leq n-1$. 
The last one to check is $m_1 m_n = m_n m_1$, and this holds by the middle element. Hence $[m_1|\dots|m_n] \in (\bcom M)_n$.
By \cref{segality top bottom} and \cref{edgy segal characterization}, this is enough to guarantee that this simplicial set is lower 3-Segal.
In case $M = G$ is a group, $\bcom G$ is a symmetric set (see \cite[Example 1.11]{HackneyLynd:PGSSS}), so we have $\deg(\bcom G) \leq 2$, with equality if and only if $G$ is nonabelian.
It is possible (outside of the group case) for $\bcom M$ to be 2-Segal when $M$ is not commutative. 
For instance, if $M$ is freely generated by two idempotent elements, then $\bcom M$ is a Segal partial monoid, so is 2-Segal by \cite[Example 2.1]{BOORS:2SSWC}.
\end{example}

\subsection{Partial groupoids and starry words}\label{bs words}

In the previous subsection, we proved \cref{edgy segal characterization} which provided a characterization higher Segality for edgy simplicial sets.
We give a useful variation in this section which is valid in the case of spiny symmetric sets, and is based around \emph{starry words}.
Recall from \cref{bous spiny Segal} that a symmetric set $X$ is spiny if and only if $\bous \colon X \to \bswords X$ is a monomorphism.

\begin{notation}
When $X$ is a partial groupoid and $I \subset [n]$ is a subset not containing $0$, $\bswords_I(X)_n$ is the set
of starry words $w$ of length $n$ such that $d_i(w)$ is in the image of $\bous_{n-1}$ for each $i \in I$. 
Let $\bswords_k(X)_n$ be the union of $\bswords_I(X)_n$ as $I$ ranges over the all subsets of $[n]$ of cardinality $k+1$ which do not contain $0$.
\end{notation}

We have the following, easier variant of \cref{lem suspects as limit} in the symmetric case.

\begin{lemma}\label{lem starry-suspects as limit}
Let $X$ be a partial groupoid and $I \subset [n]$ a subset of cardinality at least two which does not contain $0$.
There is a bijection $\delta \colon \bswords_I(X)_n \to \lim X\cube{I}_{12}$ such that $\delta \circ \bous_n = \iota_{12}^*$.
\end{lemma}
\begin{proof}
Since $d_i d_j = d_{j-1} d_i$ in $\bswords(X)$ for $1\leq i <j$, a starry word $u \in \bswords_I(X)_n$ determines an element $(x_\bullet(u)) \in \lim X\cube{I}_{12}$ defined uniquely by the conditions $\bous_{n-1}x_i(u) = d_i(u)$ and by $\bous_{n-2}x_{ij}(u) = d_i d_j(u) \in X_{n-2}$ whenever $i < j$ are in $I$. 

The map $\delta\colon w \mapsto (x_\bullet(w))$ satisfies $\delta \circ \bous_n = \iota_{12}^*$ by construction. 
Given an element $(x_\bullet)$ of the limit, there is a unique starry word $u = (g_1, \dots, g_n)$ with $\delta u = (x_\bullet)$.
To find $g_m$, choose $i\in I$ distinct from $m$, and set $g_m$ to be the $m$\textsuperscript{th} entry of $\bous_{n-1}(x_i)$ if $m < i$ and the $(m{-}1)$\textsuperscript{st} entry of $\bous_{n-1}(x_i)$ if $m > i$.
Notice that $g_m$ does not depend on the choice of $i \neq m$.
\end{proof}

\begin{proposition}\label{bs words characterization}
Let $k\geq 1$ be an integer. A partial groupoid $X$ is lower $(2k{-}1)$-Segal if and only if $\bous_n \colon X_n \to \bswords_k(X)_n$ is surjective for all $n\geq k+1$.
\end{proposition}
\begin{proof}
By \cref{lem removal of gaps} and \cref{rmk non-gapped}, $X$ is lower $(2k{-}1)$-Segal if and only if $X\cube{I}$ is cartesian for each subset $I \subset [n]\setmin 0 \subset [n]$ of cardinality $k+1$.
Using \cref{lem starry-suspects as limit} in place of \cref{lem suspects as limit} in the proof of \cref{edgy segal characterization}, we see that $X$ is lower $(2k{-}1)$-Segal if and only if $\bous_n \colon X_n \to \bswords_k(X)_n$ is surjective for all $n\geq k+1$. 
\end{proof}

\section{Actions of partial groups}\label{actions}
In this section, we propose a notion of (partial) action of a partial group or partial groupoid on a set, by generalizing a formulation of group action based on the transporter groupoid. 
\subsection{Partial actions of groups}
When the partial group is a group, our definition below generalizes the existing notion of a partial action of a group on a set due to Exel \cite{Exel:PAGAIS}. 
We will see later that each partial action of a group gives rise to a partial group (\cref{ex realizable locality}).
\begin{definition}
A partial action of a group $G$ on a set $S$ is a partially defined function $\cdot \colon G \times S \pto S$ such that for all $g, h \in G$ and $x \in S$,
\begin{enumerate}
\item $1 \cdot x = x$, 
\item $h\cdot (g \cdot x)$ implies $(hg)\cdot x$ with equality if so, and
\item $g \cdot x$ implies $g^{-1} \cdot (g \cdot x)$. 
\end{enumerate}
\label{partial group action}
\end{definition}
Removal of condition (3) gives a definition of partial action of a monoid. 
The partial action is \emph{total} if it is an ordinary action, that is, if the action map is totally defined. 
By a result of Abadie and Kellendonk and Lawson \cite{Abadie:EATDPA,KellendonkLawson:PAG}, a partial action of a group is always \emph{globalizable}: it is the restriction of a total action of $G$ on a superset $\tilde{S} \supseteq S$. 

Total actions are equivalent to functors from $G$ to $\set$. 
More generally but still very classically, if $B$ is a category and $F\colon B \to \set$ is a functor,
then the Grothendieck construction produces a category $\int F$ with objects $(b,x)$ where $x \in F(b)$
and with morphisms $(b,x) \to (b',x')$ those $g \colon b \to b'$ in $B$ with $F(g)(x) = x'$, along with a functor $\int F \to B$. 
This determines a functor 
\[
	{\textstyle\int} \colon \fun(B,\set) \to \cat_{/B}
\]
where $\cat_{/B}$ is the overcategory whose objects are functors with codomain $B$, and morphisms are commutative triangles over $B$.
The essential image of $\int$ consists of the \emph{discrete opfibrations}, or \emph{star bijective functors}.
A functor $p\colon E \to B$ is star bijective if for each object $x \in E$ and each morphism $g \colon p(x) \to b$ in $B$, there exists a unique morphism $\tilde g$ in $E$ with domain $x$ such that $p(\tilde g) = g$:
\[ \begin{tikzcd}[sep=small]
x \dar[dotted, no head] & {} \ar[drr, phantom, "\rightsquigarrow"] & &  x \rar[ "\exists !"] \dar[dotted, no head]  & x' \dar[dotted, no head] \\
p(x) \rar & b & & p(x) \rar & b
\end{tikzcd} \]
In other words, for each object $x \in E$, the induced map $\hom_E(x,-) \to \hom_B(p(x),-)$ on stars is a bijection. 
If $B$ is a group (or more generally a groupoid), the domain $E$ of a star bijective functor $E\to B$ is automatically a groupoid, so the Grothedieck correspondence restricts to an equivalence between
$\fun(G,\set)$ and the full subcategory of $\gpd_{/G}$ on the star bijective functors. 

\subsection{Star injective maps of partial groupoids}\label{sec star injective}

Given a partial action of $G$ on $S$, there is still an associated transporter groupoid $\trans{S}{G}$ with object set $S$ and morphisms $x \xrightarrow{g} g \cdot x$ whenever $g \cdot x$ is defined, as well as a canonical functor $\trans{S}{G} \to G$. 
The partiality of the action corresponds to the functor $\trans{S}{G} \to G$ being merely injective on stars, one of the themes of \cite{KellendonkLawson:PAG} (see also \cite{PinedoMarin:PGASTS} for the groupoid case). 
Upon taking nerves, this motivates the following definitions. 

\begin{definition}
Let $X$ be a simplicial or symmetric set and $x \in X_0$ an object.
The \emph{star} at $x$ is the collection
\[
    \starset x = \starset_X x = \{ f \mid f \in X_n \text{ for some } n\geq 0 \text{ and } (d_\top)^n f = x\} \subseteq \coprod_{n=0}^\infty X_n.
\]
of all $n$-simplices (as $n\geq 0$ varies) emanating from $x$.
\end{definition}

Each star is evidently a graded set, and may be further be regarded as a presheaf over the subcategory of $\ord$ or of $\fin$ consisting of bottom-preserving maps, i.e.,\ those maps $\alpha \colon [n] \to [m]$ such that $\alpha(0) = 0$. 
Namely, for $x\in X_0$, the set $(\starset x)_n$ is the pullback
\[ \begin{tikzcd}
(\starset x)_n \rar[hook] \dar \ar[dr, phantom, "\lrcorner" very near start] & X_n \dar{d_\top^n} \\
\{ x \} \rar[hook] & X_0
\end{tikzcd} \]
and $\alpha^* \colon X_m \to X_n$ restricts to $(\starset x)_m \to (\starset x)_n$ for every bottom-preserving $\alpha$.

\begin{definition}[Star injective maps]\label{def star inj maps}
Let $E$ and $L$ be partial groupoids, or just edgy simplicial sets. 
A map $p \colon E \to L$ is \emph{star injective} if $\starset x \to \starset p(x)$ is a monomorphism for all $x\in E_0$.
\end{definition}

A star injective map $p \colon E \to L$ encodes something akin to a partial left action of $L$ on $S = E_0$ (see \cref{app action}).
Specifically, there are partial functions
\begin{align*}
	L_n \times S &\nrightarrow S \\
	(g, x) & \mapsto g \cdot x
\end{align*}
defined as follows.
If $n=0$, then $g\cdot x$ is defined if and only if $p(g) = x$, in which case $g\cdot x = x$.
For $n \geq 1$, the $n$-simplex $g = [g_1 | \dots | g_n] \in L_n$ acts on $x$ if and only if there is an $n$-simplex $\tilde g = [\tilde g_1 | \dots | \tilde g_n] \in E_n$ such that the source of  $\tilde g_1$ is  $x$.
In this case, $g \cdot x$ is the target of $\tilde g_n$, an element of $E_0$.
In general, the phrase ``$g \in L_n$ acts on $x \in E_0$'' will mean exactly that $g$ is in the image of the map $\starset x \to \starset p(x)$.
Expanded, ``$g$ acts on $x$'' means that there is a (unique) $\tilde g\in L_n$ such that $d_\top^n(\tilde g) = x$ and $p(\tilde g) = g$.
It is possible to write down axioms for this action analogous to those for a partial action of a group or a monoid and that characterize star injective maps as we do in \cref{app action}.
Hayashi independently introduced these axioms in \cite{Hayashi:PGSF}.
But none of this will be needed here; instead, we will adopt the following terminology. 

\begin{definition}\label{def partial action}
Let $L$ be a partial groupoid. 
By a \emph{partial action} of $L$ on a set $S$ we mean a star injective map of partial groupoids $\rho \colon E \to L$ such that $E_0 = S$.
The category of partial actions of $L$ (with $S$ varying) is the full subcategory of $\pgpd_{/L}$ on the star injective maps.
\label{partial action}
\end{definition}

\begin{remark}\label{rem partial action of group}
When $L$ is actually a group $G$, this definition is strictly more general than Exel's notion of a partial action of $G$ on $S$, which would require that $E$ be a groupoid, not merely a partial groupoid.
\end{remark}

We next look at two examples of actions of a partial group on itself. 

\begin{example}[The action on itself by multiplication]
Let $L$ be a partial group, and consider the map of partial groupoids $d_\bot \colon \ldec L \to L$.
We have $(\ldec L)_n = L_{n+1}$, so in particular the set being acted upon is $(\ldec L)_0 = L_1$.
This is a star injective map since if 
\[ [f | g_1 | \cdots | g_n],\, [f | h_1 | \cdots | h_n] \in (\ldec L)_n = L_{n+1} \]
have the same source $f$ and map to the same element $[g_1 | \cdots | g_n] = [h_1 | \cdots | h_n]$ in $L_n$, then they were equal by spininess of $\ldec L$. 
\end{example}

For the second example, it is convenient to have a symmetric version of the edgewise subdivision for simplicial objects \cite[\S1.9]{Waldhausen:AKTS}.

\begin{construction}[Edgewise subdivision]
Let $Q \colon \fin \to \fin$ denote the doubling endofunctor sending $[n]$ to
\[
  \{ n', \dots, 1', 0', 0, 1, \dots, n\} \cong \{0,1, \dots, 2n+1\} = [2n+1]
\]
where the bijection preserves this order (i.e.\ $i'\mapsto n-i$ and $i \mapsto i+n+1$).
A map $\alpha \colon [m] \to [n]$ is sent to the function operating as $i' \mapsto \alpha(i)'$ and $i\mapsto \alpha(i)$.
If $X$ is a symmetric set then its edgewise subdivision is defined to be $\tw(X) \coloneq XQ \colon \fin^\op \to \set$.
It has $n$-simplices $\tw(X)_n = X_{2n+1}$, 
so in particular the $1$-simplices are $3$-simplices of $X$.
The Segal map $\edgemap_{2n+1}$ for $X$ factors through the Segal map $\edgemap_n^{\tw}$ for $\tw(X)$, so spininess for $X$ implies spininess for $\tw(X)$.
In this case, we visualize an $n$-simplex as follows
\[
\tw(X)_n \ni \left( \begin{tikzcd}[sep=small]
  x_0 \rar{f_1} & x_1 \rar{f_2} & \cdots \rar{f_n} & x_n \\
  y_0 \uar{u} & y_1 \lar{g_1} & \cdots \lar{g_2} & y_n \lar{g_n}
\end{tikzcd} \right) = [g_n | \dots | g_1|  u | f_1 | \dots | f_n] \in X_{2n+1}
\]
with faces and degeneracies acting symmetrically.
\end{construction}

\begin{example}[The action on itself by conjugation]
Let $L$ be a partial group, and let $\tw(L)$ be its edgewise subdivision, which is a partial groupoid.
There is a map $(l,r)\colon \tw(L) \to L^{\op} \times L$, which is star injective by spininess and the definition of $\tw(L)$. 
Consider the pullback diagram
\[
\begin{tikzcd}
\conj(L) \rar \dar \arrow[dr, phantom, "\lrcorner", very near start] & L \dar{(\tau,\id)} \\
\tw(L) \rar{(l,r)} & L^{\op} \times L, 
\end{tikzcd}
\]
where $\tau$ is inversion. 
As a pullback of a star injective map, $\conj(L) \to L$ is star injective, and $\conj(L) \to L$ is the left conjugation action of $L$ on itself.  
\end{example}

\begin{lemma}
\label{star inj nondeg}
A star injective map between edgy simplicial sets or partial groupoids sends nondegenerate simplices to nondegenerate simplices. 
\end{lemma}
\begin{proof}
Star injective maps send nonidentities to nonidentities.
Let $\rho$ be a star injective map and $f \colon x\to x'$ an edge in the source of $\rho$.
If $\rho(f)$ is an identity, then we have $\rho(f) = \id_{\rho(x)} = \rho(\id_x)$, so $f = \id_x$.
The result for partial groupoids now follows from \cref{char nondeg}, and for edgy simplicial sets follows from the characterization that an $n$-simplex $z$ is degenerate if and only if $\edgemap_n(z)$ contains an identity among its entries. 
\end{proof}

\begin{proposition}\label{nondegen star injective}
Let $L$ be a partial groupoid, $g\in L_n$, and $\name{g} \colon \rep{n} \to L$ the classifying map for $g$.
Then $g$ is nondegenerate if and only if $\name{g} $ is star injective.
\end{proposition}
\begin{proof}
If $\name{g}$ is star injective, then since $g$ is the image of the nondegenerate simplex $\id_{[n]} \in \rep{n}_n$ under $\name{g}$, it is nondegenerate by \cref{star inj nondeg}. 
Conversely, if $\name{g}$ is not star injective, then $g_{ij} = \name{g}(\epsilon_{ij}) = \name{g}(\epsilon_{ik}) = g_{ik}$ for some $i, j, k$ with $j \neq k$, so $g$ is degenerate by \cref{char nondeg}.
\end{proof}

\subsection{Characteristic actions}
As mentioned above in \cref{rem partial action of group}, there are reasons to prefer star injective maps whose domain is a groupoid, rather than a partial groupoid.
A large number of examples of actions of interest will satisfy the following stronger condition, which will play a major role for the remainder of the paper.

\begin{definition}\label{def char map}
A star injective map $\rho \colon E \to L$ of partial groupoids is called \emph{characteristic} if $\rho$ is surjective and $E$ is a groupoid.
We also refer to $\rho$ as a \emph{characteristic action} of $L$ on $E_0$.
\end{definition}

A word $w = (g_1,\dots,g_n) \in \words(L)_n$ is said to act on some element $x \in E_0$ if there are $x = x_0,\dots,x_n \in E_0$ such that $g_i \cdot x_{i-1} = x_i$ for each $i$, that is, if $w$ is in the image of $\rho_1^{\times n} \colon \words(E)_n \to \words(L)_n$. 
If $\rho \colon E \to L$ is any star injective map, then $\rho$ being characteristic is equivalent to the following statement: a word $w$ determines a simplex of $L$ if and only if it acts successively on some $x \in E_0$. 
Thus, a characteristic map characterizes the composability of a word. 

Notice that if $\rho \colon E \to L$ and $\rho' \colon E' \to L$ are characteristic maps, then so is the induced map $E \amalg E' \to L$.
This is because star injectivity is fundamentally a local property, and this generalizes the evident action of a group on $S \amalg S'$ when we start with actions on $S$ and on $S'$.

Every partial groupoid $L$ admits at least one characteristic map:

\begin{example}[A canonical example]\label{canonical action}
Suppose $L$ is a partial groupoid.
For each $g \in L_n$, there is a corresponding classifying map $\name{g} \colon \rep{n} \to L$.
We saw in \cref{nondegen star injective} that $g$ is nondegenerate if and only if $\name{g}$ is star injective.
Adding these together we thus have a star injective map
\[
	\rho \colon \coprod_{n\geq 0} \coprod_{\nondeg L_n} \rep{n} \to L
\]
where $\nondeg L_n \subseteq L_n$ is the set of nondegenerate elements.
By construction this map is surjective on nondegenerate elements.
It follows that $\rho$ is surjective, hence is a characteristic map.
\end{example}

Our next example generalizes \cite[Example~2.4(2)]{Chermak:FSL} and is the underlying source of many important partial groups.

\begin{example}\label{ex realizable locality}
Suppose a group $G$ acts partially on a set $S$.
By taking the nerve of the map from the transporter groupoid described in \S\ref{sec star injective}, we have a star injective map
\[
  E \to BG.
\]
Here, $E_0 = S$ and the $n$-simplices of $E$ have the form 
\[ \begin{tikzcd}[sep=small, cramped]
x_0 \rar{g_1} & x_1 \rar{g_2} & \cdots \rar{g_n} & x_n
\end{tikzcd} \]
with $x_i \in S$, $g_i\in L_1$, and $g_i \cdot x_{i-1} = x_i$.
We let $L_S(G) \subseteq BG$ be the image of this map, which is a partial group (if nonempty). 
Then $E \to L_S(G)$ is a characteristic action.
Notice that the $n$-simplices of $L_S(G) \subseteq BG$ are precisely those $[g_1 | \cdots | g_n]$ that act on some $x \in S$, i.e., for which there exist elements $x = x_0, \dots, x_n \in S$ with $g_i \cdot x_{i-1} = x_i$. 
\end{example}

In fact, \cref{ex realizable locality} encompasses all partial groups embeddable in a group. 
\begin{theorem}\label{pgs in a group}
Suppose $L$ is a symmetric subset of $BG$ for some group $G$. 
If $\rho \colon E \to L$ is a characteristic map, then there is a partial action of $G$ on $E_0$ and an isomorphism $E \to N(\trans{E_0}{G})$ over $L$ that is the identity on $E_0$. 
\end{theorem}
\begin{proof}
The composite $E \to L \subseteq BG$ is a star injective map, which is then a star injective map between groupoids in the classical sense. 
According to \cite[Proposition 3.7]{KellendonkLawson:PAG} this corresponds to a partial action of $G$ on $E_0$.
The image of $N(\trans{E_0}{G}) \to BG$ is $L$, and $f \mapsto (d_1(f) \xrightarrow{\rho(f)} d_0(f))$ specifies an identity-on-objects isomorphism between $E$ and $N(\trans{E_0}{G})$ over $L$. 
\end{proof}

\begin{example}[Objective partial groups]\label{obj partial groups}
A important class of motivating examples of characteristic actions are given by objective partial groups, including localities.
Let $(M,\mathcal{O})$ be an objective partial group in the sense of Chermak \cite[Definition~2.6]{Chermak:FSL}. 
Here we use $M_1$ in place of $\mathcal{M}$ and $M$ in place of Chermak's domain $\mathbf{D}(\mathcal{M})$. 
Let $E$ be the nerve of the core groupoid of the associated transporter category \cite[Remark~2.8(1)]{Chermak:FSL}. 
This has $E_0 = \mathcal{O}$, a set of subgroups of $M$, and there is a star injective map
$\rho\colon E \to M$ that sends an $n$-simplex $X_0 \xrightarrow{f_1} X_1 \xrightarrow{f_2} \cdots \xrightarrow{f_n} X_n$ to $[f_1|f_2|\cdots|f_n]$. 
Since axiom (O1) for an objective partial group translates to ``a word is a simplex if and only if it acts on some $X \in \mathcal{O}$'',
this map is characteristic. 
\end{example}

\begin{remark}\label{rem edgy char action redux}
\Cref{def char map} also makes sense in the setting of edgy simplicial sets, by instead requesting that $E$ is the nerve of a category.
\Cref{ex realizable locality} can be imitated in the monoid case to produce interesting edgy simplicial sets equipped with characteristic actions, and other edgy simplicial sets like $\bcom M$ from \cref{ex bcom M} admit natural characteristic actions.
But not every edgy simplicial set does so.
For example, let $X$ be the 2-skeleton of $BM$, where $M$ is the smallest monoid containing a nonidentity idempotent $m$. 
If $p\colon E \to X$ is a characteristic map, then surjectivity implies there is a lift $[u|v]$ of $[m|m]$, which we write as 
  $a \xrightarrow{u} b \xrightarrow{v} c$.
Then $vu$ and $u$ are both sent by $p$ to $mm=m$, so star injectivity gives $vu=u$.
This implies $b=c$, and since $E$ is the nerve of a category it contains the 3-simplex $[u | v | v]$.
This is a contradiction since $X$ does not contain $[m|m|m]$.
\end{remark}

\section{The closure space associated to an action}\label{sec action to closure}

Suppose $\rho \colon E \to L$ is star injective.
Given $f\in L_n$, define the \emph{domain} of $f$, denoted $\dom{f} \subseteq E_0$, to be the set of elements of $E_0$ on which $f$ acts.
Passing to intersections of domains, this gives $E_0$ the structure of a closure space.

\begin{definition}
A \emph{closure operator} on a set $S$ is a monotone map $\cl \colon \ps(S) \to \ps(S)$ satisfying $A \subseteq \cl(A)$ and $\cl(\cl(A)) = \cl(A)$ for all $A\subseteq S$.
A set equipped with a closure operator is a \emph{closure space}.
Equivalently, it is the data of a collection of subsets of $S$ which is closed under arbitrary intersections.
A set $A \subseteq S$ is \emph{closed} if $\cl(A) = A$.
\end{definition}

\begin{definition}
\label{def closure operator}
Given a star injective map $\rho \colon E \to L$, the associated closure operator $\cl$ on $E_0$ is the smallest closure space containing all of the domains of $n$-simplices of $L$.
Explicitly, the closure of a set $A \subseteq E_0$ is 
\begin{equation}\label{eq closure operator def}
  \cl_\rho(A) = \cl(A) = \bigcap_{A \subseteq \dom{f}} \dom{f}
\end{equation}
where $f$ ranges over all $n$-simplices of $L$ (as $n$ varies).
In other words, $\cl(A)$ is the set of those $x\in E_0$ such that whenever an $n$ simplex $f \in L_n$ acts on each $a \in A$, then $f$ acts on $x$. 
\end{definition}

\begin{lemma}\label{lem 00 containment}
Suppose $f\in L_n$, and $\alpha \colon [m] \to [n]$ is a function with $\alpha(0) = 0$.
Then $\dom{f} \subseteq \dom{\alpha^*f}$, with equality when $\alpha$ is surjective.
\end{lemma}
\begin{proof}
Let $x\in \dom{f}$, and suppose $\tilde f$ is a lift of $f$ which starts at $x$.
Then $\alpha^*\tilde f$ is a lift of $\alpha^*f$ which starts at $x$.
Now suppose $\alpha$ is a surjection, and choose a section $\delta \colon [n] \to [m]$ of $\alpha$ with $\delta(0) = 0$.
By the first part, $\dom{\alpha^*f} \subseteq \dom{\delta^*(\alpha^*f)} = \dom{f}$.
\end{proof}

There is an alternative closure operator on $E_0$, denoted by $\cl_1$, where the building blocks are domains of $1$-simplices rather than domains of $n$-simplices.
One modifies \eqref{eq closure operator def} so that $\cl_1(A)$ is the intersection of $\dom{f}$ ranging over all $f\in L_1$ with $A \subseteq \dom{f}$.
The following convenient result lets us compare the two.

\begin{lemma}\label{bous domain lemma}
Suppose $E$ is a groupoid, $L$ is a partial groupoid, and $\rho \colon E \to L$ is a star injective map.
\begin{enumerate}[label=\textup{(\arabic*)}, ref=\arabic*]
\item 
\label{item g v intersection}
Let $g\in L_n$ with $\bous_n(g) = (g_1, \dots, g_n)$. Then $\dom{g} = \bigcap_{i=1}^n \dom{g_i}$.
\item 
\label{item nonempty}
Suppose $g_1, \dots, g_n \in L_1$.
If $\bigcap_{i=1}^n \dom{g_i}$ is nonempty, then there exists $g\in L_n$ such that $\bous_n(g) = (g_1, \dots, g_n)$.
\end{enumerate} 
\end{lemma}
\begin{proof}
If $x\in \bigcap_{i=1}^n \dom{g_i}$, then there are lifts $\tilde g_i \colon x\to y_i$ for $i=1, \dots, n$.
As $E$ is a groupoid, the Bousfield--Segal map $\bous \colon E \to \bswords E$ is an isomorphism by \cref{bous spiny Segal}, so there exists a unique $\tilde g$ with $\bous_n(\tilde g) = (\tilde g_1, \dots, \tilde g_n) \in \bswords(E)_n$.
As $\tilde g$ starts at $x$, it is a witness to the fact that $x\in \dom{\rho(\tilde g)}$.

Since $\bous_n \colon L_n \to \bswords(L)_n$ is injective, $\rho(\tilde g) = g$ in \eqref{item g v intersection}, so the preceding paragraph establishes the reverse inclusion in \eqref{item g v intersection}; the forward inclusion follows from \cref{lem 00 containment}.
The preceding paragraph also establishes \eqref{item nonempty}, using $g = \rho (\tilde g) \in L_n$.
\end{proof}

From \eqref{item g v intersection}, we immediately conclude the following.

\begin{corollary}\label{cap one simplices}
Suppose $E$ is a groupoid, $L$ is a partial groupoid, and $\rho \colon E \to L$ is a star injective map.
Then the closure operator $\cl$ is equal to the alternate closure operator $\cl_1$. 
Every closed set is the intersection of domains of 1-simplices; in case closed sets satisfy the descending chain condition, this is a finite intersection.
\qed
\end{corollary}

\Cref{bous domain lemma} is used in the following example related to \cref{ex realizable locality}.

\begin{example}\label{ex subset G set}
Suppose $\tilde S$ is a $G$-set, and $S\subset \tilde S$ is a subset.
If $f\in BG_n$ has matrix form $(f_{ij})$ (i.e.\ $\bous_n(f) = (f_{01}, \dots, f_{0n})$), then the domain of $f$ is 
\[
	\dom{f} = S \cap f_{10}(S) \cap \dots \cap f_{n0}(S).
\]
This is because $\dom{f_{0i}} =S\cap f_{0i}^{-1}(S)= S \cap f_{i0}(S)$. 
In the special case where $S$ is a Sylow $p$-subgroup of a finite group $G$ and the action is conjugation, the domain of $f$ is an intersection of Sylow subgroups and is typically denoted by $S_f$ \cite[p.\ 65]{Chermak:FSL}.
\end{example}

\begin{lemma}[Domains of identities] 
\label{domains of identities}
Let $\rho\colon E \to L$ be surjective, star injective map of partial groupoids. 
The domains of identities of $a \in L_0$ are distinct for distinct $a$ and form a partition of $E_0$. 
Every proper closed subset of $E_0$ is a subset of some $\dom{\id_a}$. 
\end{lemma}
\begin{proof}
Since $\dom{\id_a}$ is just the fiber $\rho^{-1}(a)$, the first statement follows from the surjectivity of $\rho$. 
By \cref{cap one simplices} a proper closed subset $C$ is the intersection of a nonempty collection of domains of $1$-simplicies. 
For any $f\colon a \to b$ so appearing, $C$ is a subset of $\dom{\id_a} = \dom{s_0d_1f}$ by \cref{lem 00 containment}.
\end{proof}

\begin{proposition}\label{prop empty not closed}
Suppose $G$ is a finite group and $\rho \colon E \to BG$ is a characteristic action.
In the associated closure space, the empty set is not closed.
\end{proposition}
\begin{proof}
Suppose $G$ has $n$ elements $g_1, \dots, g_n$.
Let $g\in BG_n$ be the unique element with $\bous_n(g) = (g_1, \dots, g_n)$.
Since $\rho$ is surjective, $\dom{g}$ is nonempty.
But $\dom{g} = \dom{g_1} \cap \cdots \cap \dom{g_n}$ by \cref{bous domain lemma}, and this is the minimal closed set by \cref{cap one simplices}.
\end{proof}

\begin{proposition}\label{prop empty closed}
Suppose $L$ is a partial groupoid but not a group and $\rho \colon E \to L$ is a characteristic action.
In the associated closure space, the empty set is closed.
\end{proposition}
\begin{proof}
If $L = \varnothing$, then $E$ is also empty and thus $\varnothing$ is the unique closed set.
If $L$ is not a group and not empty, then there is an integer $n\geq 2$ and an element $(g_1, \dots, g_n) \in L_1^{\times n}$ which is not in the image of $\bous_n$.
Then by \cref{bous domain lemma}\eqref{item nonempty}, the closed set $\dom{g_1} \cap \cdots \cap \dom{g_n}$ must be empty.
\end{proof}

The situation is more subtle for infinite groups, as the empty set may or may not be closed.
For example, the closure space associated to the identity map $G \to G$ has only a single closed set, which is not the empty set.
But each infinite group also admits a characteristic action with empty closed.

\begin{example}[Dissolution]
Suppose $G$ is a group, and let $\varphi \colon EG \to BG$ be the discrete opfibration associated with the left action of $G$ on itself by multiplication.
Unspooling, the objects of the groupoid $EG$ are the elements of $G$, and there is a unique morphism between any two objects.
For any $\sigma \in BG_n$, let $\tilde \sigma \in EG_n$ be the unique lift starting at the identity element $e\in G = EG_0$.
Let $E_\sigma \subseteq EG$ be the full subgroupoid spanned by the objects appearing along $\tilde \sigma$.
The groupoid $E_\sigma$ is finite.
The composite $E_\sigma \to EG \to BG$ is star injective, and $\sigma$ is in its image.
The star injective map 
\[
	\rho \colon \coprod_{n\geq 0} \coprod_{\sigma \in BG_n} E_\sigma \to BG
\]
is thus a characteristic action.
If $x\in \ob E_\sigma$, then since $E_\sigma$ is a finite groupoid, $x$ can be a member of $\dom{g}$ for at most finitely many $g\in BG_1 = G$.
If $G$ is infinite, then $\bigcap_{g\in G} \dom{g}$ is empty, hence $\varnothing$ is closed.
\end{example}

\section{Helly independence}
\label{sec helly independence}

In this section, we review the definitions and various characterizations of the Helly number of a closure space in the forms that are needed later.
All this is very well established, and we do not claim any originality. 
A textbook account is \cite[Ch.\ II]{vandevel:TCS}, and we adopt a 
definition of Helly core that agrees with Diognon--Reay--Sierksma \cite{DoignonReaySierksma:TGHNCS} and Jamison-Waldner \cite{Jamison-Waldner1981}. 

\subsection{Cores and the Helly number}
Throughout this section, 
\[
\text{$S$ is a closure space with closure operator $\cl$, and $\cl(\varnothing) = \varnothing$}. 
\]
Also, $\closeds \subseteq \subsets = \ps(S)$ is its lattice of closed subsets, and $\closeds'$ and $\subsets'$ denote nonempty subsets, for short. 
By a family of subsets of $S$, we will mean an indexed family $\ul A = (A_i)_{i \in I}$ throughout. 
In other words, a family is a function $\ul A \colon I \to \subsets$.
Its \emph{size} is the cardinality of $I$. 
A \emph{subfamily} of a family $\ul A$ is a family $\ul B$ of the form $\ul A|_J$ for some subset $J \subseteq I$. 
We use $J \subset_1 I$ to indicate that $J$ is a subset of $I$ with $J = I\setmin i$ for some $i \in I$, and use $J \subseteq_1 I$ to mean that $J \subset_1 I$ or $J = I$. 

\begin{definition}\label{def core independent}
Let $S$ be a closure space and $\ul A = (A_i)_{i \in I}$ a family of subsets of $S$. 
The \emph{core} of $\ul{A}$ is the subset
\[
	\core(\ul{A}) = \bigcap_{J \subseteq_1 I}  \cl\left(\bigcup_{j \in J} A_j\right).
\]
If $\ul A$ is nonempty, this is the same as
$\bigcap_{i \in I} \cl\left(\bigcup_{j \neq i} A_j\right).$ 
\end{definition}

\begin{example}\label{ex small cardinalities}
The cores of families of sizes $0$, $1$, $2$, and $3$ are $\cl(\varnothing)$, $\cl(\varnothing)$, $\cl(A_2) \cap \cl(A_1)$,
and $\cl(A_2 \cup A_3) \cap \cl(A_1 \cup A_3) \cap \cl(A_1 \cup A_2)$, respectively.
\end{example}

\begin{definition}
\label{helly independence}
A family $\ul A$ of nonempty subsets of $S$ is \emph{Helly independent} if $\core(\ul A) = \varnothing$, and \emph{Helly dependent} otherwise.
The \emph{Helly number} of $S$ is the maximal size $h(S) = h(\cl) = h(\closeds) \in \mathbb{N}$ of a finite Helly independent family, if this exists. 
If there are independent families of arbitrarily large finite size then we will indicate this by setting $h(S) = \infty$.
\end{definition}

\begin{remark}
\label{remark edge cases}
The core of a family is a closed subset and its definition makes sense in any closure space. 
In that generality, Helly independent families exist if and only if the empty set is closed, which explains the standing assumption.
Every family of subsets of $S$ of size $0$ or $1$ is independent by \cref{ex small cardinalities}.
The empty space $S = \varnothing$ is the unique closure space with Helly number $0$. 
\end{remark}

We next examine two basic monotonicity properties of $\core(-)$ and invariance under entrywise closure, which are valid in any closure space. 

\begin{lemma}[Monotonicity of cores]\label{core monotone}
Let $\ul A \in \subsets^I$ and $\ul B \in \subsets^K$ be families. 
\begin{enumerate}[label=\textup{(\arabic*)}, ref=\arabic*]
\item\label{subfamily core} If $\ul B$ is a subfamily of $\ul A$, then $\core(\ul B) \subseteq \core(\ul A)$. 
In particular, a subfamily of an independent family is independent. 
\item\label{subset core} If $I = K$ and $\ul B \leq \ul A$ in $\subsets^I$, then $\core(\ul B) \subseteq \core(\ul A)$. 
In particular, if $\ul B \leq \ul A$ and $\ul A$ is independent, then $\ul B$ is independent. 
\end{enumerate}
\end{lemma}
\begin{proof}
\eqref{subfamily core}: 
For each $J \subseteq_1 I$, the intersection $J \cap K \subseteq_1 K$ and 
\[
	\cl \left( \bigcup_{j \in J \cap K} B_j \right) \subseteq \cl \left( \bigcup_{j \in J} A_j \right). 
\]
This shows that each term in the intersection for $\core(\ul A)$ contains some such term in $\core(\ul B)$,
and thus $\core(\ul B) \subseteq \core(\ul A)$. 

\eqref{subset core}: Here both intersections run over the same $J \subseteq_1 I$, and by assumption $B_j \subseteq A_j$ for $j \in J$. So the inclusion of cores follows from the monotonicity of $\cl$. 
\end{proof}

\begin{lemma}[Invariance under closure] 
\label{clo core}
If $\ul A$ is any family and $\cl(\ul A) = (\cl(A_i))_{i \in I}$, then $\core(\cl(\ul A)) = \core(\ul A)$. 
\end{lemma}
\begin{proof}
This comes down to the equality $\cl(\bigcup A_j) = \cl(\bigcup \cl(A_j))$, a standard fact about closure operators. 
\end{proof}

\begin{corollary}
A Helly independent family has no duplicates. 
If there is a size $k$ Helly independent family, then there is a size $k$ Helly independent family of singletons, as well as a size $k$ Helly independent family of closed subsets. 
\label{cor clo monotonicity}
\end{corollary}
\begin{proof}
If $A_i$ and $A_j$ are two members of an independent family with $i \neq j$, then the corresponding two element subfamily is independent by \cref{core monotone}\eqref{subfamily core}, and so $\cl(A_i) \cap \cl(A_j) = \varnothing$.
The nonemptiness of the sets then implies $A_i \neq A_j$. 
The next statement follows from \cref{core monotone}\eqref{subset core} by selecting elements $a_i \in A_i$ for each $i \in I$,
while the last is a consequence of \cref{clo core}. 
\end{proof}

\begin{remark}
\label{def usual helly number}
If interested only in the Helly number itself, one could just work with subsets of $S$ instead of indexed families by \cref{cor clo monotonicity}. 
In this context (\cref{def core independent} restricted to indexed families of singletons with no duplicates), the core of a subset $A$ of $S$ is
\[
\core(A) = \bigcap_{B \subseteq_1 A} \cl(B) = \cl(A) \cap \bigcap_{a \in A} \cl(A \setmin a), 
\]
and the Helly number is the maximal size of a subset $A$ of $S$ with empty core.
\end{remark}

\subsection{Helly critical families and the classical Helly number} 
The original meaning of the Helly number is in a sense dual to \cref{helly independence}. 
Classically, the Helly number of a closure space $S$ is the smallest $h \in \mathbb{N}$ such that whenever $\ul A$ is a finite family of at least $h+1$ closed subsets of $S$ such that each subfamily of size $h$ has nonempty intersection, then $\ul A$ has nonempty intersection.
Families $\ul A$ of largest size that refute the above condition are called \emph{critical}. 
The following material is based on an interpretation of 
\cite[Lemma~3.1]{Sierksma1975}; the terminology is inspired by \cite{ConfortiDiSumma:MSFSHN}. 

We now restrict attention to the lattice $\closeds$ of closed sets. 
It will be convenient to use the notation $\vee$ for the closure of a union of closed sets, that is
\[
  \bigvee_{k\in K} A_k \coloneq \cl \left(\bigcup_{k\in K} A_k\right).
\]
We will also write $\wedge$ for intersection of closed sets.
These are just the usual join and meet operations in the complete lattice $\closeds$.

\begin{definition}\label{def helly crit}
A family $\underline{A} = (A_i)_{i\in I} \in \closeds^I$ of closed subsets is said to be \emph{Helly critical} if $I = \varnothing$, or 
\[
\bigwedge_{i\in I} A_i = \varnothing \text{\quad and\quad} \bigwedge_{j \in J} A_j \neq \varnothing \text{\quad for each } J \subset_1 I. 
\]
\end{definition}

By \cref{lem critical no repeated} below, there can be no containment between members of a critical family.
If $S$ is nonempty, then $(\varnothing)$ is the unique Helly critical family of size 1.
A family of nonempty closed sets of size 2 is Helly critical if and only if it is Helly independent (see \cref{ex small cardinalities}).
Independent and critical families are not the same in general, but are dual. 

\begin{definition}
Fix a nonempty indexing set $I$, and consider the lattice $\closeds^I$ of $I$-indexed families of closed sets.
Define maps 
\[
\sepun,\, \sepin \colon \closeds^I \to \closeds^I \text{\quad and \quad} \totin \colon \closeds^I \to \closeds
\]
by 
\[
\sepun(\ul A)_i = \bigvee_{j \in I\setmin i} A_j, \quad \sepin(\ul A)_i = \bigwedge_{j \in I \setmin i} A_j, \quad \text{and} \quad \totin(\ul A) = \bigwedge_{i \in I} A_i. 
\]
\end{definition}

In terms of these maps, the nonempty Helly independent families of closed sets are precisely those $\underline{A} \in (\closeds')^I \subseteq \closeds^I$ such that $\totin \sepun(\underline{A}) = \varnothing$.
The nonempty Helly critical families are those $\underline{A} \in \closeds^I$ with $\totin(\underline{A}) = \varnothing$ and $\sepin(\underline{A}) \in (\closeds')^I$. 

\begin{lemma}\label{union intersection adjoint}
The maps $\sepun, \sepin$, and $\totin$ are all monotone, and $\sepun$ is left adjoint to $\sepin$.
\end{lemma}
\begin{proof}
Monotonicity is clear.
To establish the adjunction $\sepun \dashv \sepin$, we show there are natural transformations $\id\Rightarrow \sepin\sepun$ and $\sepun \sepin \Rightarrow \id$.
As $A_i \leq \bigvee_{k \in I \setmin j} A_k$ for all $j \in I\setmin i$, 
\[
A_i \leq \bigwedge_{j \in I\setmin i}\,\, \bigvee_{k \in I \setmin j} A_k.
\]
But the right side is the $i$\textsuperscript{th} entry of $\sepin\sepun(\underline{A})$.
Hence $\underline{A} \leq \sepin\sepun(\underline{A})$.
Dually, since $\bigwedge_{k \in I \setmin j} A_k \leq A_i$ whenever $j \in I\setmin i$, we have 
\[  
\bigvee_{j \in I\setmin i}\, \bigwedge_{k \in I\setmin j} A_k \leq A_i.
\]
Thus $\sepun \sepin(\underline{A}) \leq \underline{A}$ for all $\underline{A} \in \closeds^I$.
\end{proof}

\begin{proposition}
\label{ind crit correspondence}
The map $\sepin$ takes Helly critical families to Helly independent families. 
The map $\sepun$ takes Helly independent families to Helly critical families.
\end{proposition}
\begin{proof}
Suppose $\underline{A}$ is Helly critical.
By definition, we have $\sepin(\underline{A}) \in (\closeds')^K$.
\Cref{union intersection adjoint} supplies a natural transformation $\totin \sepun \sepin \Rightarrow \totin \id_{\closeds^K} = \totin$. 
We thus have
\[
  \totin \sepun (\sepin (\underline{A})) \leq \totin(\underline{A}) = \varnothing
\]
since $\underline{A}$ is Helly critical.
Thus $\totin \sepun (\sepin (\underline{A})) = \varnothing$, and we conclude that $\sepin(\underline{A})$ is Helly independent.

Now suppose $\underline{A} \in (\closeds')^K$ is Helly independent.
We wish to show that $\totin(\sepun(\underline{A})) = \varnothing$ and $\sepin(\sepun(\underline{A})) \in (\closeds')^K$. 
The first of these is immediate since $\underline{A}$ is Helly independent.
But we also know $\underline{A} \leq \sepin \sepun(\underline{A})$ by \cref{union intersection adjoint}, so each entry of $\sepin \sepun(\underline{A})$ is nonempty.
Thus $\sepun(\underline{A})$ is Helly critical.
\end{proof}

\begin{theorem}\label{helly rank v helly critical}
A finite Helly number is equal to the size of the largest critical family.
\end{theorem}
\begin{proof}
\cref{ind crit correspondence} implies that there is a Helly critical $I$-indexed family if and only if there is a Helly independent $I$-indexed family of closed sets, which holds if and only if there is a Helly independent $I$-indexed family by \cref{cor clo monotonicity}.  
This proves the assertion if $h < \infty$, and if $h = \infty$, then there are Helly critical families of size $n$ for each $n \in \mathbb{N}$ for the same reasons. 
\end{proof}

The following basic lemmas on critical families will be needed in \cref{sec:degree_as_helly_rank}. 

\begin{lemma}\label{lem critical no repeated}
Let $\underline{A} = (A_i)_{i \in I}$ be a family. 
If $A_k \subseteq A_\ell$ for some $k \neq \ell$, then $\underline{A}$ is not Helly critical.
In particular, a Helly critical family does not have duplicate elements, and it does not have $S$ as a member. 
\end{lemma}
\begin{proof}
With $J = I\setmin \ell$, we have $k \in J$ but $\ell \notin J$ so 
\[
	\bigwedge_{j \in J} A_j = A_k \wedge \bigwedge_{j \in J} A_j \subseteq A_\ell \wedge \bigwedge_{j \in J} A_j = \bigwedge_{i\in I} A_i.
\]
The family cannot be Helly critical, for that would require the left-hand side to be nonempty and the right-hand side to be empty.

If $S$ were a member of a critical $\ul A$, then $\ul A = (S)$ and so $S = \varnothing$. 
But $(\varnothing)$ is in fact not critical in this case. 
\end{proof}

\begin{lemma}\label{lem critical subfamily}
Let $S$ be a nonempty closure space and $\underline{A}$ an $I$-indexed family of closed sets with $I$ finite. 
If $\bigwedge_{i\in I} A_i = \varnothing$, then $\underline{A}$ has a Helly critical subfamily, nonempty if $\ul A$ is. 
\end{lemma}
\begin{proof}
This uses induction on $|I| \geq 1$.
If $|I|=1$ then the assumption implies $\underline{A} = (\varnothing)$ is Helly critical.
Suppose $|I| \geq 2$ and assume $\bigwedge_{i\in I} A_i = \varnothing$.
If $\bigwedge_{j\in J} A_j$ is nonempty for each $J \subset_1 I$, then $\underline{A}$ is already Helly critical, and we are done.
Otherwise, fix a subset $J \subset_1 I$ with $\bigwedge_{j \in J} A_j = \varnothing$.
By induction, there exists $\varnothing \neq K \subseteq J$ such that $\ul A|_K$ is Helly critical.
\end{proof}

Recall that a subspace of a closure space $S$ is a subset $U$ with the subspace closure operator $A \mapsto \cl(A) \cap U$. 
Equivalently, a subset of $U$ is closed if it is the intersection of $U$ with a closed subset of $S$.
If $U$ is closed, then $h(U) \leq h(S)$ since closed sets in $U$ are closed sets in $S$.
We now consider a situation involving subsets which may not be closed.

\begin{lemma}
\label{helly domid general}
Let $\mathcal{U}$ be a pairwise disjoint nonempty collection of nonempty subspaces of $S$, such that every proper closed subset of $S$ is contained in some member of $\mathcal{U}$.  
If $\ul A = (A_i)_{i \in I}$ is a critical family of $S$, then  $\ul A$ is a critical family of closed subsets of $U$ for some $U \in \mathcal{U}$, or else $|I| = 2$. 
In particular, either 
\begin{enumerate}[label=\textup{(\alph*)}, ref=\alph*]
\item $h(S) = \sup_{U \in \mathcal{U}} h(U)$, or\label{item sup case}
\item $h(S) = 2$ and $h(U) = 1$ for all $U \in \mathcal{U}$. 
\end{enumerate}
\end{lemma}
\begin{proof}
Since $\ul A$ is critical, each $A_i$ is a proper subset of $S$ by \cref{lem critical no repeated}.
The statement thus holds when $|I| \leq 1$, so we may assume $|I| \geq 3$.
Use the assumption to choose for $i \in I$ some $U_i \in \mathcal{U}$ containing $A_i$. 
Since
\[ 
\varnothing \neq \bigwedge_{j \in J} A_j \subseteq \bigwedge_{j \in J} U_{j}. 
\]
for each $J \subset_1 I$, it follows from $|I| \geq 3$ and the pairwise disjointness of the $\mathcal{U}$ that all the $U_i$ are equal. 
\end{proof}

\begin{corollary}\label{helly critical domid}
If $\rho \colon E \to L$ is a characteristic map with $L$ nonempty, then either
\begin{enumerate}[label=\textup{(\alph*)}, ref=\alph*]
\item $h(\rho) = \sup_{a \in L_0} h(\dom{\id_a})$, or
\item $h(\rho) = 2$ and $h(\dom{\id_a}) = 1$ for all $a \in L_0$. 
\end{enumerate}
\end{corollary}
\begin{proof}
By \cref{domains of identities}, we may apply \cref{helly domid general} to $\mathcal{U} = \{ \dom{\id_a} \}$.
\end{proof}

\begin{remark}
\label{rmk about dom id}
If $\rho \colon E \to L$ is a characteristic map, then applying \cref{lem critical no repeated} to the closure space $\dom{\id_{a}}$, 
\cref{helly domid general} gives that no $\dom{\id_a}$ can appear in a critical family of closed subsets of $E_0$ of size $\geq 3$.
\end{remark}

\section{Degree as Helly number}\label{sec:degree_as_helly_rank}

When $\rho \colon E \to L$ is a characteristic map of partial groupoids with $\cl_\rho(\varnothing) = \varnothing$, we write $h(\rho)$ for the Helly number of the closure space $E_0$. 
Recall from \cref{prop empty closed} that if $L$ is not a group, then the empty set is closed.

\begin{theorem}
\label{helly degree theorem}
Let $\rho \colon E \to L$ be a characteristic map with $L$ not a groupoid. 
Then $\deg(L) \leq h(\rho)$. If $E_0$ satisfies the descending chain condition on closed subsets, then $h(\rho) \leq \deg(L)$.
\label{main intro}
\end{theorem}
\begin{proof}
By \cref{main rel domid} and \cref{helly critical domid}, $\deg(L) \leq h(\rho)$ with equality under the descending chain condition if $h(\rho) \neq 2$. 
But if $h(\rho) = 2$, then $h(\rho) \leq \deg(L)$ because $L$ is not a groupoid. 
\end{proof}

\begin{example}\label{ctexs}
If $L$ is a groupoid ($\deg(L) = 1$), then pretty much anything that can go wrong in \cref{main intro} does go wrong. 
\begin{enumerate}[leftmargin=*, label=\arabic*., ref=\arabic*] 
\item\label{gpd two obj ctex} If $L$ is a groupoid with distinct objects $a$ and $b$, then $h(\rho) = 2$. 
This is because the non-intersecting, nonempty closed sets $\rho^{-1}(a) = \dom{\id_a}$ and $\rho^{-1}(b) = \dom{\id_b}$ show $h(\rho) \geq 2$, while $h(\rho) > 2$ is not possible by \cref{main rel domid} and \cref{helly critical domid}. 
\item\label{fin grp ctex} If $L$ is a finite group, then $\varnothing$ is never closed by \cref{prop empty not closed} so $h(\rho)$ is not defined.
\item\label{inf grp ctex} If $L$ is an infinite group, then $h(\rho)$ may or may not be defined, depending on $\rho$.
The partial action of $\mathbb{Z}$ on $\mathbb{Z}\setmin 0$ by translation does give $h(\rho) = 1$, but in general we do not know the possibilities for the Helly number in this case.  
\item\label{empty ctex} If $L = \varnothing$, then the identity map is the unique characteristic map, and $h(\id_\varnothing) = 0$. 
\end{enumerate}
\end{example}

\begin{remark}
A version of the first inequality from \cref{helly degree theorem} also holds for characteristic maps $\rho \colon E \to X$ between edgy simplicial sets: 
if $k$ is the smallest positive integer such that $X$ is lower $2k$-Segal, then $k\leq h(\rho)$.
However, the proof of this stronger statement is much more involved, and not every edgy simplicial set admits a characteristic action (\cref{rem edgy char action redux}).
\end{remark}

Before filling out the proof of \cref{helly degree theorem}, we begin with a straightforward example of the degree of the reduction of a groupoid \cite[Example 5.5]{HackneyLynd:PGSSS}.
The reduction of a partial groupoid is discussed in \cref{degree of reduction}.

\begin{example}\label{ex groupoid reduction}
Suppose $E$ is a groupoid with more than one object and let $L = \mathcal{R}E$ be its reduction.
The canonical map $E \to L$ is a characteristic action.
If $g \colon x \to y$ is in $L_1\setmin\{\id\} = E_1\setmin s_0(E_0)$, then $\dom{g} = \{x\}$, while $\dom{\id} = E_0$.
It follows that the closed sets are the empty set, $E_0$, and the singletons that are sources of nonidentity morphisms of $E$. 
In particular, this closure space satisfies the descending chain condition.
Notice that $L$ is a group if and only if there is at most one object of $E$ that is the source of a nonidentity morphism.
If $L$ is not a group, and $x,y$ are distinct objects each of which is the source of a nonidentity morphism, then $(\{x\},\{y\})$ is a Helly critical family; there are no larger Helly critical families, so $2 = h(\rho) = \deg(L)$ by \cref{helly degree theorem}.
\end{example}

The following result says that if closed sets of $E_0$ satisfy the descending chain condition, a critical family realizing the Helly number $h(\rho)$ can always be replaced by a critical family of the same size whose members are domains of $1$-simplices.
It is the only source of this assumption on $E_0$. 

\begin{proposition}\label{critical 1 simplices}
Let $L$ be a partial groupoid and $\rho \colon E \to L$ be a characteristic map such that the collection of closed sets satisfies the descending chain condition.
Let $I$ be a nonempty finite set and $\underline{A}$ a Helly critical $I$-indexed family.
Then there is a finite set $M$, a surjection $\pi \colon M \to I$, and distinct $1$-simplices $g_m \in L_1$ such that $A_{\pi(m)} \subseteq \dom{g_m}$ for each $m\in M$ and $(\dom{g_m})_{m\in M}$ is Helly critical. 
If $|I| \geq 3$, the $g_m$ are never identities.
If $|I| = h(\rho)$, then $\pi$ is necessarily a bijection.
\end{proposition}
\begin{proof}
As before, the family $\underline{A}$ does not contain the maximal closed set $E_0$ by \cref{lem critical no repeated}. 
Thus, for each $i \in I$, the descending chain condition (via \cref{cap one simplices}) gives a \emph{nonempty} finite set $\{B_{i1}, \dots, B_{in_i} \}$ such that 
\[
	A_i = \bigwedge_{j = 1}^{n_i} B_{ij}
\]
with each $B_{ij}$ the domain of a 1-simplex.
Let $K \subseteq I \times \mathbb{N}$ be the set of pairs $(i,j)$ appearing, and consider the $K$-indexed family $\underline{B}=(B_k)_{k \in K}$. 
There is a surjective map $\pi_1 \colon K \to I$ sending $(i,j)$ to $i$. 
We have
\[
	\bigwedge_{k\in K} B_k = \bigwedge_{i\in I} A_i = \varnothing,
\]
so \cref{lem critical subfamily} guarantees a nonempty subset $M\subseteq K$ such that $\underline{B}|_M = (B_k)_{k\in M}$ is Helly critical. 

Write $\pi \colon M \to I$ for the restriction of $\pi_1$ to $M$. 
Then $\pi$ is surjective. 
Otherwise, choosing $J \subset_1 I$ containing the image of $\pi$, we would have
\[
	\varnothing \neq \bigwedge_{j \in J} A_j \subseteq \bigwedge_{k \in \pi^{-1}(J) = M} B_k = \varnothing
\]
by the criticality of both $\ul A$ and $\ul B|_M$, a contradiction. 

We know that the members of $\underline{B}|_M$ are distinct by \cref{lem critical no repeated}, hence the $g_m$ are distinct.
The statement that the $g_m$ are nonidentity elements when $|I| \geq 3$ follows from \cref{rmk about dom id}.
If $|I| = h(\rho)$, then $|M| \leq h(\rho)$ by \cref{helly rank v helly critical} and so $\pi$ is a bijection.
\end{proof}

\begin{remark}\label{rem dcc and h one}
If the closure space associated to a characteristic map $\rho \colon E \to L$ satisfies the descending chain condition, then $h(\rho) \neq 1$.
For if $h(\rho) = 1$, then \cref{critical 1 simplices} produces a $g\in L_1$ such that $(\dom{g})$ is critical.
This implies that $\dom{g} = \varnothing$, which is impossible since $\rho$ is surjective.
\end{remark}

\begin{theorem}
Let $\rho \colon E \to L$ be a characteristic map of nonempty partial groupoids with $\cl_\rho(\varnothing) = \varnothing$,
and set $h = \sup_{a \in L_0} h(\dom{\id_a})$. 
Then $\deg(L) \leq h$. 
If closed subsets of $E_0$ satisfy the descending chain condition, then $h \leq \deg(L)$. 
\label{main rel domid}
\end{theorem}

\begin{proof}
Fix $k \geq 1$ and assume first that $\deg(L)$ is greater than $k$, so that $L$ is not lower $(2k{-}1)$-Segal.
There exists, by \cref{bs words characterization}, an $n \in \mathbb{N}$, a subset $I \subset J = [n]\setmin 0 \subset [n]$ of cardinality $k+1$, and a word $(g_1, \dots, g_n) \in (\bswords_I (L)_n) \setmin \bous_n(L_n)$.
Let $a \in L_0$ be the common source of the $g_j$. 
Consider the $J$-family $\ul A$ with $A_j = \dom{g_j}$;
notice that $\bigcap_{j\in J} A_j = \varnothing$ and $\bigcap_{j\neq i} A_j \neq \varnothing$ for each $i\in I$.
If the restriction $\ul A|_I$ has empty intersection, then it is a critical family of size $k+1$. 
If the restriction $\ul A|_I$ has nonempty intersection, then the $(I \cup \{*\})$-indexed family $\ul B$ with $\ul B|_I = \ul A|_I$ and $B_* = \bigcap_{j\notin I} A_j$ is critical of size $k+2$.
In any case, $h(\dom{\id_a}) > k$ and hence $\deg(L) \leq h(\dom{\id_a}) \leq h$. 

Assume now that the Helly number of the domain of the identity of some $a \in L_0$ is greater than $k$, and that closed sets in $E_0$ satisfy the descending chain condition.
Let $I = [k+1]\setmin 0$ and fix a critical $I$-family $\ul A$ of $\dom{\id_a}$; by \cref{critical 1 simplices} we take $A_i = \dom{g_i}$ with $g_i \in L_1$.
Notice that all $g_i$ have source $a$.
By construction, the word $(g_1, \dots, g_{k+1})$ is in $\bswords_I(L)_{k+1}$, but not in the image of the Bousfield--Segal map. 
So $L$ is not $(2k{-}1)$-Segal by \cref{bs words characterization}, i.e.\ $k < \deg(L)$.
\end{proof}

\section{The finite dimensional case}\label{sec:the_finite_dimensional_case}

\begin{definition}\label{def skeleton}
Let $X$ be a symmetric set and $n\geq -1$. 
The \emph{$n$-skeleton} $\sk_n X \subseteq X$ is the smallest symmetric subset of $X$ containing $X_k$ for all $k\leq n$,
and $X$ is said to be \emph{$n$-skeletal} if $X = \sk_n X$. 
We say that $X$ is \emph{$n$-dimensional}, and write $\dim X =n$, if $X$ is $n$-skeletal but not $m$-skeletal for any $m < n$.
If $X$ is $n$-skeletal for some $n$, then $X$ is \emph{finite dimensional}.
\end{definition}
A symmetric set being $n$-skeletal rarely implies that its underlying simplicial set is $n$-skeletal.
In fact, a partial group with finitely many 1-simplices is finite dimensional by the following theorem from \cite{HackneyMolinier:DPG}.

\begin{theorem}\label{HM theorem}
Suppose $L$ is a partial group such that $L_1$ has cardinality $n+1$.
The dimension of $L$ is at most $n$, and is equal to $n$ just when $L$ is a group.
\end{theorem}

Some of our primary applications will be to finite partial groups, so we will often be working with finite dimensional symmetric sets.
It is immediate from \cref{HM theorem}, \cref{helly degree theorem}, and \cref{canonical action} that a finite partial group has finite degree.
In fact general finite dimensional partial groupoids also have finite degree, as we see below. 
This comes from the fact that finite dimensionality implies an absolute bound on the height of the poset of closed subsets in the closure space associated with a characteristic action. 

\begin{lemma}
\label{L skeletal implies E skeletal}
If $L$ is $n$-skeletal and $\rho \colon E \to L$ is star injective, then $E$ is $n$-skeletal. 
\end{lemma}
\begin{proof}
If $e$ is a nondegenerate $m$-simplex of $E$, then $\rho(e) \in L_m$ is nondegenerate by \cref{star inj nondeg}, and so 
$m \leq n$ because $L$ is $n$-skeletal. 
\end{proof}

\begin{proposition}
\label{dim int domains}
Let $\rho\colon E \to L$ be star injective map with $E$ a groupoid, and $g_1,\dots,g_m$ be distinct, nonidentity elements of $L_1$. 
If $m > \dim L$, then $\dom{g_1} \cap \cdots \cap \dom{g_m} = \varnothing$. 
\end{proposition}
\begin{proof}
This is immediate from \cref{bous domain lemma}\eqref{item nonempty}.
\end{proof}

\begin{proposition}
\label{fdim implies artinian}
Let $L$ be an $n$-skeletal partial groupoid and $\rho \colon E \to L$ a characteristic map.
A strictly increasing chain of closed subsets of $E_0$ can have length at most $n+2$ (that is, may have at most $n+3$ elements).
\end{proposition}
\begin{proof}
Let $C_m \subset \cdots \subset C_2 \subset C_1$ be a strictly increasing chain of closed subsets of $E_0$.
To start, assume that $C_1$ is a proper subset of $\dom{\id_a}$ for some $a \in L_0$, and also that $C_m \neq \varnothing$.
Using \cref{cap one simplices}, choose $1 \leq k_1 < \cdots < k_m$ and distinct edges $g_1,\dots,g_{k_m} \in L_1$ such that
\[
C_i = \dom{g_1} \cap \cdots \cap \dom{g_{k_i}}
\]
for each $1 \leq i \leq m$. 
Since $C_1$ is a proper subset of $\dom{\id_a}$, we may assume that no $g_j$ is an identity. 
By \cref{dim int domains}, $k_m \leq n = \dim(L)$, hence $m \leq n$.

We may extend our chain at the top by $C_1 \subset \dom{\id_a} \subseteq E_0$ where $g_1 \colon a \to b$, and at the bottom by $\varnothing \subseteq C_m$ if $\varnothing$ is closed.
Thus any strictly increasing chain has at most $n+3$ elements.
\end{proof}

\begin{theorem}
\label{deg dim}
If $L$ is a nonempty partial groupoid, then $\deg(L) \leq \dim(L)+1$.
\end{theorem}
\begin{proof}
We assume that $L$ is finite dimensional.
If $L$ is a (nonempty) groupoid, then $\deg(L) = 1 = 0 + 1 \leq \dim(L)+1$ and we are done, so we assume that $L$ is not a groupoid.
As $0$-dimensional symmetric sets are discrete groupoids, we have in particular assumed $\dim(L) \geq 1$, so we also can disregard the case $\deg(L) = 2$.
We thus assume $\deg(L) \geq 3$.
Fix any characteristic map $\rho\colon E \to L$. 
By \cref{fdim implies artinian}, the closure space satisfies the descending chain condition. 
This implies that $h(\rho) = \deg(L) \geq 3$ by \cref{helly degree theorem}. 
Second, it means \cref{critical 1 simplices} applies to give a Helly critical family $(\dom{g_i})_{i\in I}$ where $|I| = h(\rho) = h \geq 3$ and the $g_i$ are distinct nonidentity edges of $L$ (see also \cref{lem critical no repeated}).
The nonempty intersection $\varnothing \neq \bigcap_{j \neq i} \dom{g_j}$ of $h-1$ elements implies $h-1 \leq \dim(L)$ by \cref{dim int domains}, and hence $\deg(L) \leq \dim(L) + 1$.
\end{proof}

\begin{remark}
The bound given in \cref{deg dim} does not hold if $L$ is not assumed spiny.
Indeed, the symmetric sphere from \cref{ex sym sphere} is $n$-dimensional, but has degree equal to $2n$.
\end{remark}

\begin{example}[Degree of $\NA$]
\label{ex degree NA}
Recall the partial groupoid $\NA$ from \cref{sec NA}, which we will now observe has degree 3.
First, $\NA$ is $2$-dimensional so $\deg(\NA) \leq 3$ by \cref{deg dim}.
To show that $\deg(\NA) > 2$, we must show that $\NA$ is not 3-Segal.
Using the notation from \S\ref{sec NA} and \S\ref{sec edgy segal}, consider the word
\[
w = ((g \circ f)^{-1}, g \circ f, g^{-1}, g, h, h^{-1}) \in \words_I(\NA)_{6},
\]
where $I$ is the gapped subset $\{1, 3, 5\}$ of $[6]$. 
But $w$ is not a simplex, since otherwise $d_1d_{\top}d_{\bot}w = [f|g|h] \in \NA_3$ would force associativity.
Thus $\deg(\NA) > 2$ by \cref{edgy segal characterization}. 
\end{example}

In \cref{ex groupoid reduction}, we saw that the reduction of a groupoid often has degree 2, so in this case reduction generally increases the degree by $1$. 
At least under a finite dimensionality assumption, the degree is left unchanged by reduction in all other cases. 

If $L$ is a nonempty partial groupoid, then the reduction map $r \colon L \to \bar L = \mathcal{R}L$, $g \mapsto \bar g$ is surjective and star injective.
It is star injective because if $g, g' \in L_n$ have source $a \in L_0$ and $r(g) = r(g')$, then either $g = g'$, or $g$ and $g'$ are totally degenerate, hence $g = g'$ since they must be totally degenerate on $a$. It is surjective since $L$ is nonempty.
Given a characteristic map $\rho \colon E \to L$, the composite $\bar\rho = r \circ \rho \colon E \to \bar L$ is characteristic.
The resulting closure operators $\cl_\rho = \cl$ and $\cl_{\bar \rho} = \barcl$ on $E_0$ are closely related. 

\begin{lemma}
\label{reduction cl comparison} 
Suppose $L$ is not a groupoid.
A subset of $E_0$ is closed for $\cl$ if and only if it is closed for $\barcl$ or is of the form $\dom{\id_a}$ for some $a \in L_0$. 
\end{lemma}
\begin{proof}
Writing $\bardom(\bar{g})$ for the domain of  $\bar{g} \in \bar L_1$, we have $\dom{g} \subseteq \bardom(\bar{g})$, since any lift of $g$ is also a lift of $\bar{g}$. 
If $g$ is not an identity, then equality holds, since reduction of edges is injective when restricted to nonidentities.
The lemma now follows from \cref{cap one simplices}.
\end{proof}

\begin{theorem}
\label{degree of reduction}
If $L$ is a finite dimensional partial groupoid but not a groupoid, then $\deg(\bar L) = \deg(L)$. 
\end{theorem}
\begin{proof}
The reduction $\bar{L}$ is finite dimensional and not a groupoid, and \cref{fdim implies artinian} allows us to apply the results of \cref{sec:degree_as_helly_rank}. 
Apply \cref{main rel domid} for $L$ to express $\deg(L)$ as the supremum of $h(\dom{\id_a})$ over $a \in L_0$, and 
for $\bar L$ to get $\deg(\bar L) = h(\barcl)$. 
\cref{reduction cl comparison} implies that $\barcl$ satisfies the hypothesis of \cref{helly domid general} with $\mathcal{U} = \{\dom{\id_a} \mid a \in L_0\}$, and that $\dom{\id_a}$ has the same closed sets when regarded as a subspace of either of the two closure spaces. 
As $2 \leq \deg(L) = \sup_{a \in L_0} h(\dom{\id_a})$, we must be in case \eqref{item sup case} of \cref{helly domid general}, so $h(\barcl) = \sup_{a \in L_0} h(\dom{\id_a})$. 
\end{proof}

\section{Degree of punctured Weyl groups}\label{punctured weyl}

Throughout this section, $V$ is a real Euclidean space with inner product $(u,v)$, and $\rsys$ is a root system in $V$. 
This means that $\rsys$ is a finite spanning set of nonzero vectors such that for each root $\alpha \in \rsys$,
\begin{enumerate}[label=(R\arabic*), ref=R\arabic*]
\item\label{reduced} $\reals\alpha \cap \rsys = \{\alpha, -\alpha\}$, and 
\item\label{W invariant} $s_\alpha(\rsys) = \rsys$, where
$s_{\alpha}\colon v \mapsto v-2\alpha(v,\alpha)/(\alpha,\alpha)$ denotes the
reflection in the hyperplane orthogonal to $\alpha$. 
\end{enumerate}
These are the same conventions as in \cite[Chapter 1]{Humphreys1990}, except that we assume for convenience that $\rsys$ spans $V$. 
The Weyl group of $\rsys$ is the subgroup $W = W(\rsys)$ of the orthogonal group $O(V)$ generated by the $s_\alpha$ for $\alpha \in \rsys$. 
The root system is \emph{crystallographic} if $2(\beta,\alpha)/(\alpha,\alpha) \in \ints$ for all $\alpha,\beta \in \rsys$.
It is irreducible if is not the orthogonal union of two or more root systems. 

A positive system in $\rsys$ is a subset that is the set of positive elements in some compatible total ordering $<$ on $V$, that is, one that respects addition and scalar multiplication. 
If $H$ is a hyperplane containing no root and $v$ is a nonzero vector in the orthogonal complement of $H$, then the set $\{\alpha \in \rsys \mid (v,\alpha) > 0\}$ is a positive system, and all positive systems can be described in this way. 
A positive system is convex in $\rsys$, namely $\reals_{\geq 0}\pos \cap \rsys = \pos$, and $\rsys$ is the disjoint union of $\pos$ and $-\pos$ by \eqref{reduced}. 
These two properties characterize positive systems among subsets of $\rsys$ (cf. \cite[VI.1.7, Corollary 1]{BourbakiLie4-6} when $\rsys$ is crystallographic).
The Weyl group acts freely and transitively on the set of positive systems \cite[Section~1.8]{Humphreys1990}.
If $\pos$ is a positive system, then so is its negative $-\pos$.

A base $\base$ is a subset of $\rsys$ that forms a basis for $V$ and has the property that every root is a linear combination of 
elements of $\base$ with all coefficients nonnegative or all coefficients nonpositive. 
If $\pos$ is a positive system, then the set of roots in $\pos$ not expressible as a linear combination of two or more roots in $\pos$ with strictly positive coefficients is a base.
Conversely, given a base, the roots in the convex cone spanned by it is a positive system. 
In this way bases and positive systems determine each other uniquely. 
Elements of a base are called simple roots, and generally come with a fixed ordering $\{\alpha_1,\dots,\alpha_n\}$. 
The Weyl group is generated by the simple reflections $s_i = s_{\alpha_i}$ \cite[Section~1.5]{Humphreys1990}.\footnote{The simple reflections $s_i$ should not be confused with the simplicial degeneracy operators $s_i$ introduced in \cref{sec simp symm}. Degeneracies will not appear again in this paper.} 
The inversion set $\invset(w) = \pos \cap w^{-1}(-\pos)$ is the set of positive roots sent by $w$ to negative roots, and its cardinality is equal to the length of $w$, i.e. the number of simple reflections appearing a reduced expression for $w$ \cite[Corollary 1.7]{Humphreys1990}. 
The longest element $w_0$ of $W$ (with respect to $\pos$) is characterized by $w_0(\pos) = -\pos$. 

\subsection{Punctured Weyl groups}

Let $\rsys$ be a root system with fixed positive system $\pos \subset \rsys$. 
The Weyl group $W$ acts partially on the set of positive roots, and we may apply \cref{ex realizable locality} in this case.
The partial action determines a transporter groupoid with nerve $E = E_{\pos}(W)$ and a star injective map of partial groupoids $E \to BW$. 
Let $L_{\pos}(W)$ be its image, and $\rho \colon E \to L_{\pos}(W)$ the induced characteristic map (\cref{def char map}).  

\begin{definition}
\label{pun Weyl}
$L_{\pos}(W)$ is the \emph{punctured Weyl group} of $\rsys$. 
\end{definition}

This is a combinatorial analogue of the $p$-local punctured groups of \cite{HenkeLibmanLynd2023}.
Since any two positive systems are $W$-conjugate, it is essentially unique as a symmetric subset of $BW$ in the sense that any two are in the same $W$-orbit under conjugation. 

Fix $W$ and $\pos$ and set $L = L_{\pos}(W)$ for short. 
An element $w \in W$ determines a $1$-simplex in $L$ if and only if there is some positive root $\alpha$ such that $w(\alpha)$ is also positive.
Equivalently, this is the case if the set of such $\alpha$'s, 
\[
\dom{w} = \pos \cap w^{-1}(\pos) = \pos\backslash \invset(w),
\]
the complement in $\pos$ of the inversion set $\invset(w) = \pos \cap w^{-1}(-\pos)$, is nonempty.
Thus, 
\[
L_1 = W\backslash\{w_0\}, 
\]
which explains our usage of `punctured' in \cref{pun Weyl}. 

In general a word $w = (w_1,\dots,w_n) \in BW_n$ of non-$w_0$ elements is an $n$-simplex of $L$ if there is some positive root
$\alpha$ such that the word successively acts on $\alpha$, that is, such that each of the $n$ roots 
\[
w_{01}(\alpha), w_{02}(\alpha), \dots, w_{0n}(\alpha)
\]
is positive, where as usual $w_{ij} = w_j \cdots w_{i+1}$ if $i \leq j$, and $w_{ji} = w_{ij}^{-1}$. 
Again, this is equivalent to say that the domain
\begin{align*}
\dom{w} &= \dom{w_{01}} \cap \dom{w_{02}} \cap \cdots \cap \dom{w_{0n}}\\
        &= \pos \cap w_{10}(\pos) \cap w_{20}(\pos) \cdots \cap w_{n0}(\pos),
\end{align*}
an intersection of positive systems including $\pos$, is nonempty.
As we saw in \cref{ex subset G set}, the above is typical of any partial action of a group $G$ on a subset of a $G$-set.

\begin{theorem}
\label{weyl main}
The degree of $L_{\pos}(W)$ is given in Table~\ref{hrank table} below for $\rsys$ irreducible. 
If $\rsys$ decomposes as an orthogonal union of root systems $\rsys_i$, then the degree of $L_{\pos}(W)$ is the sum of the degrees of the $L_{\pos_i}(W_i)$ for $W_i$ the corresponding direct factor of $W$. 
\end{theorem}
\begin{proof}
If $\rsys$ is $A_1$ then the punctured Weyl group is the trivial group, which has degree $1 = \lfloor \frac{(1+1)^2}{4} \rfloor$.
If the rank is larger than 1, then combine \cref{helly degree theorem} with \cref{hrank reducible} and \cref{hrank rsys} below.
\end{proof}

For orientation, we describe what this says combinatorially in the following two examples. 

\begin{example}\label{A2 example degree}
According to \cref{hrank table}, the punctured Weyl group of a rank $2$ root system is lower $3$-Segal (but not $1$-Segal). 
Recall from \cref{d seg embed in BG} that the first of the $3$-Segal conditions in a partial group has to do with words of length $m = 4$:
if $w = (w_1,w_2,w_3,w_4)$ is a list of elements in $W\backslash\{w_0\}$, and if for each of the three faces 
\[
d_0w = [w_2|w_3|w_4], \quad d_2w = [w_1|w_3w_2|w_4], \quad d_4w = [w_1|w_2|w_3]
\]
there is some positive root that the word successively keeps positive, then $w$ itself successively keeps some positive root positive, that is,
$[w_1|w_2|w_3|w_4] \in L_4$. 
To see that it is not lower $1$-Segal, one just needs to produce a list of non-longest elements that does not act on a positive root. 
For example, if $\base = \{\alpha, \beta\}$ is a base, then $(w_0s_\beta, s_\alpha s_\beta) = (w_0s_\beta, w_0s_\alpha (w_0s_\beta)^{-1})$ has this property: $w_0 s_\beta$ acts only on $\beta \in \pos$, but $w_0s_\alpha (w_0s_\beta)^{-1}$ sends $w_0 s_\beta (\beta)$ to the negative root $w_0s_\alpha(\beta)$.
\end{example}

\begin{example}\label{c3 example}
For a more involved but illustrative example (cf.\ Figure~\ref{fig:c3 free set} below), take $\rsys = C_3$. 
Fixing an orthonormal basis $\{a_1,a_2,a_3\}$ for $\reals^3$, we take as usual $\base = \{\alpha_1,\alpha_2,\alpha_3\}$ with $\alpha_1 = a_1-a_2$, $\alpha_2 = a_2-a_3$, and $\alpha_3 = 2a_3$. 
The punctured Weyl group has degree $4$. 
It is lower $7$-Segal, the lowest condition of which says that if $w = (w_1,\dots,w_8)$ is a word of length $8$ such that if each of the five faces $d_0w,d_2w,d_4w,d_6w,d_8w$
act on a positive root, then so does $w$.
On the other hand, it is not lower $5$-Segal.
For example, the word
\begin{equation}
\label{c3 word}
w = (\ko{s_3,s_3},s_2,s_3,\ko{s_2,s_2},s_3,s_1,s_3,\ko{s_2,s_2},s_3,s_2,s_1,s_3,\ko{s_2})
\end{equation}
of length 16 has the property that the four faces $d_1w, d_5w, d_{10}w,
d_{16}w$ act on a positive root, but $w$ doesn't.
The domains of the four faces of $w$ are correspondingly $\dom{d_1w} = \{\alpha_3\}$, $\dom{d_5w} = \{\alpha_2+\alpha_3\}$, $\dom{d_{10}w} = \{\alpha_1+2\alpha_2+\alpha_3\}$, and $\dom{d_{16}w} = \{\alpha_1+\alpha_2+\alpha_3\}$, whereas $\dom{w}= \varnothing$. 
\end{example}

\subsection{Closure operators on positive systems and convex geometries}

A \emph{convex geometry} in the sense of Edelman and Jamison \cite{EdelmanJamison1985} is a finite closure space $(S,\cl)$ satisfying the antiexchange condition: 
if $C \subseteq S$ is closed and $x$ and $y$ are distinct points not in $C$, then $y \in \cl(C \cup x)$ implies $x \notin \cl(C \cup y)$.
The prototypical example is a finite subset $S$ of affine space with $S$-relative convex hull $A \mapsto S \cap \conv(A)$. 

We let $\cone_\reals$ denote the ($\pos$-relative) convex cone, whose value on a subset $A$ of a root system $\rsys$ is 
$\cone_\reals(A) = \reals_{\geq 0}A \cap \pos$,
the roots that are linear combinations of the vectors in $A$ with nonnegative coefficients. 
The subset $A$ is \emph{convex} if $\cone_\reals(A) = A$. 
If $\rsys$ is crystallographic, 
the $\ints$-\emph{closure} of a subset $A$ of $\pos$ is the set 
\[
\cone_{\ints}(A) = \ints_{\geq 0}A \cap \pos,
\]
those roots of the form $\sum_{\alpha \in A} c_\alpha\alpha$ with $c_\alpha \in \mathbb{Z}_{\geq 0}$. 
A subset $A$ of $\pos$ is $\mathbb{Z}$-closed if and only if it is closed in the sense of Bourbaki \cite[VI.1.7, Definition 4]{BourbakiLie4-6}: whenever $\alpha, \beta \in A$ are such that $\alpha+\beta$ is a root, we have $\alpha+\beta \in A$. See, for example, \cite[\S 2, Lemma]{Pilkington2006}.

In \cite{Pilkington2006}, Pilkington makes a comparison of various closure operators defined on root systems, including convex and $\ints$-closure and decides in particular when they coincide. 
While $(\pos, \cone_\reals)$ is obviously a convex geometry, she shows the same for $(\pos,\cone_\ints)$. 
\begin{proposition}[Pilkington]
\label{Z-closure convex geometry}
If $\rsys$ is finite crystallographic, then the closure space $\pos$ with $\cone_\mathbb{Z}$ is a convex geometry. 
\end{proposition}

Recall that the defining characteristic map $\rho \colon E \to L_{\pos}(W)$ for a punctured Weyl group gives rise to a closure operator on the set $E_0 = \pos$ in which the closed subsets are intersections of positive systems. 

\begin{proposition}
\label{closure op weyl}
Suppose $\rsys$ has rank at least 2.
A subset of $E_0 = \pos$ is closed for the $L$-action if and only if it is convex, i.e.,
$\cl_\rho = \cone_{\reals}$.
\end{proposition}
\begin{proof}
Let $A$ be a subset of $\pos$.
If $\alpha \in \cone_{\reals}(A)$ and a simplex $w$ of $L$ acts on each root in $A$, then $w$ acts on $\alpha$ since the $W$-action is $\reals$-linear, and so $\cone_{\reals}(A) \subseteq \cl_\rho(A)$. 

Conversely, let $\alpha$ be a positive root not in $\cone_{\reals}(A)$. 
Fix a hyperplane $H$ of $V$ strictly separating $\alpha$ and $\cone_{\reals}(A)$ and containing no root \cite[1.2.4 Theorem]{Matousek2002}. 
Let $v$ be a nonzero vector in $H^\perp$ on the side of $\cone_{\reals}(A)$. 
This determines a second positive system $\pos_1 = \{\beta \in \rsys \mid (\beta,v) > 0\}$ of $\rsys$ that contains $\cone_{\reals}(A)$ but not $\alpha$. 
Since $W$ acts transitively on positive systems, there is an element $w \in W$ such that $w(\pos_1) = \pos$. 
Thus, $\cone_{\reals}(A) \subseteq \dom{w} = \pos \cap w^{-1}(\pos)$ but $\alpha \notin \dom{w}$, and this shows $\alpha \notin \cl_\rho(A)$. 
\end{proof}

\subsection{Helly number of a convex geometry and abelian sets of roots}

The relevance of the antiexchange condition for the closure space $\pos$ with $\cl_{\rho}$ given by \cref{closure op weyl} comes from the following theorem of Hoffman and  Jamison-Waldner on the Helly number of a convex geometry \cite[Theorem~7]{Jamison-Waldner1981}, \cite[Proposition~3]{Hoffman:BCHN}. 
\begin{theorem}
\label{hrank convex geometry}
The Helly number of a convex geometry is the maximal size of a free subset. 
\end{theorem}
A subset $A$ of a closure space is \emph{free} if every subset of $A$ is closed.
This is equivalent to the condition that the sets $A$ and $A\setmin a$ are closed (for each $a\in A$).
Note that when $A$ is closed, the set $A\setmin a$ is closed just when $a\notin \cl(A\setmin a)$; thus a subset is free if and only if it is closed and Helly independent (\cref{def usual helly number}).

Suppose for the moment that $\rsys$ is crystallographic.

\begin{definition}
\label{abelian def}
A subset $A$ of a crystallographic root system $\rsys$ is \emph{abelian} if the sum of two roots in $A$ is never a root.
\end{definition}

\begin{proposition}
\label{helly free Z-closure}
A subset $A$ of positive roots is abelian if and only if it is free for $\cone_{\mathbb{Z}}$. 
The Helly number of $\pos$ with respect to $\mathbb{Z}$-closure is the maximal size of an abelian subset of $\pos$. 
\end{proposition}
\begin{proof}
We need to show that a subset of positive roots is abelian if and only if each subset of it is $\ints$-closed. 
From the definitions, a subset of an abelian set is abelian, and each abelian set is $\ints$-closed. 
Conversely, if $A$ is not abelian, then there is a subset of $A$ of size two that is not $\ints$-closed. 
The second statement now follows from \cref{hrank convex geometry}. 
\end{proof}

The maximal size of an abelian set of positive roots was computed by Malcev \cite{Malcev1945} (see \cite{Malcev1962} for an English translation).
It is the same as the maximal dimension of an abelian subalgebra of the associated complex semisimple Lie algebra. 

\begin{corollary}
\label{helly number Z closure}
For an irreducible crystallographic root system $\rsys$, the Helly number of $\pos$ with respect to $\ints$-closure is given in \cref{hrank Z table}. 
\begin{table}[ht]
{
\renewcommand{\arraystretch}{1.3}
\setlength{\tabcolsep}{12pt}
\centering
\caption{Helly number for $\ints$-closure}\label{hrank Z table}
\begin{tabular}{@{}lcclc@{}}\toprule
$\rsys$ & $h_{\ints}(\pos)$ && $\rsys$ & $h_{\ints}(\pos)$\\
\cmidrule(r){1-2} \cmidrule(l){4-5}
$A_n$ & $\lfloor \frac{(n+1)^2}{4} \rfloor$ && $B_n$ & 
$\begin{cases}
\,\,2n-1 & n \leq 3\\
\,\binom{n}{2}+1, & n \geq 4
\end{cases}
$
\\\addlinespace
$D_n$ & $\binom{n}{2}$ && $C_n$ & $\binom{n+1}{2}$ \\
$E_6$ & $16$ && $F_4$ & $9$   \\
$E_7$ & $27$ && $G_2$& $3$  \\
$E_8$ & $36$ &&  &  \\
\bottomrule
\end{tabular}
}
\end{table}
\end{corollary}

We now return to general root systems, and consider the convex analogue of an abelian subset.

\begin{definition}
A \emph{really abelian} subset of positive roots is a free subset of $\pos$ for convex closure.
\label{really abelian def}
\end{definition}
Visually, a subset of $\pos$ is really abelian if the rays determined by the roots in this set are precisely the extremal rays of the cone that it spans. 
This is determinable projectively: if we cut by an affine hyperplane $V_1$ passing through the simple roots and replace each positive root $\alpha$ by the unique point in the intersection of $V_1$ with the ray $\reals_{\geq 0}\alpha$, then $\cone_{\reals}$ becomes relative convex hull for the image of $\pos$ in $V_1$,
and a really abelian set is then one whose image is precisely the set of extreme points of its convex hull. 

In a rank 2 root system, a really abelian subset of maximal size is a set of two adjacent roots. 
\cref{fig:c3 free set} is an affine picture of the really abelian (blue) subset of size $4$ that appeared in \cref{c3 example}, which realizes the maximal size of a really abelian subset for $C_3$.
By contrast, the unique maximal abelian set of roots  also includes $2a_1$ and $2a_2$.
\begin{figure}[ht]
\centering
\begin{tikzpicture}[scale=5]
    \fill[blue!10] (0,0) -- (0,{2/3/2}) -- ({2/5/2},{2/5/2 + 2/3/2}) -- ({2/5/2},{2/5/2}) -- cycle;
    \draw[blue, dotted, thin] (0,0) -- (0,{2/3/2}) -- ({2/5/2},{2/5/2 + 2/3/2}) -- ({2/5/2},{2/5/2}) -- cycle;
    \filldraw[blue] (0,0) circle (0.2pt) node[left] {$2a_3$}; 
    \filldraw[red] (0,2/3) circle (0.2pt) node[left] {$2a_2$};
    \filldraw[red] (2/5,2/5) circle (0.2pt) node[right] {$2a_1$};
    \filldraw[red] (0,1) circle (0.2pt) node[left] {$a_2-a_3$}; 
    \filldraw[red] (1,0) circle (0.2pt) node[left] {$a_1-a_2$};
    \filldraw[red] (1/2,1/2) circle (0.2pt) node[right] {$a_1-a_3$};
    \filldraw[blue] ({2/5/2}, {2/5/2 + 2/3/2}) circle (0.2pt) node[above right] {$a_1 + a_2$};
    \filldraw[blue] ({2/5/2}, {2/5/2}) circle (0.2pt) node[right] {$a_1 + a_3$};
    \filldraw[blue] (0, {2/3/2}) circle (0.2pt) node[left] {$a_2 + a_3$}; 
\end{tikzpicture}
\caption{A really abelian subset of size 4 in $C_3$}
\label{fig:c3 free set}
\end{figure}
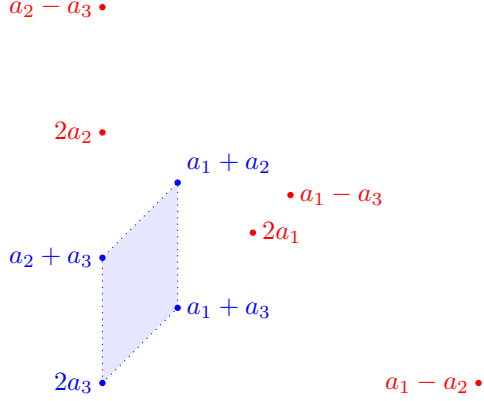

The following two basic lemmas will be used when considering $H_4$ in the proof of \cref{hrank rsys}.

\begin{lemma}
\label{constricting ra}
Let $A$ be a really abelian subset of $\pos$ containing the distinct roots $\alpha, \beta$. 
Then $\alpha$ and $\beta$ are adjacent roots in the rank 2 positive subsystem $\pos \cap \realspan(\alpha,\beta)$.
\end{lemma}
\begin{proof}
A subset of two positive roots in a rank $2$ root system is really abelian if and only if the two roots are adjacent.
So this just comes from the fact that $A \cap \realspan(\alpha,\beta)$ is really abelian in $\pos \cap \realspan(\alpha,\beta)$. 
\end{proof}

\begin{lemma}
\label{ra simple root}
If $A$ is a nonempty subset of $\pos$, then some $W$-conjugate of $A$ in $\pos$ contains a simple root. 
\end{lemma}
\begin{proof}
By induction on the minimal height $\mathrm{ht}(\alpha) = \sum_{\beta \in \base} c_{\beta}$ of a root $\alpha = \sum_{\beta \in \base}c_{\beta}\beta$ in $A$. 
Fix $\alpha$ of minimal height in $A$ and let $\beta \in \base$ be such that $(\alpha,\beta) > 0$ (which exists since $(\alpha,\alpha) > 0$). 
As $\beta \notin A$, the conjugate $s_\beta(A)$ is a subset of $\pos$ containing $s_\beta(\alpha)$ of strictly smaller height. 
By induction, some $W$-conjugate of $s_\beta(A)$ contains a simple root and so the same is true of $A$. 
\end{proof}

\subsection{The maximal size of a really abelian set}
In this section we compute the maximal size of a really abelian subset of positive roots in a finite root system and thus the Helly number for convex closure. 

Write $h_{\reals}(\pos)$ and $h_{\ints}(\pos)$ for the Helly number with respect to convex and $\ints$-closure, respectively (the latter only when $\rsys$ is crystallographic). 
From the definitions, there is a comparison $\cone_\ints \subseteq \cone_{\reals}$ of closure operators,
and thus an inequality of Helly numbers $h_{\reals}(\pos) \leq h_{\ints}(\pos)$. 
More to the point, a really abelian set is abelian. 
\begin{lemma}
\label{hrank reducible}
Let $\rsys$ be a finite root system with irreducible components $\rsys_1$, \dots, $\rsys_r$. 
Then the Helly number (with respect to either $\cone_\ints$ or $\cone_\reals$) of $\rsys^+$ is equal to the sum of the Helly numbers of the $\rsys_i^+$. 
\end{lemma}
\begin{proof}
Let $\rsys$ be a root system in $V$; it is enough to consider the case of a decomposition $\rsys = \rsys_1 \amalg \rsys_2$ and an  orthogonal decomposition $V = V_1 \oplus V_2$, with $\rsys_i$ a root system in $V_i$.
If $A_i \subseteq \pos_i$ for $i=1,2$, then $\cone(A_1\amalg A_2) \subseteq \pos$ is the disjoint union of $\cone(A_1) \subseteq \pos_1$ and $\cone(A_2) \subseteq \pos_2$, hence this is a disjoint union closure space \cite[\S I.1.14]{vandevel:TCS}.
But for such a closure space the Helly number is given by the sum of the Helly numbers \cite[Ch.\ II, Theorem 3.3]{vandevel:TCS}.
\end{proof}

\begin{theorem}
\label{hrank rsys}
Let $\rsys$ be an irreducible root system and $\pos$ a choice of positive system.  
The maximal size of a really abelian subset of $\pos$, and thus the Helly number of $\pos$ for convex closure, is given in \cref{hrank table}. 

\begin{table}[ht]
{
\renewcommand{\arraystretch}{1.3}
\setlength{\tabcolsep}{12pt}
\centering
\caption{Maximal sizes of really abelian sets of positive roots}\label{hrank table}
\begin{tabular}{@{}lcclc@{}}\toprule
$\rsys$ & $h_{\reals}(\pos)$ && $\rsys$ & $h_{\reals}(\pos)$\\
\cmidrule(r){1-2} \cmidrule(l){4-5}
$A_n$ & $\lfloor \frac{(n+1)^2}{4} \rfloor$ && $B_n/C_n$ & $\binom{n}{2}+1$  \\
$D_n$ & $\binom{n}{2}$  && $F_4$ & $6$  \\
$E_6$ & $16$ && $G_2$& $2$ \\
$E_7$ & $27$             &&  $I_2(m)$ & $2$ 
\\
$E_8$ & $36$ &&  $H_3$& $5$  \\
      &      && $H_4$& $8$ 
      \\
\bottomrule
\end{tabular}
}
\end{table}
\end{theorem}

For the proof in the crystallographic cases we follow Gorenstein--Lyons--Solomon \cite[Lemma~3.3.6]{GLS3}, who 
give case-by-case proofs of correctness for all types except $E_8$.
(Malcev didn't give a proof in the $E_8$ case either.)
For a case-independent proof of the correctness of Malcev's numbers, see Suter \cite{Suter2004}. 

The set $\{a_1,\dots,a_n\}$ denotes an orthonormal basis of a Euclidean space $V$ of dimension $n$. 
Let $\mathbf{S}^n$ be the set of all length $n$ vectors of signs, i.e., those $\epsilon = (\epsilon_1,\dots,\epsilon_n)$ such that $\epsilon_i = \pm 1$ for all $i$. 
For such an $\epsilon$, we write $a_{\epsilon} = \frac{1}{2}\sum_{i = 1}^n \epsilon_ia_i$. 
A symbol such as $a_{+--+}$ is shorthand for $a_{(1,-1,-1,1)}=\frac12[a_1 - a_2 - a_3 + a_4]$. 

We use the explicit realizations of the irreducible crystallographic root systems and orderings of simple roots from \cite[Table~1.8]{GLS3}, listed in Table~\ref{rsys explicit} but omitting $A_m$, $C_m$, and $G_2$ since they can be handled by other means.
(The choice of realization and ordering on simple roots in \cref{rsys explicit} differs from the Bourbaki ordering \cite[Plates~I--IX]{BourbakiLie4-6} for type E.)

\begin{table}
{
\renewcommand{\arraystretch}{1.2}
\setlength{\tabcolsep}{8pt}
\centering
\caption{Explicit Gorenstein--Lyons--Solomon realizations of irreducible crystallographic root systems and numberings of simple roots}\label{rsys explicit}
\begin{tabular}{ll}
$B_m$ & $\dynkin[labels={1,2,,m}] B{oo...oo}$\\
       & $\rsys = \{\pm a_i \pm a_j \mid 1 \leq i < j \leq m\} \cup \{\pm a_i \mid 1 \leq i \leq m\}$\\
       & $\base = \{a_1-a_2, a_2-a_3,\dots,a_{m-1}-a_m, a_m\}$\\
$D_m$
& $\dynkin[labels={1,2,,m-1,m}] D{oo...ooo}$\\
& $\rsys = \{\pm a_i \pm a_j \mid 1 \leq i < j \leq m\}$\\
& $\base = \{a_1-a_2, a_2-a_3,\dots,a_{m-1}-a_m, a_{m-1}+a_{m}\}$\\
$F_4$ 
& $\dynkin[label, */.style={solid,draw=black,fill=white}, ordering=Bourbaki] F4$\\
& $\rsys = \{\pm a_i \pm a_j \mid 1 \leq i < j \leq 4\} \cup \{\pm a_i \mid 1 \leq i \leq 4\} \cup \{a_{\epsilon} \mid \epsilon \in \mathbf{S}^4\}$\\
& $\base= \{a_2-a_3, a_3-a_4,a_4,a_{+---}\}$\\
$E_8$ 
& $\dynkin[labels={1,4,2,3,5,6,7,8}, */.style={solid,draw=black,fill=white}] E8$\\ 
& $\rsys = \{\pm a_i \pm a_j \mid 1 \leq i < j \leq 4\} \cup \{a_{\epsilon} \mid \epsilon \in \mathbf{S}^8, \epsilon_1\cdots \epsilon_8 = 1\}$\\ 
 & $\base = \{a_{+------+}, a_7-a_8,a_6-a_7,a_7+a_8, a_5-a_6, a_4-a_5, a_3-a_4, a_2-a_3\}$\\
$E_7$ 
& $\dynkin[labels={1,4,2,3,5,6,7}, */.style={solid,draw=black,fill=white}] E7$\\
& $\rsys = \{\gamma \in \rsys_{E_8} \mid \gamma \perp a_1 + a_2\}$ \\
& $\base = \base_{E_8}-\{a_2-a_3\}$\\
$E_6$ 
& $\dynkin[labels={1,4,2,3,5,6}, */.style={solid,draw=black,fill=white}] E6$\\
& $\rsys = \{\gamma \in \rsys_{E_7} \mid \gamma \perp a_2 - a_3\}$ \\
& $\base = \base_{E_7}-\{a_3-a_4\}$
\end{tabular}
}
\end{table}

\begin{proof}[Proof of \cref{hrank rsys}]
This is case by case. 
If $\rsys$ has rank $1$, then $h(\pos) = 1$.
If it has rank 2, then $h(\pos) = 2$ by \cref{constricting ra}. 
Thus, we may assume that the rank is at least $3$ from now on. 

\smallskip
\noindent
\textit{Type $A$:}  In type $A$ only, $\cone_\ints = \cone_\reals$ by \cite[Theorem]{Pilkington2006}. 
So the maximal size of a really abelian set of positive roots in $A_n$ is the same as the maximal size of an abelian set, which is $\lfloor (n+1)^2/4 \rfloor$. 

\smallskip
\noindent
\textit{Type $D$:}
The corresponding positive system for $D_n$ is $\{a_i \pm a_j \mid 1 \leq i < j \leq n\}$. 
When $n = 4$, there is a unique abelian subset $\Gamma = \{a_1 \pm a_j \mid j = 2,3,4\}$ of positive roots of maximal size $6 = \binom{4}{2}$. 
This is the full set of roots of $D_4$ in which $a_1$ occurs with positive coefficent.
It follows from this property that $\Gamma$ is convex. 
If $\Lambda = \Gamma\setmin\alpha$, where $\alpha$ has nonzero coefficient on $a_j$ ($j = 2,3,4$), then $\alpha \notin \cone_\reals(\Lambda)$.
This shows $\Lambda$ is convex as well and hence $\Gamma$ is really abelian.

Assume now that $n > 4$. 
There are exactly two abelian sets of maximal size in $D_n$ by \cite[p.~112]{GLS3}, one of which is $\Gamma = \{a_i+a_j \mid 1 \leq i < j \leq n\}$.
This is the set of roots of $D_n$ in which each basis vector occurs with positive coefficient, and so it is convex. 
Any subset $\Lambda \subset \Gamma$ of cocardinality $1$ is convex for a similar reason as before (no $a_i+a_j$ is in the convex cone of the remaining elements of $\Gamma$).
Hence, $\Gamma$ realizes $h(D_n^+) = \binom{n}{2}$ when $n > 4$.  

\smallskip
\noindent
\textit{Type $E$:}
Consider the following abelian sets given in \cite[p.113--114]{GLS3}: 
\begin{align*}
\Gamma_6 &= \left\{\sum_{i = 1}^6 c_i\alpha_i \in E_6^+ \tallmid c_1 = 0 \text{ or } 1 \text{ and } c_6 = 1\right\} \subseteq E_6^+,\\
\Gamma_7 &= \left\{\sum_{i = 1}^7 c_i\alpha_i \in E_7^+ \tallmid c_1 = 1\right\} \subseteq E_7^+,\\
\Gamma_8 &= \left\{a_1\pm a_i \mid i = 2,\dots,8\right\} \cup \left\{a_{+\epsilon_2\cdots\epsilon_8} \tallmid \sum_{i = 2}^8 \epsilon_i = 5 \text{ or } 7\right\} \subseteq E_8^+,
\end{align*}
of sizes $16$, $27$, and $36$, respectively. 
A computation using Magma \cite{Magma}, using the built-in Toric Geometry package to compute convex cones, shows that each of these is really abelian. 

\smallskip
\noindent
\textit{Type $B_n/C_n$}: 
Duality $\alpha \mapsto 2\alpha/(\alpha,\alpha)$ is an isomorphism between the closure spaces $B_n$ and $C_n$ with convex closure, so it suffices to work with $B_n$, where $B_n^+ = \{a_i \pm a_j\} \cup \{a_i\}$. 
Note that the set of long roots in $B_n$ forms a system of type $D_n$. 
There is always an abelian set of size $1 + \binom{n}{2}$, namely
\[
\Gamma = \{a_1\} \cup \{a_i+a_j \mid 1 \leq i < j \leq n\},
\]
which is of maximal size when $n \geq 4$. 
As in the $D_n$ case, the set of long roots in $\Gamma$ is really abelian, and further, its convex cone does not contain $a_1$. 
Similarly, as in the $D_n$ case, $a_i + a_j$ is not in the convex cone of the other long roots of $B_n$, and it doesn't help to throw in $a_1$. 
That is, if $\Lambda$ is of cocardinality $1$ in $\Gamma$ and contains $a_1$, then it is again convex. 
Thus, $\Gamma$ is really abelian.
This completes the $B_n$ case when $n \geq 4$. 

When $n = 3$, the maximal size of an abelian subset of $B_3$ is $5$, not $1+\binom{3}{2} = 4$. 
However, we claim that there is a unique abelian set $\Gamma$ of positive roots of size $5$ in $B_3$, and it is not really abelian. 
Indeed, $\Gamma$ can contain at most one short root $a_i$, and the set of long roots in $\Gamma$ is an abelian set in $D_3 = A_3$ of size $4$. 
This is unique: $\Gamma \cap D_3 = \{a_1 \pm a_2, a_1 \pm a_3\}$.
Abelian then implies $i = 1$, and hence $\Gamma = \{a_1, a_1 \pm a_2, a_1 \pm a_3\}$ is the unique abelian set of size $5$. 
But $\Gamma$ is not really abelian as $a_1 = \frac{1}{2}(a_1+a_2) + \frac{1}{2}(a_1-a_2)$. 
Hence, $h(B_3^+) = 4 = 1+\binom{3}{2}$. 

\smallskip
\noindent
\textit{Type $F_4$:}
Consider now $F_4$, whose short roots comprise the last two sets of the union describing $\rsys$ in Table~\ref{rsys explicit}, and where the corresponding positive system is
\[
F_4^+ = \{a_i \pm a_j \mid 1 \leq i < j \leq 4\} \cup \{a_i \mid 1\leq i \leq 4\} \cup \{a_{+\epsilon_2\epsilon_3\epsilon_4} \mid \epsilon_i = \pm\}. 
\]
A computation on Magma (or by hand) shows that the subset 
\[
\{a_1, a_1 + a_2, a_1+a_3, a_1+a_4, a_{+++-}, a_{++++}\}
\]
of $F_4^+$ of size $6$ is really abelian. 
Conversely, suppose that $\Gamma \subseteq F_4^+$ is really abelian.
As argued in \cite[p.112--113]{GLS3}, there can be at most one short root in $\Gamma$ of the form $a_i$, since the sum of two distinct such is a root. 
There can be at most one short root in $\Gamma$ of the form $a_{+\epsilon_2\epsilon_3\epsilon_4}$ with one minus sign, and if there is one, then $a_{+---} \notin \Gamma$. 
Likewise, there can be at most one short root in $\Gamma$ of the form $a_{+\epsilon_2\epsilon_3\epsilon_4}$ with two minus signs, and if there is one, then $a_{++++} \notin \Gamma$. 
Altogether, there can be at most three short roots in $\Gamma$.
For each $1 \leq i < j \leq 4$, $a_i$ is in the convex cone of $\{a_i\pm a_j\}$, so we can have at most one of $a_i+a_j$ or $a_i-a_j$ in $\Gamma$. 
But if $\{i,j,k,l\} = \{1,2,3,4\}$, then the presence of $a_i+a_j$ or $a_i-a_j$ in $\Gamma$ excludes both $a_k \pm a_l$, since the half sum of $a_i\pm a_j$ and $a_k\pm a_l$ is a root of the form $a_{+\epsilon_2\epsilon_3\epsilon_4}$. 
These two restrictions show there can be at most $\binom{4}{2}/3 = 3$ long roots as well, and hence $|\Gamma| \leq 6$ as desired. 
This completes the proof in the crystallographic cases. 

\smallskip
\noindent
\textit{Type $H$}:
The numbers $h_{\reals}(H_3^+) = 5$ and $h_{\reals}(H_4^+) = 8$ were computed with Python code written explicitly for this purpose. 
The number of positive roots in $H_4$ is 60, which is too large for a brute force search of really abelian subsets $\Gamma$. 
To cut down on the search space we first use \cref{ra simple root} to assume $\Gamma$ contains a simple root $\alpha$. 
A rank $2$ subsystem of ($H_3$ or) $H_4$ is of type $A_1 \times A_1$, $A_2$, or $H_2$ by Theorem 4.1 of \cite{ChenMoodyPatera:NCRS}. 
\Cref{constricting ra} then implies that $(\alpha,\beta) \in \{2,\varphi,1,0\}$ for each $\beta \in \Gamma$, where $\varphi = \frac{1+\sqrt{5}}{2}$ is the golden ratio, assuming the roots have norm $2$. 
As there are $6$, $10$, and $15$ positive roots at inner products $\varphi$, $1$, and $0$ with the simple root $\alpha$ (half the number of vertices of a icosahedron, dodecahedron, and icosidodecahedron, respectively) this places $\Gamma$ in a subset of $\pos$ of size $1+6+10+15 = 32$. 
\end{proof}

\begin{remark}
There are exactly $3$ really abelian subsets of $H_3$ of size $5$ corresponding to the three pentagons one sees in a view of an icosidodecahedron. 
On the other hand, there are hundreds of really abelian subsets of positive roots in $H_4$ of maximal size $8$, including 236 that contain one of the four simple roots. 
\end{remark}

\appendix

\section{Proof of the third cube lemma}\label{cube appendix}
We prove \cref{generalized pasting law}.
Fix $n \geq 0$.
For concision write $\ncube \coloneq [1]^n$ for the generic $n$-cube, and $0\in \ncube$ for its minimal element.
Subsets of posets are taken to be full subcategories.
Let $X \colon [2] \times \ncube \to \mathcal{C}$ be a stacked pair of $(n{+}1)$-cubes, and set
\[
P \coloneq X|_{\{0,1\} \times \ncube}, \quad
Q \coloneq X|_{\{1,2\} \times \ncube}, \quad
R \coloneq X|_{\{0,2\} \times \ncube}.
\]
Assuming $Q$ is cartesian, we want to show that $P$ is cartesian if and only if $R$ is cartesian.
To make our diagrams more manageable, write $00 = (0,0)$ and $10 = (1,0)$ for the elements of $[0,2] \times \ncube$.
We reformulate the cartesian properties for $R, P$ using the following squares.
\[
\begin{tikzcd}[column sep=small]
(\{0,2\} \times \ncube) \setmin 00 \rar \dar & 
([2] \times \ncube) \setmin \{00,10\} \dar{i} &[-0.3cm]
(\{0,1\} \times \ncube) \setmin 00 \rar \dar & 
([2] \times \ncube) \setmin 00 \dar{j}
\\
{\{0,2\} \times \ncube} \rar & {[2] \times \ncube} &
{\{0,1\} \times \ncube} \rar & {[2] \times \ncube} 
\end{tikzcd} 
\]
The horizontal maps are left adjoints, hence initial.
Thus we have $R$ cartesian if and only if $X \to i_*i^*X$ is an equivalence, and $P$ is cartesian if and only if $X \to j_*j^*X$ is an equivalence.
Consider the following diagram of inclusions:
\[ \begin{tikzcd}[column sep=0]
(\{1,2\} \times \ncube) \setmin \{10\} \dar \ar[rr]{k'} && \{1,2\} \times \ncube \dar \\
([2] \times \ncube) \setmin \{00,10\} \ar[rr,"k"]\drar["i"']  && ([2] \times \ncube) \setmin 00 \dlar{j}\\
& {[2] \times \ncube}
\end{tikzcd} \]
Examination of the top square implies that $j^*X = X|_{([2] \times \ncube) \setmin 00}$ is right Kan extended along $k$ from $X|_{([2] \times \ncube) \setmin \{00,10\}} = i^*X$. 
Since all maps are fully faithful, we only need to check that this is right Kan extended at $10$, which is true using the top arrow since $Q$ is cartesian.
We then have
\[
  X \to j_*j^* X \xrightarrow\simeq j_*k_*i^*X = i_*i^*X
\]
and the first map is an equivalence if and only if the composite is so. \qed

\section{The degree of the symmetric spheres}
\label{app sphere}

In this appendix we consider, for $n\geq 1$, the concrete model $\rep{n}/\partial\rep{n}$ for the symmetric sphere, and prove that its degree is $2n$, as asserted in \cref{ex sym sphere}.
We first show that the symmetric $n$-sphere is lower $(4n{-}1)$-Segal, which relies on a simple combinatorial analysis of certain maps of finite sets.
The uniqueness and existence parts of this $d$-Segal condition are separated into the first two subsections.
The model $\Delta^n/\partial\Delta^n$ for the simplicial sphere is also lower $(4n{-}1)$-Segal (\cref{cor simplicial sphere}). 
We give an explicit example to show that $\rep{n}/\partial\rep{n}$ is not lower $(4n{-}3)$-Segal in \cref{lem sphere lower bound}.

Recall that \cref{rmk non-gapped} provides an alternative description of the $(2k{-}1)$-Segal condition for a symmetric set $X$.
By passing from the category $\fin$ to the equivalent category of nonempty finite sets, the condition becomes that for each finite set $S$ and each proper subset $I\subset S$ of cardinality $k+1$, if we are provided simplices $x_i \in X_{S_i}$ for $i\in I$ (where $S_i = S\setmin i$) such that the restrictions of $x_i$ and $x_j$ to $X_{S_i\cap S_j}$ are equal for all $i,j\in I$, then there exists a unique $x\in X_S$ whose restriction to $X_{S_i}$ is equal to $x_i$.
We aim to establish this in the case when $X = \rep{n}/\partial\rep{n}$ and $k=2n$.
The next two subsections cover the uniqueness and existence parts of the statement; the existence argument is more delicate and works only for $n\geq 2$.
A concrete extension of this $X$ to the category of all nonempty finite sets has $X_S$ the quotient of $\hom(S,[n])$ which identifies the nonsurjective functions.

\subsection{Uniqueness}
Below we consider functions between nonempty finite sets.
A function $f\colon S \to R$ is \emph{trivial} if it is not surjective, and otherwise we say $f$ is \emph{epi}. 
Write $[f]$ for the equivalence class under the relation on $\hom(S,R)$ identifying all of the trivial functions to a single element $\ast$.
We'll use superscript notation to indicate restriction to a complement of an element, so that
\[
	f^i \colon S_i \coloneq S \setmin i \to R
\]
is the restriction of $f$, and similarly $f^{ij} = (f^i)^j = (f^j)^i \colon S_{ij} \coloneq S \setmin \{i,j\} \to R$.

\begin{lemma}\label{lem trivial}
Let $f\colon S \to R$ and $T \subsetneq S$ be a proper subset of $S$ containing at least $|R|$ elements.
If $f^t$ is trivial for all $t\in T$, then $f$ is trivial.
\end{lemma}
\begin{proof}
If $f$ is epi, then $\image(f^t)$ is a proper subset of $R = \image(f) = \image(f^t) \cup \{f(t) \}$.
Thus $f^{-1}f(t) = \{t\}$ for all $t\in T$, so $f$ has nowhere to send elements of $S \setmin T \neq \varnothing$.
\end{proof}

\begin{lemma}\label{lem two agree}
Let $f, g\colon S \to R$ be functions, and $i\neq j$ in $S$.
If $f^i = g^i$ and $f^j = g^j$, then $f=g$.
\end{lemma}
\begin{proof}
It is immediate that $f$ and $g$ agree outside of the two-element subset $\{i,j\}$.
But also $f(j) = f^i(j) = g^i(j) = g(j)$ and $f(i) = f^j(i) = g^j(i) = g(i)$.
\end{proof}

\begin{lemma}\label{lem three agree triv}
Let $f, g\colon S \to R$ with $|S| > |R| = r$. 
If $[f^i] = [g^i]$ for $i$ ranging over an $r$-element subset $I$ of $S$, then $[f]=[g]$. (If not all $[f^i] = \ast$, then $f=g$.)
\end{lemma}
\begin{proof}
For $i\in I$, since  $[f^i] = [g^i]$ we either have $f^i$ and $g^i$ are epi and equal, or are both trivial.
We assume $f^i$ is epi for exactly one $i\in I$ -- the case when there are no such $i$ has been handled in \cref{lem trivial} and the case with two or more such $i$ in \cref{lem two agree}.
Since $f^i$ and $g^i$ are epi, so are $f$ and $g$; we know $f(t) = g(t)$ for $t\neq i$. 
For each $j\in I \setmin i = I_i$ we have 
\[
	f^{-1}f(j) = \{j\} = g^{-1}g(j)
\]
by triviality of $f^j$ and $g^j$.
In particular $f(i)$ and $g(i)$ are not in the $(r{-}1)$-element set $f(I_i) \subset R$, hence $f(i) = g(i)$.
\end{proof}

\subsection{Existence}
Let $R$ be a set of cardinality $r \geq 3$.
Fix a set $S$ with at least $2r$ elements and a $(2r{-}1)$-element subset $I\subsetneq S$.
Assume we are given functions
\[
	f_i \colon S_i \coloneq S\setmin i \to R
\]
for $i\in I$ such that $[f_i^j] = [f_j^i]$ for all $i\neq j$ in $I$.

We form a colored graph from this data.
The vertices are the elements of $I$, and there is an edge between $i$ and $j$ if and only if $f_i^j = f_j^i$.
We color this graph as follows:
\begin{enumerate}[label=\arabic*.]
\item $i \in I$ is a \emph{black vertex} when $f_i$ is epi.
\item $i \in I$ is a \emph{red vertex} when $f_i$ is trivial.
\item An edge between $i,j$ is a \emph{black edge} if $f_i^j = f_j^i$ is epi.
\item An edge between $i,j$ is a \emph{red edge} if $f_i^j = f_j^i$ is trivial.
\end{enumerate}
Notice that black edges always are between black vertices, but red edges could join vertices of any color.

\begin{lemma}\label{lem two black}
Any black vertex is incident to at least $r-1$ black edges.
\end{lemma}
\begin{proof}
Suppose $i$ is a black vertex. 
By \cref{lem trivial}, since $f_i$ is epi, we can have at most $r-1$ elements $j\in I_i$ with $f_i^j$ trivial.
But $I_i$ has cardinality $2r-2$.
\end{proof}

\begin{lemma}\label{lem edge exists}
If $ij$ and $ik$ are black edges, then $jk$ is an edge (either black or red).
\end{lemma}
\begin{proof}
The assumption is that $f_{ij} \coloneq f_i^j = f_j^i$ and $f_{ik} \coloneq f_i^k = f_k^i$ are epi; we write $f_{ijk} = f_{ij}^k = f_{ik}^j$.
If $f_k^j$ is epi, then by definition there is a black edge $jk$, so we assume $f_k^j$ and $f_j^k$ are both trivial.
Let $T = S \setmin \{i,j,k\}$; our assumption implies $f_j(t) = f_k(t)$ for $t\in T$.
Our aim is to show $f_j(i) = f_k(i)$.
Notice that 
\[
f_{ij}(T \cup \{k\}) = \image(f_{ij}) = R = \image(f_{ik}) = f_{ik}(T\cup \{j\}),
\]
while $f_{ijk}(T) \subseteq \image (f_k^j) \neq R$.
Thus $f_{ijk}(T)$ is an $r-1$ element subset of $R$. 

Write $g \colon T \twoheadrightarrow f_{ijk}(T)$ for the codomain restriction of $f_{ijk}$.
Since $T$ contains at least $2r-3$ elements (since $S$ contained at least $2r$), there must be $\ell$ in the $2r-4$ element set $T\cap I = I\setmin \{i,j,k\}$ such that $g^{-1}g(\ell)$ has cardinality strictly greater than one. (This is an application of \cref{lem trivial} to $g$ -- we are using $|T\cap I| = 2r-4 \geq r-1 = |f_{ijk}(T)|$, which is true since $r \geq 3$.)
But 
\[
	g^{-1}g(\ell) \subseteq f_p^{-1} f_p(\ell) \quad p\in \{i,j,k\}
\]
so $f_i^\ell$ and $f_j^\ell$ and $f_k^\ell$ are all epi.
Then 
\[
	f_j(i) = f_j^\ell(i) = f_\ell^j(i) = f_\ell^k(i) = f_k^\ell(i) = f_k(i)
\]
and we conclude $f_j^k = f_k^j$.
\end{proof}

If all vertices of the graph are red, then for any trivial $f\colon S \to R$ we have $[f^i] = [f_i]$ for all $i\in I$.
We consider the case where there is at least one black vertex, hence at least $r-1$ black edges. 
Assume the graph has a black edge between $i_0$ and $i_1$.
Define $f \colon S \to R$ by 
\[
	f(t) = \begin{cases}
		f_{i_0}(t) & i_0 \neq t \\
		f_{i_1}(t) & i_1 \neq t.
	\end{cases}
\]

\begin{proposition}\label{blck up low}
For each black vertex $i$, we have $f^i = f_i$.
\end{proposition}
\begin{proof}
This is by definition if $i$ is $i_0$ or $i_1$.
Let $J \subseteq I$ be the set consisting of $i_0$ and all of the vertices connected to $i_0$ via a black edge.
By \cref{lem two black} we have $|J| \geq r$.
Let $j\in J \setmin \{i_0,i_1\}$, so that there is a black edge between $j$ and $i_0$. As \cref{lem edge exists} guarantees an edge between $j$ and $i_1$, 
we have the first equality in the second line below.
\[
\begin{gathered}
f_j^{i_0} = f_{i_0}^j = (f^{i_0})^j = (f^j)^{i_0} \\
f_j^{i_1} = f_{i_1}^j = (f^{i_1})^j = (f^j)^{i_1} 
\end{gathered}
\]
By \cref{lem two agree}, $f_j = f^j$.

Let $k$ be a black vertex in $I \setmin J$.
Since $k$ is incident to at least $r-1$ black edges and $|I \setmin J| = 2r-1 - |J| \leq (2r-1) - r = r-1$, there must be a black edge between $k$ and some $j \in J$.
Then there is a red or black edge between $k$ and $i_0$, and we can  reapply the same argument above to see $f^k=f_k$.
\end{proof}

In particular, the definition of $f$ is independent of the choice of $i_0$ and $i_1$ joined by a black edge.
Still assuming that the graph has a black vertex, we have:

\begin{lemma}\label{red triv}
If $i \in I$ is a red vertex, then $f^i$ is trivial.
\end{lemma}
\begin{proof}
Let $B \subseteq I$ be a set consisting of $r$ black vertices (guaranteed by \cref{lem two black}).
For each $j\in B$ we have
\[
	(f^i)^j = (f^j)^i = f_j^i,
\]
so $[(f^i)^j] = [f_j^i] = [f_i^j]$ for all $j\in B$.
By \cref{lem three agree triv}, $[f^i] = [f_i]$.
\end{proof}

\begin{theorem}\label{thm existence}
Let $R$ be a set of cardinality $r \geq 3$ and $I \subsetneq S$ a proper subset of cardinality $|I| = 2r-1$.
Assume we are given functions $f_i \colon S_i  \to R$ for $i\in I$ such that $[f_i^j] = [f_j^i]$ for all $i\neq j$ in $I$.
Then there exists a function $f\colon S \to R$ such that $[f^i] = [f_i]$ for all $i\in I$.
Moreover, $[f]$ is unique.
\end{theorem}
\begin{proof}
If all vertices are red, then any trivial $f\colon S \to R$ will do.
If there is at least one black vertex, then we constructed a specific epi $f \colon S \to R$ which has the property $[f^i]=[f_i]$ by \cref{blck up low} and \cref{red triv}.
Uniqueness of $[f]$ follows from \cref{lem three agree triv}.
\end{proof}

\subsection{Degree of the symmetric spheres}

In the symmetric sphere $\rep{n} / \partial \rep{n}$, for each $k\geq 0$, the set of $k$-simplices of $\rep{n} / \partial\rep{n}$ consists of a unique totally degenerate element $\ast_k$, along with the surjective functions $[k] \twoheadrightarrow [n]$.

\begin{lemma}\label{lem sphere lower bound}
If $n\geq 1$, then $\deg(\rep{n}/\partial \rep{n}) > 2n-1$.
\end{lemma}
\begin{proof}
We show that $X= \rep{n}/\partial \rep{n}$ is not lower $d$-Segal for $d= 2(2n-1)-1 = 4n-3$.
The proof is valid for $n \geq 2$; the $n=1$ case holds since $\rep{1}/\partial \rep{1}$ is the free partial group on one generator, which is not a group.
Consider the following elements $a,b \in X_{2n}$ and $a',b' \in X_{2n-1}$:
\begin{align*}
a &= \underbrace{0\cdots0}_n12 \cdots n0 & a' &= \underbrace{0\cdots0}_{n-1}12 \cdots n0 \\
b &= 12 \cdots n\underbrace{0\cdots0}_{n+1} & b' &= 12 \cdots n\underbrace{0\cdots0}_n.
\end{align*}
Since $n\geq 1$, $a\neq b$. 
The faces (excluding the top one) of $a$ and $b$ are 
\[
	d_i a =\begin{cases}
		a' & 0 \leq i \leq n-1 \\
		\ast & n \leq i \leq 2n-1
	\end{cases}
	\qquad
	d_i b =\begin{cases}
		\ast & 0 \leq i \leq n-1 \\
		b' & n \leq i \leq 2n-1.
	\end{cases}
\]
Consider the set $I = \{0, \dots, 2n-1\} \subsetneq [2n]$, and let $x_0 = x_1 = \dots = x_{n-1} = a'$ and $x_n = x_{n+1} = \dots = x_{2n-1} = b'$. 
Then for $0 \leq i < j \leq 2n-1$ we have $d_i x_j = d_{j-1} x_i$; this comes down to the simplicial identity $d_i d_j = d_{j-1} d_i$ when $0\leq i < j \leq n-1$ or $n \leq i < j \leq 2n-1$, while for $i\in [0,n-1]$ and $j\in [n,2n-1]$ our definition gives $d_i x_j = \ast = d_{j-1} x_i$.
There is no $x \in X_{2n}$ having $d_i x = x_i$ for $i=0, \dots, 2n-1$.
If there were, by \cref{lem two agree} we would have $a = x = b$.
\end{proof}
Notice this proof breaks down if there is an overlap between $\{0,1\}$ and $\{n,n+1\}$.

\begin{theorem}\label{thm degree sphere}
For $n\geq 1$, the degree of $\rep{n}/\partial \rep{n}$ is $2n$.
\end{theorem}
\begin{proof}
Let $X = \rep{n}/\partial \rep{n}$.
By \cref{lem sphere lower bound}, we need $\deg(X) \leq 2n$.
If $n=1$, then $X$ is a 1-dimensional partial group, so \cref{deg dim} implies $\deg(X) \leq \dim(X) + 1 = 2$. 
When $n\geq 2$, taking $R = [n]$ in \cref{thm existence} gives that $X$ is lower $(2(2n){-}1)$-Segal, hence has degree at most $2n$.
\end{proof}

We conclude with this deduction about a simplicial model of the $n$-sphere.

\begin{corollary}\label{cor simplicial sphere}
For $n \geq 1$, the simplicial set $X = \Delta^n / \partial \Delta^n$ is lower $(4n{-}1)$-Segal.
\end{corollary}
\begin{proof}
Let $I \subset [m]$ be a gapped subset of size $2n+1$, and suppose we have elements $x_i \in X_{m-1}$ satisfying $d_i x_j = d_{j-1} x_i$ for $i < j$ in $I$.
Since $X$ is a simplicial subset of $\rep{n}/\partial \rep{n}$, by \cref{thm degree sphere} there is a unique element $x \in (\rep{n}/\partial \rep{n})_m$ such that $d_i x = x_i$ for all $i\in I$.
The only question is whether or not $x$ is an element of $X_m$.
Suppose not; then in particular $x \neq \ast$.
Write $f\colon [m] \twoheadrightarrow [n]$ for $x$ considered as function, and let $\ell$ be the least integer such that $f(\ell) > f(\ell+1)$ (if there is no such $\ell$, then $f$ is order-preserving, contrary to assumption.)

As $f^i$ is trivial for at most $n$ elements $i\in I$ by \cref{lem trivial}, there are at least $n+1$ elements of $I$ with $f^i$ epi.
Since $I$ is gapped, we may choose $i$ with $f^i$ epi and $i \neq \ell, \ell+1$. 
But then $f(\ell) = f^i(\ell) \leq f^i(\ell+1) = f(\ell+1)$, contrary to assumption.
We conclude that $x\in X_m$ after all, and that $X$ is indeed lower $(4n{-}1)$-Segal.
\end{proof}

\section{Actions}
\label{app action}

We unravel the notion of action from \cref{def partial action}.

\begin{definition}\label{def L-sets}
Let $L$ be a partial groupoid and $S$ a set.
We say that $S$ is an \emph{$L$-set} if it is equipped with 
partially-defined functions $L_n \times S \nrightarrow S$ for $n\geq 0$ subject to the following conditions:

\begin{enumerate}[label = A\arabic*), ref = A\arabic*]
\item Given $x\in S$, there exists a unique vertex $a\in L_0$ such that $a\cdot x$ is defined. Moreover, $a\cdot x = x = [\id_a] \cdot y$.
\label{action identity}
\end{enumerate}
For the remaining conditions, assume 
$[f_1 | \dots | f_n] \cdot x$ is defined, where $n\geq 1$ and $f_i \colon a_{i-1} \to a_i$.
\begin{enumerate}[resume*]
\item For each $0\leq i \leq n$ we have $[f_1 | \dots | f_i | \id_{a_i} | f_{i+1} | \dots | f_n] \cdot x = [f_1 | \dots | f_n] \cdot x.$
\label{action degen}
\item For each $1 \leq i \leq n-1$ we have $[f_1 | \dots | f_{i+1} \circ f_i  | \dots | f_n] \cdot x = [f_1 | \dots | f_n] \cdot x.$
\label{action inner face}
\item If $n\geq 2$, then $[f_1 | \dots | f_{n-1} ] \cdot x$ is defined.
\label{action last face}
\item If $n\geq 2$, then $[f_2 | \dots | f_n] \cdot ( [f_1] \cdot x) = [f_1 | \dots | f_n] \cdot x.$
\label{action splitting}
\item $[f_1 | \dots | f_n | f_n^{-1} | \dots | f_1^{-1}] \cdot x = x$.
\label{action inversion}
\end{enumerate}
A map of $L$-sets is a function $\phi \colon S \to S'$ such that if $f\cdot x$, then $\phi(f \cdot x) = f\cdot \phi(x)$.
\end{definition}

Condition \ref{action identity} provides a function $S \to L_0$, and this action of $L$ on $S$ is really an action of $L$ on $\{S_a\}$ where $S_a$ is the preimage of $a\in L_0$ and $S \cong \coprod_{a\in L_0} S_a$.

\begin{lemma}\label{lem ordb action}
Suppose $f\in L_n$ and $f \cdot x$ is defined.
If $\alpha \colon [m] \to [n]$ in $\ord$ satisfies $\alpha(0) = 0$, then $(\alpha^*f) \cdot x$ is defined.
If, additionally, $\alpha(m) = n$, then $(\alpha^*f) \cdot x = f \cdot x$.
\end{lemma}
\begin{proof}
This follows by decomposing $\alpha^*$ into face and degeneracy maps, without using $d_\bot$ (and without using $d_\top$ in the second instance).
Conditions \ref{action identity} and \ref{action degen} cover  degeneracies, \ref{action inner face} covers inner faces, and \ref{action last face} covers the top face $d_\top$.
\end{proof}

\begin{lemma}\label{lem arb action}
Suppose $f\in L_n$ and $f \cdot x$ is defined.
Let $\gamma \colon [m] \to [n]$ be an arbitrary map in $\ord$.
Define $\bar \gamma \colon [m+1] \to [n]$ by $\bar\gamma(0) = 0$ and $\bar \gamma(i) = i-1$ for $i >0$.
Then 
\[
  (\gamma^* f) \cdot ([f_{0,\gamma(0)}] \cdot x) = (\bar \gamma^*f) \cdot x. 
\]
\end{lemma}
\begin{proof}
Set $g = \bar \gamma^*f$, and notice that $g\cdot x$ is defined by \cref{lem ordb action}. 
We have $g_1 = (\bar \gamma^*f)_{01} = f_{\bar\gamma(0),\bar\gamma(1)} = f_{0,\gamma(0)}$
and $[ g_2 | \dots | g_{m+1}] = d_0 (\bar \gamma^* f) = (\bar \gamma \delta^0)^* f = \gamma^* f$, so the equality holds by \ref{action splitting}.
\end{proof}

For each $n\geq 0$, let $E_n \subseteq L_n \times Y$ be the set of pairs $(f,y)$ such that $f\cdot y$ is defined.
For each $\gamma \colon [m] \to [n]$ in $\ord$, define a function $\gamma^* \colon E_n \to E_m$ by the rule  
\begin{equation}\label{eq gamma star}
  \gamma^*(f,y) \coloneq (\gamma^* f, f_{0,\gamma(0)} \cdot y);
\end{equation}
this element is in $E_m$ by \cref{lem arb action}.
Given a composable pair $\beta \colon [p] \to [m]$, $\gamma\colon [m] \to [n]$ in $\ord$, a quick check shows that $(\gamma \beta)^* = \beta^*\gamma^*$, hence the $E_n$ assemble into a simplicial set $E$.
Notice that the set of vertices $E_0 = \{(a, x) \mid a \cdot x\}$ is in bijection with $S$ by \ref{action identity}. 

Projection onto the first coordinate is a simplicial map $\pi \colon E \to L$.
This map is star injective since the source vertex of $(f,x)$ is (up to the isomorphism just mentioned) just $x$. 
This implies that $E$ is edgy.
If $[f_1 | \dots | f_n] \cdot x = y$ is defined, then the axioms imply that 
\[
  [f_1 | \dots | f_n | f_n^{-1} | \dots | f_1^{-1} | f_1 | \dots | f_n] \cdot x = [f_n^{-1} | \dots | f_1^{-1} | f_1 | \dots | f_n ] \cdot y
\]
is defined.
Thus every simplex of $E$ is germinable in the sense of \cite[Definition 2.4]{HackneyLynd:PGSSS}, for the anti-involution given by 
\[
  ([f_1 | \dots | f_n], x) \mapsto ([f_n^{-1} | \dots | f_1^{-1}], y).
\]
It follows from \cite[Theorem 4.1]{HackneyLynd:PGSSS} that $E$ canonically has the structure of a symmetric set.

\begin{theorem}
The preceding construction gives an equivalence between the category of $L$-sets of \cref{def L-sets} and the category of partial actions of $L$ from \cref{def partial action}.
\end{theorem}

The functor in the reverse direction is more or less evident: given a star injective map $E \to L$, the partial functions $L_n \times E_0 \pto E_0$ described after \cref{def star inj maps} endow $E_0$ with the structure of an $L$-set. 

\begin{remark}
Hayashi independently gave a version of \cref{def L-sets} for right actions, as well as a version of the preceding construction \cite[\S3,4]{Hayashi:PGSF}.
\end{remark}

\bibliographystyle{amsalpha}
\bibliography{pgds}
\end{document}